\documentclass[11pt, twoside]{article}
\usepackage{amsfonts,amssymb,amsmath,amsthm}
\usepackage{authblk}
\usepackage{graphicx}
\usepackage[all]{xy}
\usepackage{tikz}
\usepackage{changepage}
\usepackage{psfrag,xmpmulti,amscd,color,pstricks, import}

 \divide\oddsidemargin by 2
\newtheorem{thm}{Theorem}[section]
\newtheorem{cor}[thm]{Corollary}
\newtheorem{lem}[thm]{Lemma}
\newtheorem{prop}[thm]{Proposition}

\newtheorem*{thmA}{Theorem A}
\newtheorem*{thmB}{Theorem B}
\newtheorem*{thmD}{Theorem D}
\newtheorem*{thmC}{Theorem C}
\newtheorem*{thmE}{Theorem E}

\theoremstyle{definition}
\newtheorem{defn}[thm]{Definition}

\newtheorem*{rem*}{Remark}
\newtheorem{ex}{Example}

\numberwithin{equation}{section}

\definecolor{OrangeRed}{cmyk}{0,0.6,1,0}            % half magenta only, full yellow
\definecolor{DarkBlue}{cmyk}{1,1,0,0.20}
\definecolor{DarkGreen}{cmyk}{1,0,0.6,0.2}
\definecolor{myblue}{rgb}{0.66,0.78,1.00}
\definecolor{Violet}{cmyk}{0.79,0.88,0,0}
\definecolor{Lavender}{cmyk}{0,0.48,0,0}

\renewcommand{\Im}{\operatorname{Im}}
\renewcommand{\Re}{\operatorname{Re}}

\newcommand{\dist}{\operatorname{dist}}
\newcommand{\length}{\operatorname{length}}

\newcommand{\inter}{\operatorname{int}}
\newcommand{\ext}{\operatorname{ext}}

\newcommand{\FF}{{\cal F}}

\newcommand{\A}{{\mathbb A}}
\newcommand{\C}{{\mathbb C}}

\newcommand{\D}{{\mathbb D}}

\newcommand{\N}{{\mathbb N}}

\newcommand{\Z}{{\mathbb Z}}

\newcommand{\ra}{\rightarrow}
\newcommand{\lran}{\underset{n\to\infty}{\longrightarrow}}

\newcommand{\ov}{\overline}

\renewcommand{\epsilon}{\varepsilon}
\renewcommand{\phi}{\varphi}
\newcommand{\dD}{\dist_{\D}}
\renewcommand{\tilde}{\widetilde}

\newcommand{\Gt}{\tilde{\Gamma}}
\newcommand{\leucl}{{\ell_{\operatorname{Eucl}}}}
\newcommand{\deucl}{{\dist_{\operatorname{Eucl}}}}

\vspace{5cm}

\title{Classifying simply connected wandering domains}

\author[1]{Anna Miriam Benini\thanks{ This project has received funding from the European Union’s Horizon 2020 research and innovation programme under the Marie Sk{\l}odowska-Curie Grant Agreement No. 703269 COTRADY.}}
\author[2]{Vasiliki Evdoridou \thanks{ Supported by the EPSRC grant EP/R010560/1.}}
\author[3]{N\'uria Fagella\thanks{Partially supported by the Spanish grant MTM2017-86795-C3-3-P, the Maria de Maeztu Excellence Grant MDM-2014-0445 of the BGSMath and the Catalan grant 2017SGR1374.}}
\author[2]{Philip J.  Rippon\textsuperscript{ \textdagger }}
\author[2]{Gwyneth M.  Stallard\textsuperscript{ \textdagger }}
\affil[1]{\small Dep. of Matematical, Physical and Computer Sciences, Universit\`a di Parma, Italy.}
\affil[2]{\small School of Mathematics and Statistics, The Open University, Milton Keynes, UK.}
\affil[3]{\small Dep. de Matem\`atiques i Inform\`atica, Universitat de Barcelona, Catalonia.}

\date{}

\begin{document}

\maketitle

\begin{center}
\emph{Dedicated to Misha Lyubich on his 60th birthday}
\end{center}

\begin{abstract}  While the dynamics of transcendental entire functions in periodic Fatou components and in multiply connected wandering domains are well understood, the dynamics  in simply connected wandering domains have so far eluded classification. We give a detailed classification of the dynamics in such wandering domains in terms of the hyperbolic distances between iterates and also in terms of the behaviour of orbits in relation to the boundaries of the wandering domains.  In establishing these classifications, we obtain new results of wider interest concerning non-autonomous forward dynamical systems of holomorphic self maps of the unit disk.    We also develop  a new general technique for constructing examples of bounded, simply connected wandering domains with prescribed internal dynamics, and a criterion to ensure that the resulting boundaries are Jordan curves. Using this technique, based on approximation theory, we show that  all of the  nine possible types of simply connected wandering domain resulting from our classifications are indeed realizable.
\end{abstract}

\section{Introduction}
We consider dynamical systems defined by the iteration of  holomorphic maps
\[
f:\C\to \C
\]
on the complex plane, and particularly transcendental ones,   that is, those with an essential singularity at infinity.  The complex plane, seen as the phase space of the system, splits into two completely invariant subsets: the {\em Fatou set}, or those points in a neighbourhood of which the iterates $\{f^n\}$ form a normal family, and its complement, the {\em Julia set}. The Fatou set is open and consists typically of infinitely many connected components called {\em Fatou components}. Fatou components map from one to another and this leads to dynamics on the set of these components.

In this setting, periodic Fatou components were completely classified a century ago by Fatou, in terms of the possible limit functions of the family of iterates;  see, for example,~\cite{bergweiler93}. Indeed, if~$U$ is a periodic Fatou component of period $p\geq 1$, then $U$ can only be one of the following: a domain on which the iterates $\{f^{pn}|_U\}_n$ converge to an attracting or parabolic fixed point of $f^p$ (known as an attracting or parabolic component, respectively); or a domain on which the iterates $\{f^{pn}|_U\}_n$ converge to infinity locally uniformly (known as a {\em Baker domain}); or a topological disk on which $f^p$ is conjugate to a rigid irrational rotation (known as a {\em Siegel disk}).

If a Fatou component $U$ is neither periodic, nor preperiodic (that is, eventually periodic), then $f^i(U)\cap f^j(U)=\emptyset$ for all $i,j\geq 0$, $i\neq j$ and $U$ is called a {\em wandering domain}. On a wandering domain all limit functions must be constant  \cite{Fa20}.
% Bergweiler in ON THE LIMIT FUNCTIONS OF ITERATES IN WANDERING DOMAINS says that it is in Section 28 of Fatou's paper. I found the first part of the paper which has section 1-25 (published in 1919), and the third? part  which starts with  section 38.  He also cites also  p. 317 in Cremer, H.: U ̈ber die Schr ̈odersche Funktionalgleichung und das Schwarzsche Eckenab- bildungsproblem. - Sitzungsber. Math.-phys. Klasse, Ber. Verh. Sa ̈chs. Akad. Wiss. Leipzig 84, 1932, 291–324
Those for which the only limit function is the point at infinity are called {\em escaping}, while the rest are either {\em oscillating} (if infinity is a limit function and some other finite value also) or {\em dynamically bounded} (if all limit functions are points in the plane). A major open problem in transcendental dynamics is whether dynamically bounded wandering domains exist at all.

An essential role in the theory of holomorphic dynamics is played by the {\em singular values},   that is, those points for which not all inverse branches are locally well defined. In transcendental dynamics, these can be {\em critical values} (images of zeros of $f'$), {\em asymptotic values} or accumulations thereof.

For a wide class of functions known as {\em finite type maps} (those maps with a finite number of singular values), every Fatou component is periodic or preperiodic. Indeed, the absence of wandering domains for polynomials (actually for rational maps) \cite{sullivan} and for transcendental entire functions of finite type \cite{GK}, \cite{EandL}  was a major breakthrough in the theory of complex dynamics,  and meant that the possible types of dynamical behaviours of all such maps within the Fatou set was fully classified.  The result about the absence of wandering domains for the class of transcendental maps of finite type was particularly striking because in the 1970's  Baker \cite{Baker76} had constructed a transcendental entire function which had a  nested sequence of multiply connected Fatou components, each mapping to the next and whose orbits escaped to infinity, showing that wandering domains can indeed exist. While the wandering domains in Baker's example were multiply connected, since then a wide variety of examples
of simply connected wandering domains have been given; see, for example, \cite[p. 106]{Herman}, \cite[p. 414]{sullivan}, \cite[p. 564, p. 567]{Baker-wd}, \cite[p. 222]{Devaney-entire}, \cite[Examples 1 and 2]{EandL} and \cite[Sect. 4.3.]{FagHen2}. But it is only more recently that wandering domains have emerged as a major focus of attention, as the least understood of all the different types of Fatou components.

Indeed, several important advances have been made in recent years. For example, (oscillating) wandering domains have been constructed for functions in the Eremenko-Lyubich class $\mathcal{B}$ (those maps with a bounded sets of singular values) \cite{bishopwd,marshi, FJL}, a landmark result because escaping wandering domains have been shown not to exist for maps in this class \cite{EandL}. On another front, progress has been made \cite{bfjk, lasse-helena} in relating wandering domains  to the {\em postsingular set} (that is, the forward orbits of the singular values), a central classical problem in holomorphic dynamics that is well understood for the case of periodic components. We also mention the recent construction by Bishop \cite{Bish1} of  an entire function with Julia set of Hausdorff dimension~1, solving a long standing problem in transcendental dynamics; this function has multiply connected wandering domains, all of whose boundary components are Jordan curves.

Moreover, a detailed description of the dynamics of entire functions {\it within} multiply connected wandering domains was obtained in \cite{BRS}. Perhaps surprisingly it turns out that in these wandering domains all orbits behave in essentially the same manner, eventually landing in and remaining in a sequence of very large nested round annuli, and this detailed description has proved crucial in establishing results about classes of commuting transcendental entire functions \cite{BRS16}.

 Noticeably, however, very little is known about the full range of possible behaviours of the orbits inside {\em simply} connected  wandering domains, relative to the components themselves. One of the challenges is that several different types of behaviour are known to exist. Let us elaborate a bit further on this observation, while explaining at the same time the motivation for this paper.  Consider any holomorphic self-map of $\C\setminus\{0\}$,  or an entire  map $F:\C\to\C$ for which $z=0$ is either an omitted value or has itself as its only preimage;   for example, $F_\lambda(z)=\lambda z^d \exp(z)$ with $d\in\N$, $\lambda \in \C \setminus \{0\}$. Such a map $F$ can be lifted by the exponential map to a transcendental entire function $f:\C\to \C$ satisfying $\exp(f(z)) = F(\exp(z))$.  Observe that $f$ is not uniquely defined, since  any map of the form $f_k(z)=f(z)+2k\pi i$ for $k\in\Z$ will   have the same property. Now notice that if $F$ had, say, an   attracting component $U$ (not containing $z=0$), then any logarithm of~$U$, say $\widetilde{U}$, would be a wandering domain for $f_k$ (for an appropriate choice of $k$). Nevertheless, the orbits of points in $\widetilde{U}$ would still ``remember'' that they were lifted from an attracting component, in the sense that the iterates of any given point would be successively closer to the orbit of $\widetilde{p}:=\log p \in\widetilde{U}$,  where~$p$ is the fixed point of~$F$ in~$U$. Likewise, if $U$ had been, for example, a Siegel disk, the iterates of points in the successive images of $\widetilde{U}$ would ``rotate'' around a centre point (actually orbit), again the iterates of $\widetilde{p}$.   See Figure~\ref{siegel},   and also Figure~\ref{figparab} in Section~\ref{parabolic} for a lift of a parabolic component.  

\begin{figure}[hbt!]
\fboxsep=0.5pt
\begin{center}
\framebox{\includegraphics[width=0.45\textwidth]{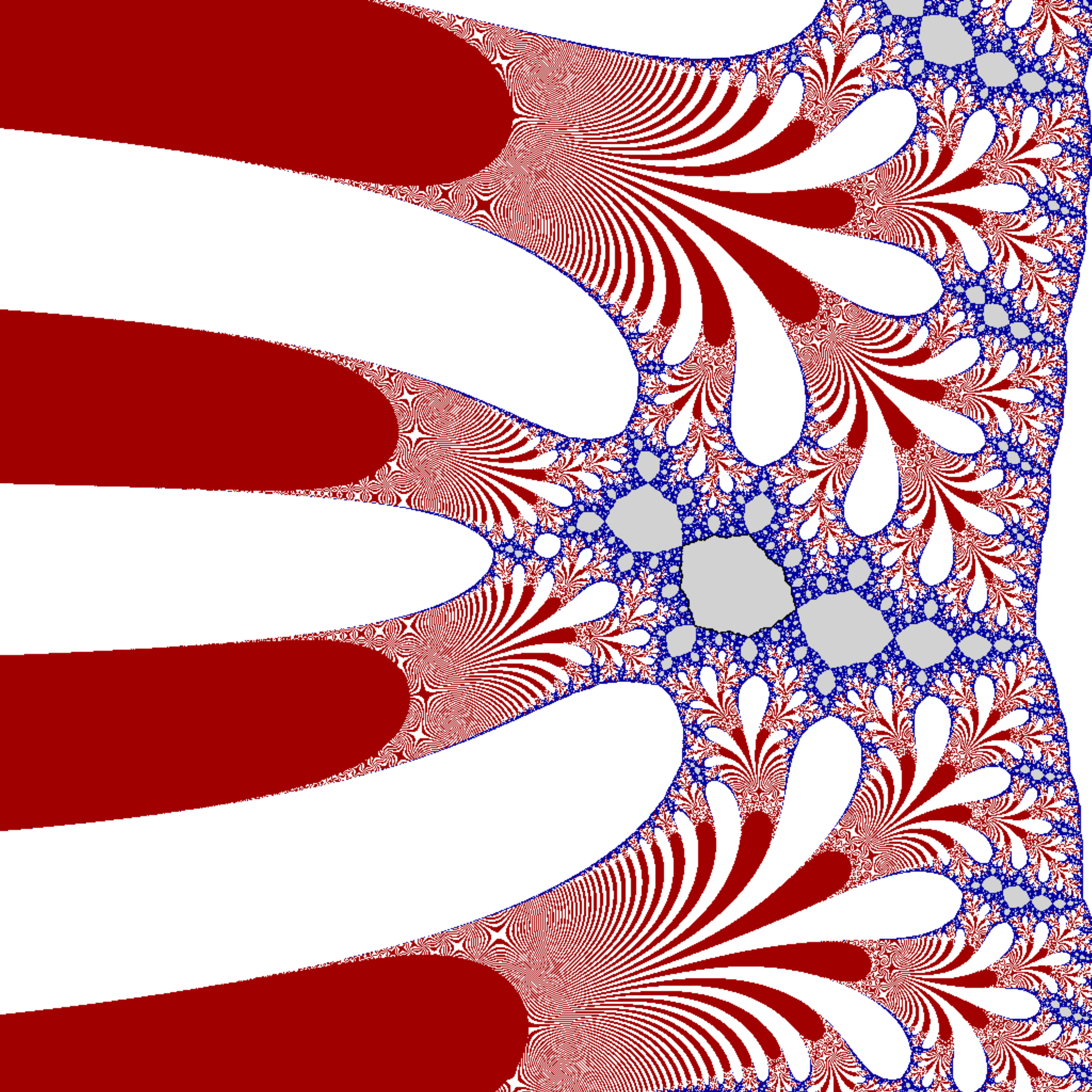}}
\framebox{\includegraphics[width=0.45\textwidth]{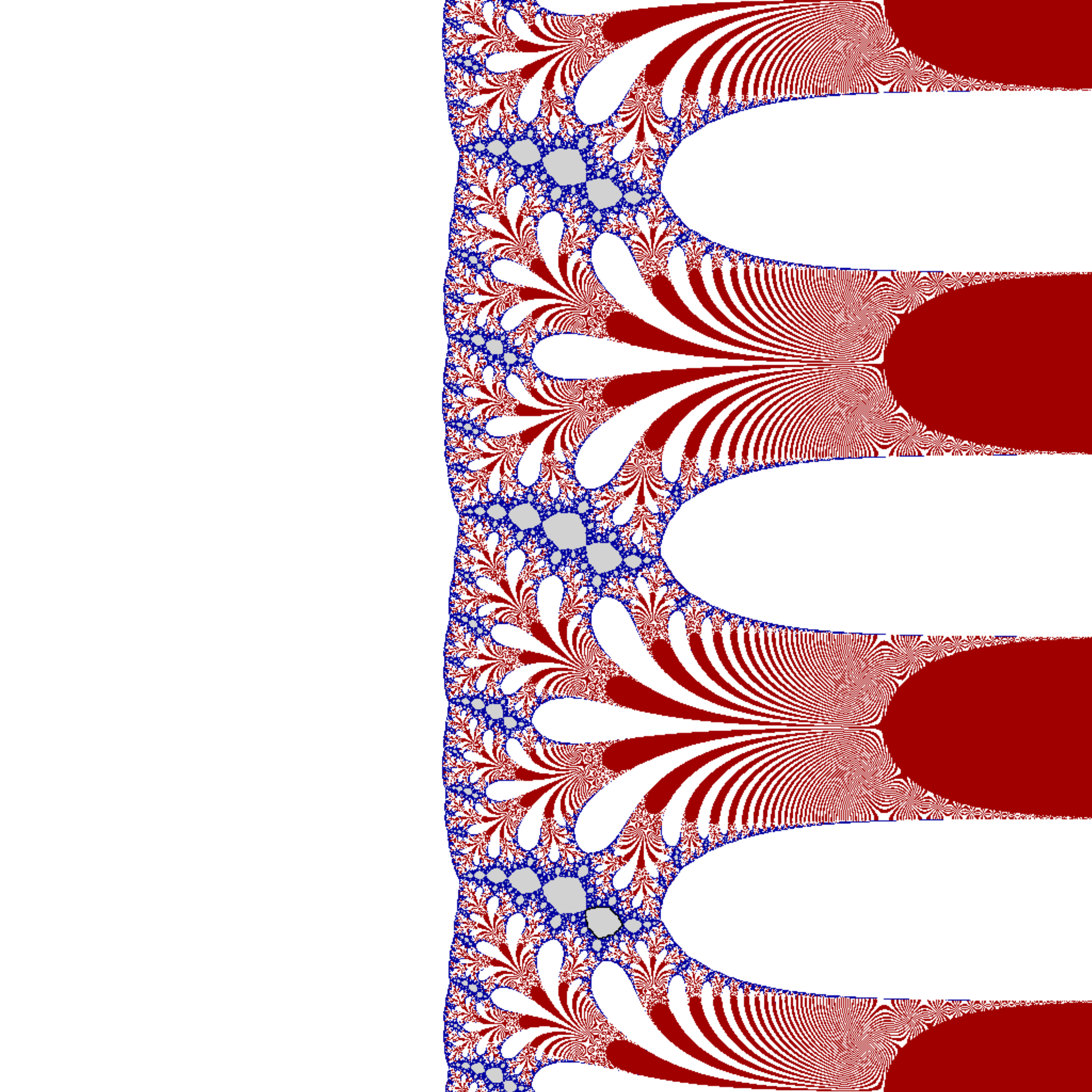}}
\end{center}
\caption{\label{siegel} \small Left:  Dynamical plane of $F(w)= \lambda w^2e^{-w}$ with  $\lambda=e^{2-\rho}/(2-\rho) $ and $\rho=e^{\pi i (1-\sqrt{5})}$. There is a (bounded) super-attracting component centred at $w=0$ (white) and a Siegel disk centred at $w_0=2-\lambda$ (gray). Right: Dynamical plane of $f(z)=2 z - e^z + \log \lambda$ satisfying $\exp(f(z))=F(\exp(z))$. The super-attracting component lifts to a Baker domain (white), while the Siegel disk lifts to infinitely many orbits of wandering domains on which $f$ is univalent (gray). See \cite{bergweilerexample,FagHen2,FaGa03} for details. The range is $[-9,9]\times[-9,9]$.}
\end{figure}

With this lifting procedure, one can construct examples of simply connected escaping wandering domains exhibiting the  three different types of internal dynamics that correspond to the possible dynamics inside a periodic component: attracting, parabolic or rotation-like. Thus we already have a contrast with multiply connected wandering domains, where only one type of dynamical behaviour is possible, as noted above. These observations suggest a very natural question: How special are the three examples above in the general world of wandering domains; in other words, is there a classification of wandering domains in the spirit of Fatou's classification of periodic Fatou components or is {\it any} orbit behaviour realizable? Let us note that, due to the lack of periodicity, the dynamics of~$f$ on a sequence of wandering domains can be thought of as a non-autonomous system (at every iterate we apply a ``different'' map), and such systems are {\em a priori} difficult to study because they may exhibit a wide range of behaviours. This might be an indication that such a  classification may not exist.  On the other hand, the successful description of the dynamics in multiply connected wandering domains obtained in \cite{BRS} is encouraging, and in this paper we obtain a classification of the dynamics in simply connected wandering domains.

%The dynamics of points which belong to wandering domains can be seen from two perspectives. On the one hand, the points have to move together with the wandering domain which contains them,   in the way that a passenger on a cruise ship has to follow the ship's trajectory. On the other hand, points  have an intrinsic dynamical behaviour relative to each other,   in the way that passengers   on the ship gather for dinner at the buffet, or move towards the ship's edges  to contemplate the water.

The dynamics of points which belong to wandering domains can be seen from two perspectives. While points have to move together with the wandering domain which contains them (in the way that passengers on a cruise ship must follow the ship's trajectory), on the one hand they may or may not cluster together as they move along (as happens when lifting an attracting component but not when lifting a Siegel disk), and on the other hand orbits may stay away from the boundaries of their domains (as happens when lifting an attracting basin but not when lifting a parabolic basin). Our results will address both of those points of view.

The most natural  intrinsic quantity that we have to hand -- intrinsic in that it does not depend on the embedding of the wandering domains in the plane -- are the hyperbolic distances between pairs of corresponding points of two orbits, and so our approach will be to evaluate how hyperbolic distances between such pairs of points evolve under iteration.

Let us recall that a domain $U\subset\C$ is \emph{hyperbolic} if its boundary (in $\C$) contains at least two points. For a hyperbolic domain $U$, let $\rho_U(z)$  denote the hyperbolic density at $z\in U$ and for $z,z'\in U$ let $\dist_U(z,z')$ denote the hyperbolic distance in $U$ between $z$ and $z'$. Also recall that if $U,V $ are hyperbolic domains, and $f:U\ra V$ is a holomorphic map, then the Schwarz-Pick  Lemma ensures that $f$ is a contraction for the hyperbolic distance. Hence, if $U\subset \C$ is a  wandering domain  of a transcendental entire function $f$ and we define $U_n$ to be the Fatou component containing $f^n(U)$, for $n\in \N$, we have that, given any two points $z,z' \in U$, the sequence
\[
\dist_{U_n}(f^n(z), f^n(z'))
\]
is decreasing and therefore converges to a value that we denote by
\[
c(z,z')=c_{U}(z,z'):=\lim_{n\to\infty}\dist_{U_n}(f^n(z), f^n(z')) \geq 0.
\]
Our first classification result shows that whether or not $c(z,z')$  is zero does not  actually depend on the chosen pair $(z,z')$, provided that the two points have distinct orbits.  We also give a criterion to discriminate between these cases based on the concept of \emph{hyperbolic distortion} \cite[Sect. 5,11]{BeardonMinda}.

\begin{defn}[Hyperbolic distortion]\label{HDdef}
If $f:U\ra V$ is a holomorphic map between two hyperbolic domains $U$ and $V$, then the \emph{hyperbolic distortion} of $f$ at $z$ is
\[
\|Df(z)\|_U^V:=\lim_{z'\ra z}\frac{\dist_V(f(z'),f(z))}{\dist_U(z',z)},
\]
and it equals the modulus of the hyperbolic derivative.
\end{defn}

\begin{thmA}[First classification theorem] \label{thm:Theorem A introduction}
Let $U$ be a simply connected wandering domain   of a transcendental entire function $f$ and let $U_n$ be the Fatou component containing $f^n(U)$, for $n \in \N$. Define the countable set of pairs
\[
E=\{(z,z')\in U\times U : f^k(z)=f^k(z') \text{\ for some $k\in\N$}\}.
\]
Then, exactly one of the following holds.
\begin{itemize}
\item[\rm(1)] $\dist_{U_n}(f^n(z), f^n(z'))\lran c(z,z')= 0 $ for all $z,z'\in U$, and we say that $U$ is {\em (hyperbolically) contracting};
\item [\rm(2)] $\dist_{U_n}(f^n(z), f^n(z'))\lran c(z,z') >0$ and $\dist_{U_n}(f^n(z), f^n(z')) \neq c(z,z')$
for all $(z,z')\in (U \times U) \setminus E$, $n \in \N$, and we say that $U$ is {\em  (hyperbolically) semi-contracting}; or
\item[\rm(3)] there exists $N>0$ such that for all $n\geq N$, $\dist_{U_n}(f^n(z), f^n(z')) = c(z,z') >0$ for all $(z,z') \in (U \times U) \setminus E$, and we say that $U$ is {\em (hyperbolically) eventually isometric}.
\end{itemize}

Moreover for $z \in U$ let $\lambda_n(z)$ be the hyperbolic distortion $\|Df(f^n(z))\|_{U_n}^{U_{n+1}}$. Then
\begin{itemize}
\item $U$ is contracting  if and only if $\sum_{n=0}^{\infty} (1-\lambda_n(z))=\infty$;
\item $U$ is eventually isometric if and only if $\lambda_n(z) = 1$, for $n$ sufficently large.
\end{itemize} 
\end{thmA}

Note that, by the Schwarz-Pick Lemma,~$U$ is eventually isometric if and only if $f:U_n\to U_{n+1}$ is univalent for large~$n$ and so a wandering domain obtained by lifting a Siegel disk is always eventually isometric. In contrast,  we show that lifting an attracting or parabolic component results in a contracting wandering domain. To distinguish between these two cases, we refine the classification of contracting wandering domains according to the rate of contraction.

 \begin{defn}[Rate of contraction]
\label{def:strongly contracting}
Let $U$ be a simply connected wandering domain of a transcendental entire function $f$ and let $U_n$ be the Fatou component containing $f^n(U)$, for $n \in \N$. We say that  $U$ is {\emph{strongly contracting}} if there exists $c\in (0,1)$ such that
$$\operatorname{dist}_{U_n}(f^n(z),f^n(z'))=O(c^n), \quad \text{for } z,z' \in U.$$
We say that $U$ is \emph{super-contracting} if it satisfies the stronger condition that
\[
\lim_{n \to \infty} (\operatorname{dist}_{U_n}(f^n(z),f^n(z')))^{1/n} = 0, \quad \text{for } z,z' \in U.
\]
\end{defn}

It is easy to see that the lift of an attracting component is strongly contracting, and we prove in Section \ref{sect:Contraction Trichotomy} that the lift of a parabolic component is contracting but not strongly contracting. We do this by a careful analysis of the behaviour of the hyperbolic distance between pairs of points in two orbits in any parabolic component; see Theorem~\ref{parabolic weakly contracting}.

A special case of super-contracting wandering domains is given by wandering domains which contain an orbit consisting of critical points. An example of such a super-contracting domain which does not arise from a lifting procedure is given in Theorem F.

Next, we give sufficient criteria for a wandering domain to be strongly contracting or super-contracting in terms of the long term average values of the hyperbolic distortion along the orbit of a point $z_0 \in U$. We also show that this quantity is independent of the point $z_0$.

\begin{thmB}\label{thm:strongly and super contracting}
Let $U$ be a simply connected wandering domain of a transcendental entire function $f$. For $n\in\N$, let $U_n$ be the Fatou component containing $f^n(U)$.  Fix  a point  $z_0 \in U$, and for $z \in U$,  $n \in \N$  let $\lambda_n(z) = \|Df(f^n(z))\|_{U_n}^{U_{n+1}}$. Then the following facts hold:
%Let $U$ be a simply connected wandering domain of a transcendental entire function $f$, let $U_n$ be the Fatou component containing $f^n(U)$, for $n \in \N$, let $z_0 \in U$, and let $\lambda_n(z) = \|Df(f^n(z))\|_{U_n}^{U_{n+1}}$, for $n \in \N$ and $z \in U$.
\begin{itemize}
\item[\rm(a)] If
$\limsup_{n \to \infty} \frac{1}{n}\sum_{k=1}^{n} \lambda_k(z_0)<1$, then $U$ is strongly contracting.
\item[\rm(b)] If $\lim_{n \to \infty} \frac{1}{n}\sum_{k=1}^{n} \lambda_k(z_0)=0$, then $U$ is super-contracting.
\item[\rm(c)] If $z \in U$, then $\limsup_{n \to \infty} \frac{1}{n}\sum_{k=1}^{n} \lambda_k(z) = \limsup_{n \to \infty} \frac{1}{n}\sum_{k=1}^{n} \lambda_k(z_0)$.
    \end{itemize}
\end{thmB}

Once the  behaviour of orbits  in relation to each other within simply connected wandering domains is well understood,  we turn to the question of how these orbits interact with the boundaries of the wandering domains.
The concept of orbits `approaching the boundary' is in itself delicate to define since it depends on the shape of the sets $U_n$, which may become highly distorted (think for example of the ratio of the diameter of the domains to their conformal radius, which may tend to infinity).  There are alternative candidates for the definition of convergence to the boundary (see Section \ref{sec:convergence}), but in this paper we use the following definition.

\begin{defn}[Boundary convergence] \label{defn:tending to boundary}
Let $U$ be a simply connected wandering domain of a transcendental entire function $f$ and let $U_n$ be the Fatou component containing $f^n(U)$, for $n \in \N$. We say that the orbit of $z\in U$ \emph{converges to the boundary} (of $U_n$) if and only if $\operatorname{dist}(f^{n}(z),\partial U_{n})\to 0$ as $n \to \infty$.
\end{defn}

We show that, with this definition, the following trichotomy holds.

\begin{thmC}[Second classification theorem]
Let $U$ be a simply connected wandering domain  of a transcendental entire function $f$ and let $U_n$ be the Fatou component containing $f^n(U)$, for $n \in \N$. Then exactly one of the following holds:
\begin{itemize}
\item[\rm(a)]
  $\liminf_{n\to\infty} \operatorname{dist}(f^{n}(z),\partial U_{n})>0$  for all $z\in U$,
    that is, all orbits stay away from the boundary;
\item[\rm(b)] there exists a subsequence $n_k\to \infty$ for which $\operatorname{dist}(f^{n_k}(z),\partial U_{n_k})\to 0$ for all $z\in U$, while for a different subsequence $m_k\to\infty$ we have that
    \[\liminf_{k \to \infty} \operatorname{dist}(f^{m_k}(z),\partial U_{m_k})>0, \quad\text{for }z\in U;\]
\item[\rm(c)] $\operatorname{dist}(f^{n}(z),\partial U_{n})\to 0$ for all $z\in U$, that is, all orbits converge to the boundary.
\end{itemize}
\end{thmC}

We remark that we actually prove a stronger version of Theorem C (see Theorem \ref{thm:everybody converges to the boundary or nobody}), which takes into account different definitions of converging to the boundary.

\subsubsection*{Construction of examples}
Theorems A and C combine to give nine different dynamical types of simply connected wandering domains. A natural question to ask is whether all of these can be realized.  As far as we know, the existing examples of simply connected wandering domains in the literature belong to one of the following three cases:  contracting and converging to the boundary (e.g.~lifts of parabolic components); contracting and staying away from the boundary (e.g.~lifts of attracting components); isometric and staying away from the boundary (e.g.lifts of Siegel disks). We will use approximation theory (see Section \ref{sect:examples}) to construct examples of each of the nine possibilities. In fact, we present a new general technique to construct bounded simply connected wandering domains (see Theorem \ref{thm:main construction}) which allows us to keep good control on the internal dynamics, as well as on the degree of the resulting maps  from one Fatou component to the next. As a key step we prove the following general result to show the existence of bounded simply connected wandering domains. Its statement uses the following terminology.

\begin{defn} We say that a  curve $\sigma$ \emph{surrounds} a set $B$ if and only if $B$ is contained in a bounded complementary component of $\sigma$. Also, for a Jordan curve $\eta$ we denote by $\operatorname{int}\eta$ the  bounded component of $\C \setminus \eta$ and by $\operatorname{ext}\eta$ the unbounded component of $\C \setminus \eta$.
\end{defn}

We can now state our result.
 
\begin{thmD}[Existence criteria for wandering domains]\label{WDexist}
Let $f$ be a transcendental entire function and suppose that there exist Jordan curves $\gamma_n$ and $\Gamma_n$, $n\ge 0$, a bounded domain $D$, a subsequence $n_k\ra\infty$ and  compact sets $L_k$   (associated with $\Gamma_{n_k}$)  such that% all sets are pairwise disjoint and such that
\begin{itemize}
\item[\rm(a)] $\Gamma_n$ surrounds $\gamma_n$, for $n \geq 0$;
\item[\rm(b)]  for every $k,n, m\geq0$, $m\neq n$ the sets $L_k, \ov{D}, \Gamma_m$  are in $\operatorname{ext}\Gamma_n$;
\item[\rm (c)]  $\gamma_{n+1}$ surrounds $f(\gamma_n)$, for $n \geq 0$;
\item[\rm (d)] $f(\Gamma_n)$ surrounds $\Gamma_{n+1}$, for $n \geq 0$;
\item[\rm (e)] $f(\overline D \cup \bigcup_{k\ge 0} L_k)\subset D$;
\item[\rm (f)]$ \max\{\operatorname{dist}(z,L_{k}): z \in \Gamma_{n_k}\} = o(\operatorname{dist}(\gamma_{n_k}, \Gamma_{n_k}))\;\text{ as}\; k \to \infty.
$
\end{itemize}
Then there exists an orbit of simply connected wandering domains $U_n$ %= f^n(U_0)$
such that $\overline{\operatorname{int} \gamma_n} \subset U_n \subset \operatorname{int}\Gamma_n$, for $n \geq 0$.

Moreover, if there exists $z_n \in \operatorname{int}\gamma_n$ such that both  $f(\gamma_n)$ and $f(\Gamma_n)$ wind $d_n$ times around $f(z_n),$  then $f:U_n \to U_{n+1}$ has degree $d_n$, for $n \geq 0$.
\end{thmD}

 We use Theorem~D to construct examples of each of the nine possible types,  and  also to construct simply connected wandering domains that contain any prescribed (finite) number of orbits consisting of critical points.  A wandering domain~$U$ will be called {\em $k$-super-attracting} if there exist critical points $z_1, \ldots, z_k \in U$, such that $f^{n}(z_1), \ldots, f^n(z_k)$ are critical points of~$f$, for all $n \in \N$.

\begin{thmE}[All types are realizable]\label{realizable}
\emph{(a)}\quad For each of the nine possible types of simply connected wandering domains arising from Theorems A and C, there exists a transcendental entire function with a bounded, simply connected escaping wandering domain of that type.

\emph{(b)}\quad
  For each $k\in \N$, there exists a transcendental entire function~$f$ having a bounded, simply connected escaping wandering domain~$U$ which is $k$-super-attracting.
\end{thmE}
Note that our examples in part~(b) of Theorem~E are super-contracting wandering domains that are not lifts of super-attracting components.

The bounded, simply connected wandering domains, $(U_n)$ say, constructed in Theorem~E all have a shape that tends to the shape of a Euclidean disk as $n \to \infty$, but in fact the construction can easily be modified to give wandering domains with different limiting shapes. Also, in forthcoming work we combine our new ideas with techniques introduced by Eremenko and Lyubich \cite{pathex} to construct examples of oscillating simply connected wandering domains that are bounded and have various types of internal dynamics.

Finally, we show that our methods  can be adapted to construct simply connected wandering domains bounded by Jordan curves. Theorem~E is proved by obtaining entire functions that approximate sequences of translates of Blaschke products associated with sequences of Jordan curves with the properties given in Theorem~D, and we show that, if these Blaschke products are in a certain sense uniformly expanding and have uniformly bounded degree, then the resulting wandering domains have Jordan curve boundaries. Our proof includes a result on the Euclidean lengths of vertical geodesics of annuli whose boundary components have finite length, which is of independent interest.

\subsubsection*{Structure of the paper}
The first part of the paper (Sections~2, 3 and 4) is devoted to studying the possible behaviours of orbits in simply connected wandering domains, proving Theorems~A,~B and~C. We begin in Section 2 by setting up related non-autonomous dynamical systems of self maps of the unit disk. We prove several results in this general setting which may be of wider interest. In Section~3 we use our results from Section~2 to prove Theorems~A and~B. We prove Theorem~C in Section~4.

The second part of the paper (Sections 5, 6 and 7) is devoted to the construction of examples. In Section~5 we give the proof of Theorem~D and develop a new general technique for constructing   bounded wandering domains. In Section 6 we use this technique to construct examples of every possible behaviour classified in the first part of the paper, proving Theorem~E. Finally, in Section 7 we show that, under certain conditions, our new construction technique gives simply connected wandering domains that are Jordan domains.

\subsection*{Acknowledgments}
We are grateful to the Universitat de Barcelona, to the IMUB and to the Open University for hosting part of this research. We would like to thank Chris Bishop, Xavier Jarque, Misha Lyubich, Lasse Rempe-Gillen and Dave Sixsmith for inspiring discussions.

%\subsection*{Notation} {\violet We denote the complex plane by $\C$, the Riemann sphere by $\widehat{\C}$, the Euclidean unit disk by $\D$ [Add notation for lengths ettc?]}Given a set $K$ we call \emph{complementary component} any connected component of $\C\setminus \ov{K}$.  {\violet By a \emph{preimage} of a set $K$ (under a map $f$) we mean a connected component of $f^{-1}(K)$.}}
%%%%%%%%%%%%%%%%%%%%%%%%%%%%%%%%%%%%%%%
\section{Non-autonomous dynamical systems of self maps of the unit disk}\label{sect:inner}

In this section we prove several results in the general setting of non-autonomous forward dynamical systems of holomorphic self maps of the unit disk fixing the origin. These results may be of wider interest with applications outside holomorphic dynamics. In the next section, we apply them to the case of transcendental entire functions with simply connected wandering domains in order to prove Theorem A and Theorem B.

% I think that everything from here to ***, and especially the figure, could/should be moved   to the beginning of Section 3. I now think that this construction is only used there, not even in Section 4. This would also simplify the statement of Lemma 3.1 if we move the setup right before it.]

Our proofs are based on hyperbolic distances in the unit disk and we make frequent use of the fact that

\begin{equation}\label{hdist0}
\dist_{\D}(w,0) = \int_{0}^{|w|} \frac{2\,dt}{1-t^2} =  \log \left( \frac{1 + |w|}{1 - |w|}\right), \quad \text{for } w \in \D.
\end{equation}

  In our first result, we characterize when the limits of such systems of holomorphic self maps of the unit disk are identically equal to zero, in terms of the values of the derivatives of the maps at~0. In particular, unless $|g_n'(0)|\to 1$ as $n\to\infty$, the limit of the maps $G_n$ is always zero.

\begin{thm}[Criterion for converging to zero]\label{Thm zero}
For each $n \in \N$, let $g_n: \D \to \D$ be holomorphic with $g_n(0) = 0$ and $|g_n'(0)| = \lambda_n$, and let $G_n = g_{n}\circ \cdots\circ g_1$.
\begin{itemize}
\item[\rm(a)]
If $\sum_{n=1}^{\infty}(1 - \lambda_n) = \infty$, then $G_n(w) \to  0$ as $n \to \infty$, for all $w \in \D$.

\item[\rm(b)]If $\sum_{n=1}^{\infty}(1 - \lambda_n) < \infty$, then $G_n(w) \nrightarrow 0$ as $n \to \infty$, for all $w \in \D$ for which $G_n(w) \neq 0$ for all $n \in \N$.
\end{itemize}
\end{thm}

\begin{proof}
We begin with ideas used by Beardon and Carne~\cite{BC92}. First, it follows from the hyperbolic triangle inequality and hyperbolic contraction that, if $\psi : \D \to \D$ is holomorphic, then for all $w\in\D$ we have
\begin{equation}\label{zero1}
\dD(0,\psi(w))\leq \dD(0,\psi(0))+\dD(\psi(0),\psi(w))\leq  \dD(0,\psi(0))+\dD(0,w),
\end{equation}
and, similarly,
\begin{equation}\label{zero2}
\dD(0,\psi(0))\leq \dD(0,\psi(w))+\dD(0,w), \text{ for all } w\in\D.
\end{equation}

We also use the fact that
\begin{equation}\label{zero3}
\sum_{n=1}^{\infty}(1 - \lambda_n) = \infty \iff \lambda_{m+n}\cdots \lambda_{m+1} \to 0 \text{ as } n \to \infty, \quad \text{for all } m \in \N.
\end{equation}
In the case when $\lambda_n \ne 0$, for all~$n$, and the right-hand side is $\lambda_{n}\cdots \lambda_{1} \to 0 \text{ as } n \to \infty$, this statement is a standard property of infinite products proved by taking logarithms. Here it is possible that some or all of the terms $\lambda_n$ are zero, so the right-hand side of \eqref{zero3} takes account of these possibilities.

Now take $w_0 \in \D$ and, for simplicity, denote $G_n(w_0)$ by $w_n$, for $n \in \N$.

To prove part (a), we assume that $\lambda_{m+n}\cdots \lambda_{m+1} \to 0$ as $n \to \infty$, for all $m \in \N$, and deduce that $w_n \to 0$ as $n\to \infty$. Suppose that $w_n \nrightarrow 0$ as $n\to \infty$. Since $w_n = g_n(w_{n-1})$,  we deduce by Schwarz's Lemma that  $|w_n| \leq |w_{n-1}|$, and hence that $|w_n|$ decreases to some $d>0$ as $n \to \infty$.

First choose $m \in \N$ so large that $|w_m|$ is sufficiently close to $d$ to ensure that
\begin{equation}\label{zero4}
\dD(0,w_{n+m}/w_m) > \dD(0,w_m), \quad \text{for } n \in \N.
\end{equation}
Next we fix $n \in \N$ and
%apply~\eqref{zero1} to the holomorphic function defined by
define the holomorphic map
\[
\psi(w) = (g_{m+n} \circ \cdots \circ g_{m+1}(w))/w, \quad \text{for } w \in \D \setminus \{0\},
\]
with
\[
\psi(0) = (g_{m+n} \circ \cdots \circ g_{m+1})'(0) = \lambda_{m+n} \cdots \lambda_{m+1}.
\]
Applying \eqref{zero1} to the function $\psi$ at the point $w=w_m$ gives
\begin{eqnarray*}
\dD(0,w_{n+m}/w_m) & = & \dD(0,\psi(w_m))\\
& \leq & \dD(0,\psi(0)) + \dD(0,w_m)\\
& \leq & \dD(0,\lambda_{m+n} \cdots \lambda_{m+1}) + \dD(0,w_m).
\end{eqnarray*}
Since we have assumed that $\lambda_{m+n}\cdots \lambda_{m+1} \to 0$ as $n \to \infty$,  it follows that $\dD(0,w_{n+m}/w_m) \leq \dD(0,w_m)$, for $m$ sufficiently large. This, however, contradicts~\eqref{zero4}, showing that  $w_n \to 0$ as $n \to \infty$.

To prove part (b), we assume that, for some $m_0 \in \N$, $\lambda_{m_0+n}\cdots \lambda_{m_0+1} \to \lambda > 0$ as $n \to \infty$, and deduce that whenever $w_n \neq 0$, for all $n \in \N$, we  have $w_n \nrightarrow 0$ as $n \to \infty$. Suppose that $w_n \rightarrow 0$ as $n\to \infty$.

First choose $m$ so large that $m \geq m_0$ and
\begin{equation}\label{zero5}
\dD(0,w_m) < \dD(0,\lambda),
\end{equation}
and note that, for such $m$,
\begin{equation}\label{zero6}
\lambda_{m+n}\cdots \lambda_{m+1} \geq \lambda_{m+n}\cdots\lambda_{m_0+1} = \lambda_{m_0 + (m-m_0)+n}\cdots\lambda_{m_0+1} \geq \lambda, \quad \text{for } n \in \N.
\end{equation}
Next we fix $n \in \N$ and apply~\eqref{zero2} with $\psi$ defined as earlier and $w=w_m$ to give
\[
\dD(0,\lambda_{m+n}\cdots \lambda_{m+1}) \leq \dD(0,w_{m+n}/w_m) + \dD(0,w_m).
\]
Letting $n \to \infty$, we obtain a contradiction to~\eqref{zero5} in view of~\eqref{zero6} and hence to the  supposition that $w_n \to 0$ as $n \to \infty$. This completes the proof.
\end{proof}

The following corollary to Theorem~\ref{Thm zero} shows that if  the hyperbolic distance between two distinct orbits converges to zero, then the same occurs for every pair of orbits.

\begin{cor}\label{cor zero}
For $n \in \N$, let $g_n: \D \to \D$ be holomorphic and let $G_n = g_{n}\circ \cdots\circ g_1$. If there exist $w_0,w_0' \in \D$ such that $G_n(w_0') \neq G_n(w_0)$ for all $n \in \N$ and $\dist_{\D}(G_n(w_0'),G_n(w_0)) \to 0$ as $n \to \infty$, then
\[
\dist_{\D}(G_n(w),G_n(w_0)) \to 0 \text{ as } n \to \infty, \quad \text{for all } w \in \D.
\]
\end{cor}
\begin{proof}
For each $n \in \N$, let $w_n = g_n(w_{n-1})$ and, for $n \geq 0$, let $M_n:\D \to \D$ be a M\"obius map satisfying $M_n(w_n)=0$. Then, for each $n \in \N$, the map $h_n=M_n \circ g_n \circ M_{n-1}^{-1}$ is a holomorphic self map of the unit disk and $h_n(0)=0$. For $n \in \N$, let $H_n := h_{n}\circ \cdots\circ h_1$ and notice that $H_n(0)=0$.
%[The previous sentence stated that you apply theorem~\ref{Thm zero} but you do it later in the proof, not here, and only to the $H_n$, not to the $h_n$.  Ifound this proof difficult to read in general so I added some detail.]  % We now apply Theorem~\ref{Thm zero} to the maps $h_n$ and the maps $H_n$ defined by $H_n = h_{n}\circ \cdots\circ h_1$, for $n \in \N$.
Since M\"obius maps are isometries and $H_n = M_n \circ G_n \circ M_0^{-1}$, for $n \in \N$, we have
\begin{align*}
\dist_{\D}(0, H_n(M_0(w_0'))) &= \dist_{\D}(H_n(0), H_n(M_0(w_0'))) \\
&= \dist_{\D}(M_n\circ G_n\circ M_0^{-1}(0),M_n\circ G_n\circ M_0^{-1}\circ M_0 (w_0'))\\
&= \dist_{\D}(M_n\circ G_n(w_0),M_n\circ G_n(w_0'))\\
&= \dist_{\D}(G_n(w_0),G_n(w_0')) \to 0 \quad \text{as } n \to \infty,
\end{align*}
and hence $H_n(M_0(w_0')) \to 0$ as $n \to \infty$. Since $H_n(M_0(w_0')) = M_n(G_n(w_0')) \neq 0$, for each $n \in \N$, it follows from Theorem~\ref{Thm zero} that $H_n(w') \to 0$ as $n \to \infty$ for all $w' \in \D$. The result now follows since
\[
\dist_{\D}(G_n(w),G_n(w_0)) = \dist_{\D}(H_n(M_0(w)),H_n(0)) = \dist_{\D}(H_n(M_0(w)),0), \quad \text{for } w \in \D. \qedhere
\]
\end{proof}

Theorem~\ref{Thm zero} and Corollary~\ref{cor zero} will be used in the proof of Theorem A (see Section~\ref{sec:proofofthmA}).

We now prove several results giving estimates for the {\it rate} at which limits tend to zero in the case when the limit in Theorem~\ref{Thm zero} is identically equal to zero. The results proven in the remainder of this section will be used in Section~\ref{sec:subclassification} to prove Theorem B, that is, the subclassification of contracting wandering domains.

We use the following result which includes a generalization of Schwarz's Lemma.

\begin{lem}[Variation of Schwarz's Lemma] \label{lem:condition convergence technical} Let $\psi:\D\ra\D$ be holomorphic. Then
\[\frac{|\psi(0)|-|w|}{1-|\psi(0)||w|}\leq |\psi(w)|\leq\frac{|\psi(0)|+|w|}{1+|\psi(0)||w|}, \quad \text{ for $w \in\D$}. \]
\end{lem}
\begin{proof}
The right-hand inequality arises from~\eqref{zero1} and is given in \cite[p.217]{BC92}. We prove the left-hand inequality using similar methods.
First note that it follows from~\eqref{zero2} that
\[
\dD(0,\psi(w))\geq \dD(0,\psi(0))-\dD(0,w),
\]
that is,
$$
\log\frac{1+|\psi(w)|}{1-|\psi(w)|}\geq\log\frac{1+|\psi(0)|}{1-|\psi(0)|}-\log\frac{1+|w|}{1-|w|}.
$$
By the monotonicity of the logarithm, this is equivalent to the following inequality:
$$\frac{1+|\psi(w)|}{1-|\psi(w)|}\geq \left(\frac{1+|\psi(0)|}{1-|\psi(0)|}\right)\left(\frac{1-|w|}{1+|w|}\right),$$
which gives
$$ |\psi(w)| \geq \frac{\left(\frac{1+|\psi(0)|}{1-|\psi(0)|}\right)\left(\frac{1-|w|}{1+|w|}\right)-1}{\left(\frac{1+|\psi(0)|}{1-|\psi(0)|}\right)\left(\frac{1-|w|}{1+|w|}\right)+1}=\frac{|\psi(0)|-|w|}{1-|\psi(0)||w|},$$
as claimed.
\end{proof}

We make frequent use of the following corollary of Lemma~\ref{lem:condition convergence technical}.

\begin{cor}\label{corSchwarz}
Let $g:\D \to \D$ be holomorphic with $g(0) = 0$ and $|g'(0)| = \lambda$. Then, for all $w \in \D$,
\[
|w| \left( \frac{\lambda - |w|}{1 - \lambda |w|} \right) \leq |g(w)| \leq |w| \left( \frac{\lambda + |w|}{1 + \lambda |w|} \right)
\]
\end{cor}
\begin{proof}
The result follows by applying Lemma~\ref{lem:condition convergence technical} to the holomorphic map $\psi:\D \to \D$ defined by
\[
\psi(w) = g(w)/w, \quad \text{for } w \in \D \setminus \{0\},
\]
with $\psi(0) = g'(0)$.
\end{proof}

We first use Corollary~\ref{corSchwarz} to prove the following result giving rather precise upper and lower estimates of the rate at which the sequences $|G_n(w)|$ in Theorem~\ref{Thm zero} decrease, expressed in terms of the derivatives $|g'_n(0)|$. This result can be used to give a more direct proof of Theorem~\ref{Thm zero}; see the remark after the proof of Theorem \ref{contractionrate}.

\begin{thm}\label{contractionrate}
For each $n \in \N$, let $g_n: \D \to \D$ be holomorphic with $g_n(0) = 0$ and $|g_n'(0)| = \lambda_n = 1 - \mu_n$, and let $G_n = g_{n}\circ \cdots\circ g_1$.  If $w\in\D$ and $w_n=G_n(w)$, $n\in\N$, then
\begin{itemize}
\item[\rm(a)]
\begin{equation}\label{cont1}
|w_n| \leq |w| \prod_{k=1}^n (1 - c_w\mu_k), \quad \text{where } c_w = (1-|w|)/2;
\end{equation}
\item[\rm(b)] if $|w| \leq \lambda_k$, for $1 \leq k \leq n$, then
\begin{equation}\label{cont2}
|w_n| \geq |w| \prod_{k=1}^n (1-d_w\mu_k), \quad \text{where } d_w = \frac{1+|w|}{1-|w|}.   \end{equation}
\end{itemize}
\end{thm}
\begin{proof}
Set $w_0=w$. We begin the proof of part (a) by noting that it follows from Corollary~\ref{corSchwarz} that, for $k \geq 0$ and $w \in \D$,
\begin{eqnarray*}
|w_{k+1}|  =   |g_{k+1}(w_k)| & \leq & |w_k| \left( \frac{\lambda_{k+1} + |w_k|}{1 + \lambda_{k+1} |w_k|} \right)\\
& = & |w_k| \left( 1 - \frac{\mu_{k+1}( 1- |w_k|)}{1 + \lambda_{k+1} |w_k|} \right) \\
& \leq & |w_k| (1-\frac{ \mu_{k+1}(1-|w_k|)}{2}))\\
& \leq & |w_k|(1-c_w\mu_{k+1}),
\end{eqnarray*}
where the third inequality follows because $\lambda_{k+1} |G_k(w)|<1$ and the last inequality follows because $|w_k|=|G_k(w)| \leq |w|$ by Schwarz's Lemma.
%A similar argument shows that, for $w \in \D$,
%\[
%|G_1(w)| = |g_1(w)| \leq |w|  (1-c_w\mu_1).
%\]
The result of \eqref{cont1} now follows and this completes the proof of part (a).

We now prove part~(b). Using Corollary~\ref{corSchwarz} again,
\begin{equation}\label{Gnlower}
|w_{k+1}|  =   |g_{k+1}(w_{k})|  \ge  |w_{k}| \left( \frac{\lambda_{k+1} - |w_k|}{1 - \lambda_{k+1} |w_k|} \right).
\end{equation}
Now we use the elementary calculus estimate that
\[
\frac{\lambda-r}{1-\lambda r} \ge 1-\left(\frac{1+r}{1-r}\right)(1-\lambda),\quad \text{for }0<r<\lambda \le1,
\]
to deduce from \eqref{Gnlower} that, for $k \geq 0$, if $|w|\le \lambda_{k+1}$, then
\[
|w_{k+1}| \ge |w_k|\left(1-\left(\frac{1+|w_k|}{1-|w_k|}\right)\mu_{k+1}\right) \ge |w_k|\left(1-\left(\frac{1+|w|}{1-|w|}\right)\mu_{k+1}\right),
\]
using the fact that $|w_k|=|G_k(w)| \le |w|$ again. The result of \eqref{cont2} now follows and this completes the proof of part (b).
\end{proof}

\begin{rem*}
Theorem~\ref{contractionrate} can be used  to give a proof of Theorem~\ref{Thm zero}. To do so,  it is first necessary to use Hurwitz' Theorem in order to show that either $G_n(w) \to 0$ as $n \to \infty$ for all $w \in \D$ or $G_n(w) \to 0$ as $n \to \infty$ only for those points $w \in \D$ for which $G_n(w) = 0$ eventually.
\end{rem*}

We now prove another result giving upper estimates for the rate at which the sequences $|G_n(w)|$ decrease, this time expressed in terms of the average of the derivatives $|g_n'(0)|$. The proof of this result is also based on Corollary~\ref{corSchwarz}.

\begin{thm}\label{avcont}
For each $n \in \N$, let $g_n: \D \to \D$ be holomorphic with $g_n(0) = 0$ and $|g_n'(0)| = \lambda_n = 1 - \mu_n$, and let $G_n = g_{n}\circ \cdots\circ g_1$. Then, for all $n \in \N$, if $w_0 \in \D$ and $w_n = G_n(w_0)$, for $n \in \N$,
\begin{equation}\label{amgm}
|w_n| \leq \left( \frac{1}{n}\sum_{k=1}^n \lambda_k + \frac{1}{n} \sum_{k=0}^{n-1} |w_k|\right)^n.
\end{equation}
Hence
\begin{itemize}
\item[\rm(a)]
if
\[
\limsup_{n \to \infty} \frac{1}{n} \sum_{k=1}^n \lambda_k = a < 1,
\]
then
\[
|G_n(w)| = O(c^n)\;\text{ as } n \to \infty, \quad \text{for } w \in \D, \text{ where } c \in (a,1);
\]
\item[\rm(b)]
if
\[
\lim_{n \to \infty} \frac{1}{n} \sum_{k=1}^n \lambda_k = 0,
\]
then
\[
|G_n(w)|^{1/n} \to 0 \;\text{ as } n \to \infty, \quad \text{for } w \in \D.
\]
\end{itemize}
\end{thm}
\begin{proof}
By using Corollary~\ref{corSchwarz}, and then applying the fact that the geometric mean   of $n$ positive numbers is at most equal to their arithmetic mean, we see that, for $w_0 \in \D$ and $n \in \N$,

\begin{eqnarray*}
|w_n| &\leq& |w_0| \left(\frac{\lambda_{n}+|w_{n-1}|}{1+\lambda_{n}|w_{n-1}|}\right)\dots \left(\frac{\lambda_{1}+|w_{0}|}{1+\lambda_{1}|w_{0}|}\right)\\
&\leq& |w_0|((\lambda_{n}+|w_{n-1}|)\dots (\lambda_{1}+|w_{0}|))  \\
&\leq& |w_0| \left(\frac{1}{n}((\lambda_{n}+|w_{n-1}|) + \cdots + (\lambda_{1}+|w_{0}|))\right)^n  \\
&=& |w_0| \left( \frac{1}{n} \sum _{k=1}^{n}\lambda_k+ \frac{1}{n} \sum _{k=0}^{n-1}|w_k|\right)^n.
\end{eqnarray*}
This proves~\eqref{amgm}.

Next, if $\limsup_{n \to \infty} \frac{1}{n} \sum_{k=1}^n \lambda_k = a < 1$, then $\sum_{k=1}^{\infty}(1-\lambda_k) = \infty$ and so it follows from Theorem~\ref{Thm zero} that $w_n \to 0$ as $n \to \infty$ and hence that $\frac{1}{n} \sum _{k=0}^{n-1}|w_k| \to 0$ as $n \to \infty$. So, in this case, it follows from~\eqref{amgm} that
\[
|G_n(w_0)|^{1/n} = |w_n|^{1/n} \leq a + o(1) \; \text{ as } n \to \infty.
\]
The results of parts~(a) and~(b) now follow.
\end{proof}

Theorem~\ref{contractionrate}~(a) and Theorem~\ref{avcont} give uniform upper estimates on the rate that $|G_n(w)|$ tends to~0, in the situation where $\sum_{n=1}^{\infty} (1-\lambda_n)=\infty$. It is natural to ask whether we can demonstrate such a uniform rate if we know the rate at which $|G_n(w)|$ tends to~0 on some subset of $\D$. It is clear that we cannot deduce any uniform rate at which $G_n(w)\to 0$ from information about the behaviour of $G_n$ at a {\it single} point $w_0\in\D$. However, if we have an upper bound for $|G_n(w)|$ on some circle $\{w:|w|=r_0\}$, where $0<r_0<1$, then we can obtain an upper estimate for $|G_n(w)|$ for all $w\in \D$ by applying the following simple proposition.

\begin{prop}[Hadamard convexity] \label{lem:ST}
Let $f:\D\to\D$ be holomorphic and satisfy
\[
|f(w)|\leq a, \quad \text{for $|w|\leq r_0$,}
\]
where $0 <a\leq r_0<1$. Then,
\[
|f(w)|\leq a^{\frac{\log r }{\log r_0}} \quad \text{for $|w|\leq r$,}
\]
for all $r$ such that $r_0\leq r < 1$.
\end{prop}
\begin{proof}
For $0\leq r<1$, let
\[
M(r)=M(r,f):=\sup_{|z|=r}|f(z)|
\]
 denote the maximum modulus function and put
\[
\varphi(t)=\log M(e^t), \quad \text{for $-\infty < t < 0$.}
\]
Then $\varphi$ is convex by Hadamard's Three Circles Theorem \cite[page~172]{Titch39}, negative and increasing, and by hypothesis
$\varphi(\log r_0) \leq \log a$. Hence
\[
\varphi(t)\leq \left( \frac{\log a}{\log r_0} \right) t,  \quad  \text{for } -\infty < t < 0;
\]
that is,
\[
\log M(r) \leq  \left( \frac{\log a}{\log r_0} \right) \log r, \quad \text{for } r_0\leq r < 1,
\]
 and hence
\[
M(r)\leq a^{\frac{\log r }{\log r_0}}, \quad \text{for $r_0\leq r < 1$},
\]
as required.
\end{proof}

\begin{rem*}
In Proposition~\ref{lem:ST}, the circle $\{w : |w| = r_0\}$ can be replaced by any subset of $\D$ of positive logarithmic capacity, using a more delicate argument involving Green potentials in~$\D$. We omit the details.
\end{rem*}

 \section{Contraction trichotomy: Proof of Theorems A and B}\label{sect:Contraction Trichotomy}
\label{sec:1stclassification}
This section is devoted to a classification of simply connected wandering domains based on hyperbolic distances between orbits of points.  More precisely we prove Theorems A and B  and  we also show that lifts of parabolic components are contracting yet not strongly contracting; (see Theorem~\ref{parabolic weakly contracting}).

The proofs are based on the results from Section~2 concerning self maps of the unit disk. We first show how the hyperbolic distances between orbits of points in the wandering domain compare with the distances between related orbits of points in the unit disk. We also compare the hyperbolic distortion along an orbit of a point in the wandering domain with the  derivatives of the related maps of the unit disk.

Let $f$ be a transcendental entire function with a simply connected  wandering domain $U_0$ and let $U_n$ be the Fatou component containing $f^n(U_0)$, for $n \in \N$. Note that each of the domains $U_n$ is simply connected; indeed, if some $U_n$ is multiply connected, then by \cite[Theorem~3.1]{Baker-wd}, all the Fatou components are bounded, so~$f$ is a proper map between Fatou components and the claim follows from the Riemann--Hurwitz formula.   Although $U_n = f^n(U_0)$ if $U_0$ is bounded, this is not necessarily true in the case that $U_0$ is unbounded when $U_n \setminus f^n(U)$ may contain one point; see for example \cite{Herring}.

We prove Theorem A and Theorem B by considering a sequence $(g_n)$ of  holomorphic self maps of the unit disk  associated to $f$ and $U_n$ in the following way.
Fix a point $z_0 \in U_0$ and, for each $n \ge 0$, choose $\varphi_n: U_n \to {\mathbb D}$ to be a  Riemann  map such that $\varphi_n(f^n(z_0)) = 0$. Then, for $n\in \N$, consider the holomorphic maps $g_n:\D\ra\D$ defined as
\[
g_n = \varphi_n \circ  f \circ \varphi_{n-1}^{-1},
\]
and the composite maps $G_n:\D\to \D$  defined as
\[
 G_n =   g_{n}\circ \cdots\circ g_1 =  \varphi_n \circ  f^n \circ \varphi_0^{-1}.
\]
Because of the choice of  normalization for the Riemann maps we have that $g_n(0)=G_n(0)=0$.
This set up is illustrated in Figure~\ref{hbt!}.
Each of the maps $g_n$ and $G_n$ is an inner function, but we do not use this fact in this paper.

 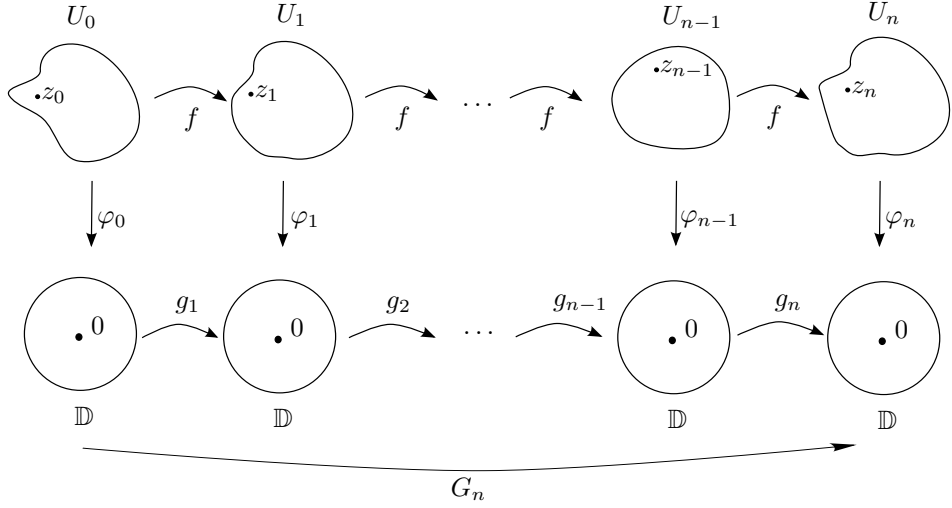
\begin{figure}\label{hbt!}
\begin{center}
\def\svgwidth{0.8\textwidth}
\begingroup%
  \makeatletter%
    \setlength{\unitlength}{\svgwidth}%
  \makeatother%
  \begin{picture}(1,0.50422362)%
    \put(0,0){\includegraphics[width=\unitlength]{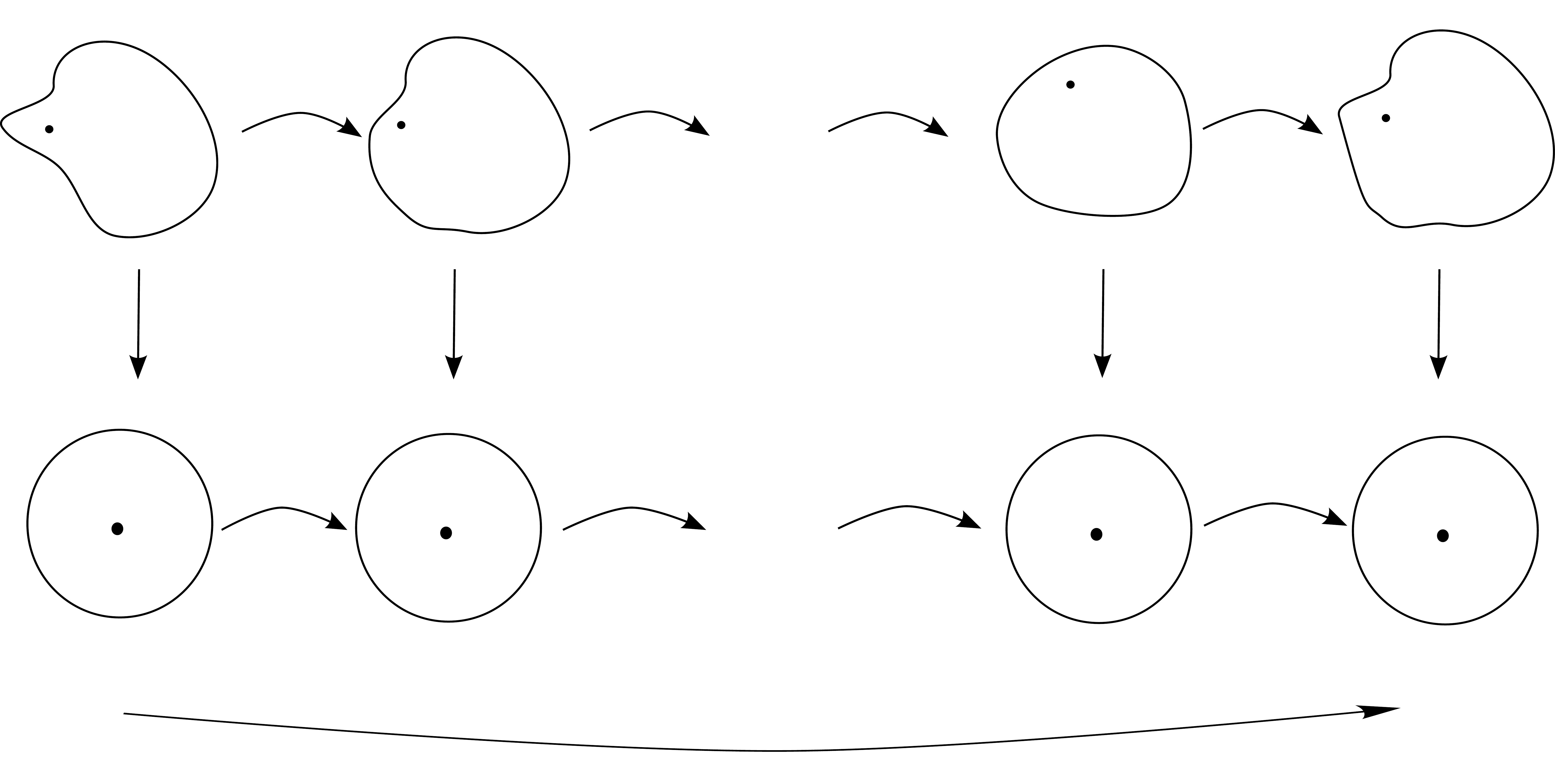}}%
    \put(0.28209737,0.06776929){\color[rgb]{0,0,0}\makebox(0,0)[lb]{\small{$\D$}}}%
    \put(0.30003041,0.16489989){\color[rgb]{0,0,0}\makebox(0,0)[lb]{\small{0}}}%
    \put(0.3,0.28){\color[rgb]{0,0,0}\makebox(0,0)[lb]{\small{$\phi_1$}}}%
    \put(0.71341713,0.28){\color[rgb]{0,0,0}\makebox(0,0)[lb]{\small{$\phi_{n-1}$}}}%
    \put(0.40163958,0.18928622){\color[rgb]{0,0,0}\makebox(0,0)[lb]{\small{$g_2$}}}%
    \put(0.63606945,0.08524224){\color[rgb]{0,0,0}\makebox(0,0)[lb]{\small{ }}}%
    \put(0.53620543,0.26271489){\color[rgb]{0,0,0}\makebox(0,0)[lb]{\small{ }}}%
    \put(0.47,-0.01){\color[rgb]{0,0,0}\makebox(0,0)[lb]{\small{$G_n$}}}%
    \put(0.70038518,0.06687104){\color[rgb]{0,0,0}\makebox(0,0)[lb]{\small{$\D$}}}%
    \put(0.71831822,0.16400164){\color[rgb]{0,0,0}\makebox(0,0)[lb]{\small{0}}}%
    \put(0.48320879,0.1670355){\color[rgb]{0,0,0}\makebox(0,0)[lb]{\small{$\ldots$}}}%
    \put(0.26207261,0.41663462){\color[rgb]{0,0,0}\makebox(0,0)[lb]{\small{$z_1$}}}%
    \put(0.69181647,0.44134599){\color[rgb]{0,0,0}\makebox(0,0)[lb]{\small{$z_{n-1}$}}}%
    \put(0.41031436,0.38555638){\color[rgb]{0,0,0}\makebox(0,0)[lb]{\small{$f$}}}%
    \put(0.48214841,0.41116253){\color[rgb]{0,0,0}\makebox(0,0)[lb]{\small{$\ldots$}}}%
    \put(0.56378393,0.38490117){\color[rgb]{0,0,0}\makebox(0,0)[lb]{\small{$f$}}}%
    \put(0.06245137,0.49488345){\color[rgb]{0,0,0}\makebox(0,0)[lb]{\small{$U_0$}}}%
    \put(0.69443471,0.4920694){\color[rgb]{0,0,0}\makebox(0,0)[lb]{\small{$U_{n-1}$}}}%
    \put(0.07071062,0.07046405){\color[rgb]{0,0,0}\makebox(0,0)[lb]{\small{$\D$}}}%
    \put(0.08864366,0.16759465){\color[rgb]{0,0,0}\makebox(0,0)[lb]{\small{0}}}%
    \put(0.92315163,0.06597279){\color[rgb]{0,0,0}\makebox(0,0)[lb]{\small{$\D$}}}%
    \put(0.94108466,0.16310339){\color[rgb]{0,0,0}\makebox(0,0)[lb]{\small{0}}}%
    \put(0.17797486,0.18928622){\color[rgb]{0,0,0}\makebox(0,0)[lb]{\small{$g_1$}}}%
    \put(0.57859519,0.19018447){\color[rgb]{0,0,0}\makebox(0,0)[lb]{\small{$g_{n-1}$}}}%
    \put(0.8139372,0.19198097){\color[rgb]{0,0,0}\makebox(0,0)[lb]{\small{$g_{n}$}}}%
    \put(0.89534017,0.42112588){\color[rgb]{0,0,0}\makebox(0,0)[lb]{\small{$z_{n}$}}}%
    \put(0.03571314,0.41393987){\color[rgb]{0,0,0}\makebox(0,0)[lb]{\small{$z_0$}}}%
    \put(0.09490254,0.28){\color[rgb]{0,0,0}\makebox(0,0)[lb]{\small{$\phi_0$}}}%
    \put(0.93117506,0.28){\color[rgb]{0,0,0}\makebox(0,0)[lb]{\small{$\phi_{n}$}}}%
    \put(0.80464693,0.38645463){\color[rgb]{0,0,0}\makebox(0,0)[lb]{\small{$f$}}}%
    \put(0.18664964,0.38465813){\color[rgb]{0,0,0}\makebox(0,0)[lb]{\small{$f$}}}%
    \put(0.28521784,0.49667995){\color[rgb]{0,0,0}\makebox(0,0)[lb]{\small{$U_1$}}}%
    \put(0.9127099,0.49835716){\color[rgb]{0,0,0}\makebox(0,0)[lb]{\small{$U_{n}$}}}%
  \end{picture}%
\endgroup%
 %%%%%%%%%%%%%%%%%%%55
\end{center}
\caption{\small Self maps of the unit disk arising from an orbit of wandering domains}
\end{figure}

Before stating the next theorem, we recall that if $f:U\ra V$ is a holomorphic map between two hyperbolic domains~$U$ and~$V$, then the hyperbolic distortion of~$f$ at $z$ is defined to be
\[
\|Df(z)\|_U^V:=\lim_{z'\ra z}\frac{\dist_V(f(z'),f(z))}{\dist_U(z',z)}.
\]

\begin{lem}\label{lemUD}
Let $U=U_0$ be a simply connected wandering domain of a transcendental entire function $f$ and let $U_n$ be the Fatou component containing $f^n(U)$, for $n \in \N$. Let $z_0\in U_0$ and let  $g_n,\ G_n$ be as defined above.
%Fix $z_0 \in U$ and, for $n \geq 0$, let $\phi_n:U_n \to \D$ denote the Riemann map with $\phi_n(f^n(z_0)) = 0$. For $n \in \N$, let $g_n:\D \to \D$ be defined by $g_n=\phi_n \circ f \circ \phi_{n-1}$ and let $G_n:\D \to \D$ be defined by $G_n = g_n \circ \cdots \circ g_1$.
\begin{itemize}
\item[\rm(a)] If $z \in U$ and $\phi_0(z) = w$, then
\[
\dist_{U_n}(f^n(z),f^n(z_0)) = \log \left( \frac{1 + |G_n(w)|}{1 - |G_n(w)|}\right), \quad \text{for } n \in \N.
\]

\item[\rm(b)] For each $n \in \N$,
\[
 |g_n'(0)| = \lambda_n(z_0):=\|Df(f^n(z_0))\|_{U_n}^{U_{n+1}}.
\]
\end{itemize}
\end{lem}

\begin{proof}

\begin{itemize}
\item[\rm(a)]
Let $n\in\N$. Since  $G_n = \phi_n \circ f^n \circ \phi_0^{-1}$  and $\phi_n$ is conformal,  if $z \in U$ and $\phi_0(z) = w$  then
\[
\dist_{U_n}(f^n(z),f^n(z_0)) = \dist_{\D}(G_n(w),G_n(0)) = \dist_{\D}(G_n(w),0) = \log \left( \frac{1 + |G_n(w)|}{1 - |G_n(w)|}\right),
\]
where the last equality follows from~\eqref{hdist0}.
\item[\rm(b)]
Let $n\in\N$. Since $g_n = \phi_n \circ f \circ \phi_{n-1}^{-1}$ and $\phi_n$ is conformal  we have
\[
\|Dg_n(0)\|_\D^\D=\|Df(f^n(z_0))\|_{U_n}^{U_{n+1}}=\lambda_n(z_0).
\]
Since $g_n(0) = 0$,  it follows from~\eqref{hdist0} that
\begin{eqnarray*}
\|Dg_n(0)\|_\D^\D & = & \lim_{w\to 0}\frac{\dist_\D(g_n(w), g_n(0))}{\dist_\D(w,0)} = \lim_{w\to 0}\frac{\dist_\D(g_n(w), 0)}{\dist_\D(w,0)}\\
& = & \lim_{w \to 0} \frac{\log \left( \frac{1 + |g_n(w)|}{1 - |g_n(w)|}\right)}{\log \left( \frac{1 + |w|}{1 - |w|}\right)}\\
& = & \lim_{w \to 0} \frac{2|g_n(w)|}{2|w|} = |g_n'(0)|,
\end{eqnarray*}
by using the Taylor expansion for the logarithm.\qedhere
\end{itemize}
\end{proof}

\subsection{Proof of Theorem A}
\label{sec:proofofthmA}

We now use the results of Section 2 together with Lemma~\ref{lemUD} to prove Theorem A, that is the  classification of simply connected wandering domains according to the  behaviour of the hyperbolic distances between orbits of points.

Let $U=U_0$ be a simply connected wandering domain of a transcendental entire function $f$ and let $U_n$ be the Fatou component containing $f^n(U)$, for $n \in \N$. Also, let
\[
E=\{(z,z')\in U\times U : f^k(z)=f^k(z') \text{\ for some $k\in\N$}\}.
\]

Let $z_0\in U_0$ and let $\phi_n, g_n,G_n$ be as defined in the beginning of this section.

First we suppose that there exists $z_0' \in U_0$ with
\begin{equation}\label{eqcont}
(z_0',z_0) \notin E \quad \text{and } \dist_{U_n}(f^n(z_0'), f^n(z_0))\lran 0.
\end{equation}

Let  $w_0' = \phi_0(z_0')$. By \eqref{eqcont} and Lemma~\ref{lemUD}~(a), we have that $G_n(w_0') \lran 0$ and $G_n(w_0') \neq 0$  for all $n\in\N$. Hence
by Theorem~\ref{Thm zero}~(b) we have that $\sum_{n=1}^{\infty}(1-\lambda_n) = \infty$, where $\lambda_n = |g_n'(0)|$, and therefore that  $G_n(w)\lran 0$, for all $w \in \D$, by Theorem~\ref{Thm zero}~(a). By Lemma~\ref{lemUD}~(a) again, $\dist_{U_n}(f^n(z), f^n(z_0))\lran 0$, for all $z \in U_0$. We conclude that $\dist_{U_n}(f^n(z), f^n(z'))\lran 0$, for all $z,z' \in U_0$, by the triangle inequality, which is case~(1).

We have shown that~\eqref{eqcont} implies that $\sum_{n=1}^{\infty}(1-\lambda_n) = \infty$ and that this implies that $U_0$ is contracting. Thus $U_0$ is contracting if and only if $\sum_{n=1}^{\infty}(1-\lambda_n) = \infty$, where
\[
\lambda_n = |g_n'(0)| = \|Df(f^n(z_0))\|_{U_n}^{U_{n+1}}=\lambda_n(z_0),\;\text{ for }n\in\N,
\]
by Lemma~\ref{lemUD}~(b).

Now suppose that there exist $z,z' \in U_0$ and $N \in \N$ with
\begin{equation}\label{eqiso}
\dist_{U_n} (f^n(z),f^n(z')) = c(z,z') > 0, \quad \text{ for all } n \geq N.
\end{equation}
Then, by Schwarz's Lemma, $f:U_n \to U_{n+1}$ is an isometry, for all $n \geq N$,  and so for every pair $z,z'\in U_0$ we have that
\[
\dist_{U_n} (f^n(z),f^n(z')) = \dist_{U_N} (f^N(z),f^N(z')), \quad \text{for } n \ge N.
\]
Thus, if $(z,z') \in (U \times U) \setminus E$ we have that   $\dist_{U_n} (f^n(z),f^n(z')) = c(z,z') > 0$ for all $n \geq N$ and that  $U_0$ is eventually isometric. In this case, $\lambda_n(z) = 1$ for all $z\in \D$ and for $n \geq N$, by Schwarz's Lemma, which is case~(3).

Finally, we show that case~(2) is the only other possibility. It follows from the above proof that, if there exists $z_0' \in U_0$ for which neither~\eqref{eqcont} nor~\eqref{eqiso} holds, then the only possibility is that neither of these conditions hold for {\it any} $z \in U_0$; that is, $U_0$ is semi-contracting, which is case~(2). This completes the proof of Theorem A.

%%%%%%%%%%%%%%%%%%%%%%%%%%%%%%%%%%%%%%%%%%

\subsection{Subclassification of contracting wandering domains: Proof of Theorem~B}\label{sec:subclassification}
In this subsection we prove Theorem B, which gives sufficient conditions for a simply connected wandering domain to be strongly contracting or super-contracting. We prove parts~(a) and~(b) by using the results of Section~2 together with Lemma~\ref{lemUD}.

Let $U_0$ be a simply connected wandering domain of a transcendental entire function $f$ and let $U_n$ denote the Fatou component containing $f^n(U)$, for $n \in \N$. Let $z_0\in U_0$ and let $g_n,G_n$ be as defined in the beginning of this section.

Also, for $n \in \N$, we let $\lambda_n = \|Df(f^n(z_0))\|_{U_n}^{U_{n+1}} $ and note from Lemma~\ref{lemUD}~(b) that $\lambda_n = |g_n'(0)|$.

To show (a), observe that if $\lim \sup_{n \to \infty} \frac{1}{n}\sum_{k=1}^n \lambda_k = a < 1$, then it follows from Theorem~\ref{avcont} (a) that
\[
|G_n(w)| = O(c^n) \;\text{ as }n\to\infty, \quad \text{for } w \in \D, \; c \in (a,1).
\]
So, by Lemma~\ref{lemUD} (a), if we take $z \in U_0$ and put $w = \phi_0(z)$, then
\[
\dist_{U_n}(f^n(z), f^n(z_0)) = O(c^n)\;\text{ as }n\to\infty, \quad \text{for } c\in (a,1).
\]
This proves part~(a) of Theorem B.

To prove part~(b), we note that, if $\lim_{n \to \infty} \frac{1}{n}\sum_{k=1}^n \lambda_k = 0$, then, from Theorem~\ref{avcont}~(b),
\[
\left(\dist_{U_n}(f^n(z), f^n(z_0))\right)^{1/n} \to 0 \quad \text{as $n\to\infty$},
\]
and hence $U_0$ is super-contracting.

To prove part~(c) we need to show that, for $n \in \N$, $z \in U_0$,
\begin{equation}\label{equal}
\lim \sup_{n \to \infty} \frac{1}{n}\sum_{k=1}^n \lambda_k(z) = \lim \sup_{n \to \infty} \frac{1}{n}\sum_{k=1}^n \lambda_k, \quad \text{for } z \in U_0.
\end{equation}
(Recall that $\lambda_k = \lambda_k(z_0)$.) We begin by supposing that $\lim \sup_{n \to \infty} \frac{1}{n}\sum_{k=1}^n \lambda_k = a < 1$ and fix $c \in (a,1)$ and $z \in U_0$. From part (a) above, there exists $C > 0$ such that
\begin{equation}\label{cbound}
\dist_{U_k}(f^k(z), f^k(z_0)) \leq C c^k, \quad \text{for } k \in \N.
\end{equation}
We now use the following result of Beardon and Minda to obtain a bound on the difference between  $\lambda_k(z)$ and $\lambda_k$.

\begin{lem}[{\cite[Theorem 11.2]{BeardonMinda}\label{lBM}}]
Let $U,V$ be hyperbolic domains and let $f:U\to V$ be holomorphic. Then
\[
\dist_{\D}(\|Df(z)\|_U^V, \|Df(w)\|_U^V) \leq 2 \dist_{U}(z,w), \quad \text{for all } z,w \in U.
\]
\end{lem}

It follows from Lemma~\ref{lBM} together with~\eqref{cbound} that, under our supposition,
\[
\dist_{\D}(\lambda_k(z), \lambda_k) \leq 2C c^k, \quad \text{for } k \in \N.
\]
%So, as in the proof of Lemma~\ref{lemUD}, we see that an upper bound for $\lambda_k(z)$ is given by
Since
\[
\dist_{\D}(\lambda_k(z), \lambda_k) = \left|\int_{\lambda_k}^{\lambda_k(z)} \frac{2\,dt}{1-t^2}\right| \geq \left|\int_{\lambda_k}^{\lambda_k(z)} 2\,dt\right| = 2|\lambda_k(z) - \lambda_k|,
\]
it follows that
\[
|\lambda_k(z) - \lambda_k| \leq C c^k, \quad \text{for } k \in \N.
\]

So, if $\lim \sup_{n \to \infty} \frac{1}{n}\sum_{k=1}^n \lambda_k = a < 1$, then
\begin{eqnarray*}
\lim \sup_{n \to \infty} \frac{1}{n}\sum_{k=1}^n \lambda_k(z)& \leq & \lim \sup_{n \to \infty} \frac{1}{n}\sum_{k=1}^n \lambda_k + \lim \sup_{n \to \infty} \frac{C}{n}\sum_{k=1}^n c^k\\
& = & \lim \sup_{n \to \infty} \frac{1}{n}\sum_{k=1}^n \lambda_k = a.
\end{eqnarray*}
Since the roles of $z_0$ and $z$ are interchangeable, we have shown that~\eqref{equal} holds whenever $\lim \sup_{n \to \infty} \frac{1}{n}\sum_{k=1}^n \lambda_k < 1$. The only remaining case is that
\[
\lim \sup_{n \to \infty} \frac{1}{n}\sum_{k=1}^n \lambda_k(z) = 1 = \lim \sup_{n \to \infty} \frac{1}{n}\sum_{k=1}^n \lambda_k.
\]
This completes the proof of Theorem B.

\subsection{Rate of contraction in parabolic components}\label{parabolic}

It is clear that if a wandering domain $U$ is the lift of an attracting component $V$, then $U$ is strongly contracting and, if $V$ is super-attracting, then $U$ is super-contracting. We end this section by showing that if a wandering domain $U$ occurs as a lift of a parabolic component, then $U$ is contracting but not strongly contracting. We need the following lemma; see \cite[Lemma p. 157]{Shapiro}, for example.
\begin{lem}
\label{lem:shapiro}
If $G$ is a simply connected domain, not the whole complex plane, then for $z,w \in G$,
$$\operatorname{dist}_G(z,w)\geq \frac{1}{2}\log \left(1+ \frac{|z-w
|}{\min\{\operatorname{dist}(z,\partial G), \operatorname{dist}(w,\partial G)\}}\right).$$
\end{lem}
We have the following general result about the contraction rate in a parabolic component. The estimates in this result may be known but we are not aware of a reference.

\begin{thm}\label{parabolic weakly contracting}
Let~$V$ be an invariant parabolic component of a transcendental entire function~$f$. Then, for all $z_0, z_0' \in V$, either $f^m(z_0)=f^m(z'_0)$ for some $m\in\N$ or there exist positive constants~$k$ and~$K$ depending on $z_0$, $z'_0$ and~$p$, the number of petals, such that
\begin{equation}\label{kandK}
\frac{k}{n} \le \operatorname{dist}_V(f^n(z_0),f^n(z'_0))\le \frac{K}{n},\;\text{ for }n\in \N.
\end{equation}
\end{thm}
\begin{proof}
Without loss of generality we assume that $0$ is the parabolic fixed point in $\partial V$ and let~$p$ be the number of petals of $f$ at 0. The following proof uses detailed estimates from the discussion of Abel's functional equation in \cite[pages~110--122]{Beardon} and we start by summarising this discussion, mainly using the notation from \cite{Beardon}.

First, the function~$f$ is conformally conjugate near 0 to a function of the form
\[
F(z)=z-z^{p+1}+O(z^{2p+1}) \quad\text{as }z \to 0.
\]
Substituting $w=z^{-p}$, $z=w^{-1/p}$, where  $w^{-1/p}$ denotes the principal root, we obtain
\[
g(w)=1/\left(F(w^{-1/p})\right)^p = w+p+A/w+O(1/w^{1+1/p}) \quad\text{as }w \to \infty,
\]
for some constant $A$, from which it follows that there exists a parabola-shaped domain of the form  $\Pi=\{u+iv:v^2>4K(K-u)\}$, $K>0$, that is forward invariant under~$g$. For $w\in \Pi$, we have
\begin{equation}\label{un}
g^n(w) = np+\frac{A}{p}\log n+u_n(w), \quad\text{for } n\in \N,
\end{equation}
where the functions $u_n$ are holomorphic in $\Pi$ and converge locally uniformly on $\Pi$ to a univalent function $u$, which satisfies  $u(g(w))=u(w)+p$, a form of Abel's functional equation.

Now suppose that $w_0, w'_0\in \Pi$ are distinct points and put $z_0 = w_0^{-1/p}$ and $z'_0 = (w'_0)^{-1/p}$. Then $z_0$ and $z'_0$ lie in the invariant petal-shaped domain for $F$ at~0, which corresponds to $\Pi$ under the mapping $w\mapsto w^{-1/p}$ and which is symmetric with respect to the positive real axis, subtending an angle of $2\pi/p$ at~0. It follows from the above properties of the functions $u_n$ and the univalence of $u$ that
%\begin{equation}\label{gnw0}
%g^n(w_0)=np+\frac{A}{p}\log n+ O(1) \quad\text{as } n\to \infty,
%\end{equation}
%\begin{equation}\label{gnw'0}
%g^n(w'_0)=np+\frac{A}{p}\log n+ O(1) \quad\text{as } n\to \infty,
%\end{equation}
%and
\begin{equation}\label{uw0uw'0}
g^n(w'_0)-g^n(w_0) = u_n(w'_0)-u_n(w_0) \to u(w'_0)-u(w_0) \ne 0 \quad\text{as } n\to \infty.
\end{equation}
On substituting $z=w^{-1/p}$ we find that $z_n=F^n(z_0)$ and $z'_n= F^n(z'_0)$ both approach~0 tangentially to the positive real axis through the petal-shaped domain mentioned above. Moreover, by \eqref{un},
\begin{equation}\label{eq:parabolic1}
|z_n|=|g^n(w_0)|^{-1/p}=\frac{1}{|np+\frac Ap\log n+u_n(w_0)|^{1/p}}\sim \frac{1}{(np)^{1/p}} \;\text{ as}\;n\to \infty,
\end{equation}
and
\begin{equation}\label{eq:parabolic2}
|z'_n|=|g^n(w'_0)|^{-1/p}=\frac{1}{|np+\frac Ap\log n+u_n(w'_0)|^{1/p}}\sim \frac{1}{(np)^{1/p}} \;\text{ as}\;n\to \infty.
\end{equation}
Also, by \eqref{un}, \eqref{uw0uw'0} and a short calculation,
\begin{equation}
\label{eq:parabolic3}
|z_n-z_n'|=|g^n(w_0)^{-1/p}-g^n(w'_0)^{-1/p}|\sim \frac{|u(w_0)-u(w'_0)|}{p(np)^{1+1/p}}\quad\text{as}\;n\to \infty.
\end{equation}
Since $F$ is conformally conjugate to $f$ near 0, we deduce that these estimates for $z_n$ and $z'_n$ hold if $z_n=f^n(z_0)$ and $z'_n=f^n(z'_0)$, for $n \in \N$, where $z_0$ and $z'_0$ are redefined to be the corresponding points in the invariant parabolic component $V$ for $f$. Also, note that $V$ is one of~$p$ distinct invariant parabolic components for $f$ at 0, each containing an invariant petal-shaped domain subtending an angle of $2\pi/p$ at~0.

Therefore, by Lemma \ref{lem:shapiro}, (\ref{eq:parabolic1}), (\ref{eq:parabolic2}) and (\ref{eq:parabolic3}), together with the fact that $\operatorname{dist}(z_n,\partial V)\le |z_n|$, we have
\begin{eqnarray}
\operatorname{dist}_{V}(z_n,z_n')&\geq& \frac{1}{2} \log \left(1+ \frac{|z_n-z_n'|}{\min\{\operatorname{dist}(z_n,\partial V), \operatorname{dist}(z_n',\partial V)\}}\right) \nonumber\\
&\geq& k \frac{1/n^{1+1/p}}{1/n^{1/p}}=\frac{k}{n},\;\text{ for } n\in \N,\nonumber
\end{eqnarray}
for some positive constant~$k$ depending on $z_0$, $z'_0$ and $p$.

Finally, for $n \in \N$, let $\gamma_n$ denote the line segment joining $z_n$ to $z_n'$. Then, in view of the fact that $z_n$ and $z'_n$ approach 0 tangentially to the positive real axis, the line segment $\gamma_n$, for~$n$ sufficiently large, lies in the invariant petal-shaped domain in $V$, which is conformally equivalent near 0 to the petal-shaped domain obtained from $\Pi$. Also, for~$n$ sufficiently large, we have
\[
\min\{\operatorname{dist}(z_n,\partial V), \operatorname{dist}(z_n',\partial V)\} \ge \frac12 \sin (\pi/p) |z_n|.
\]
Therefore, by the standard hyperbolic density estimate in a simply connected domain (see,  for example,~\cite[page 13]{Carleson-Gamelin}), (\ref{eq:parabolic1}), (\ref{eq:parabolic2}), and the triangle inequality, we have
\begin{eqnarray}
\operatorname{dist}_{V}(z_n,z_n')
&\leq& \int_{\gamma_n} \rho_{V}(z)\,|dz|  \nonumber\\
&\leq& \int_{\gamma_n} \frac{2}{\operatorname{dist}(z,\partial V)}\,|dz|  \nonumber\\
&\leq& \frac{2|z_n-z_n'|}{\min\{\operatorname{dist}(z_n,\partial V), \operatorname{dist}(z_n',\partial V)\}-\tfrac12|z_n-z'_n|}\nonumber\\
&\leq& K \frac{1/n^{1+1/p}}{1/n^{1/p}}=\frac{K}{n},\nonumber
\end{eqnarray}
for some positive constant $K$ depending on $z_0$, $z'_0$ and $p$, and $n$ sufficiently large.

Finally, we note that, for all pairs of points $z_0, z'_0 \in V$ with disjoint orbits, we have $f^n(z_0), f^n(z'_0)\in V$ for $n$ sufficiently large. This completes the proof.
\end{proof}
\begin{rem*} \; Using a more careful analysis of the size of the hyperbolic  density in~$V$ near the points $z_n$ and $z'_n$ we can show that the estimate \eqref{kandK} in Theorem~\ref{parabolic weakly contracting} can be replaced by
\[
 \operatorname{dist}_V(f^n(z_0),f^n(z'_0))\sim \frac{c}{n}\;\text{ as }n\to \infty,
 \]
for some positive constant~$c$ depending on $z_0$, $z'_0$ and~$p$. The proof uses results about the behaviour of  a Riemann map from a sector of angle $2\pi/p$ onto~$V$ which maps~0 to~0, justified by using standard results about angular derivative of conformal mappings at boundary points.
\end{rem*}

By conformality and Definition~\ref{def:strongly contracting}, we have the following corollary of Theorem~\ref{parabolic weakly contracting}.

\begin{cor}\label{weak}
If $U$ is a simply connected wandering domain that is the lift of an invariant parabolic component $V$ and $U_n$ is the Fatou component containing $f^n(U)$, $n\in \N$. Then either $f^m(z_0)=f^m(z'_0)$ for some $m\in\N$ or there exists positive constants~$k$ and~$K$ depending on $z_0$ and $z'_0$ such that
\[
\frac{k}{n} \leq \operatorname{dist}_{U_n}(z_n,z_n')\leq \frac{K}{n}, \quad \text{for } n \in \N.
\]
In particular,~$U$ is contracting but not strongly contracting.
\end{cor}

Here are two examples of simply connected wandering domains, obtained by lifting parabolic components, which are contracting but not strongly contracting.

\begin{ex}
 Consider the entire functions
\[
f(z)=z+e^{-z}+2\pi i, \quad g(z)=z+e^{-z} \quad \text{and}\quad F(w)=we^{-w}.
\]
Then both~$f$ and~$g$ are obtained by lifting~$F$ under the exponential function $w=e^{-z}$. Since  $F$ has an invariant parabolic component associated with the fixed point at 0, the function~$g$ has congruent unbounded invariant Baker domains $U_n$, $n\in \Z$, such that $U_n\subset \{z:(2n-1)\pi<\Im(z)<(2n+1)\pi\}$; see~\cite{Dom98,FagHen}. Since $J(f)=J(g)$, by~\cite{Berg95}, the components $U_n$ form a sequence of simply connected wandering domains which, by Corollary~\ref{weak}, are contracting but not strongly contracting.
\end{ex}

\begin{figure}[hbt!]
\fboxsep=0.5pt
\begin{center}
\framebox{\includegraphics[width=0.45\textwidth]{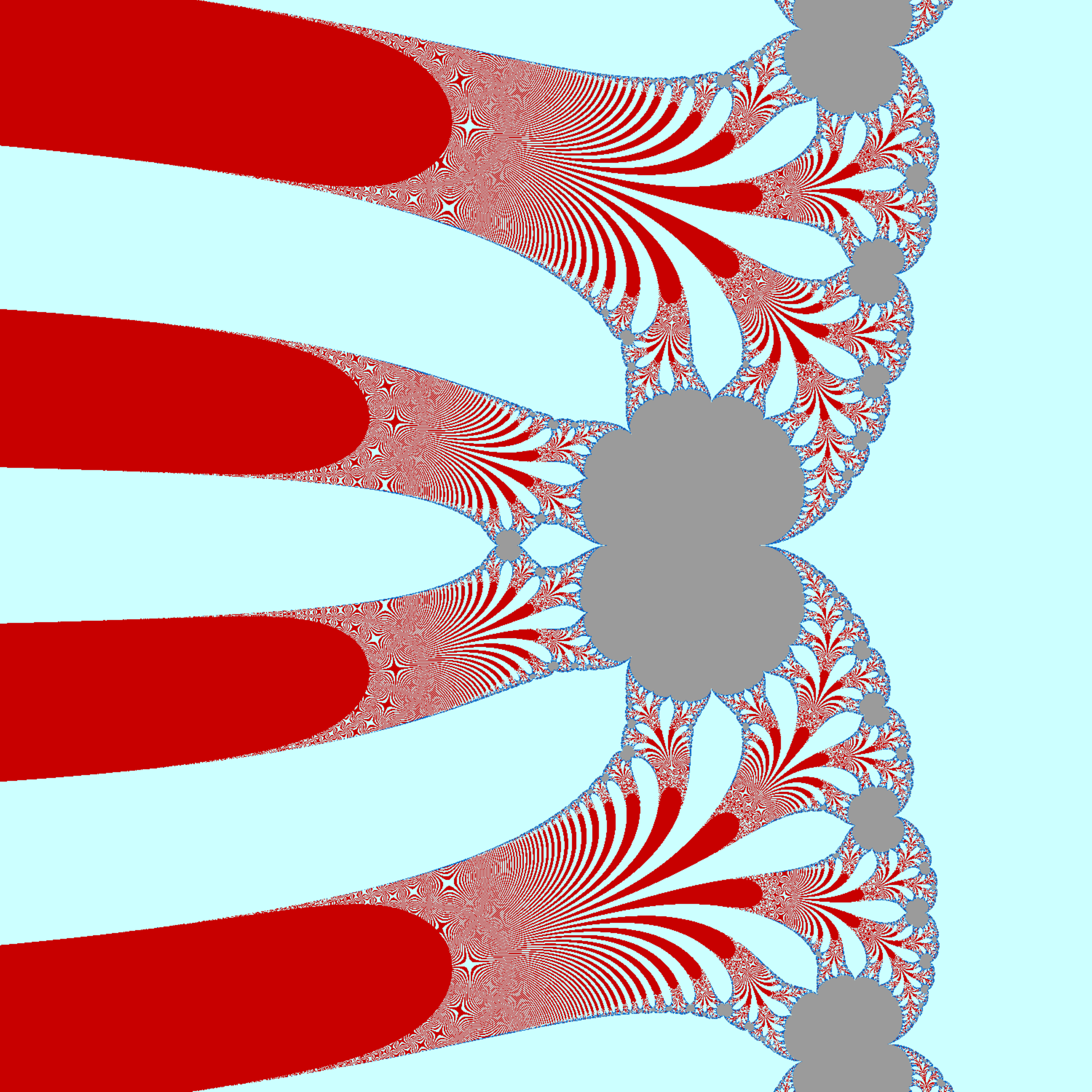}}
\framebox{\includegraphics[width=0.45\textwidth]{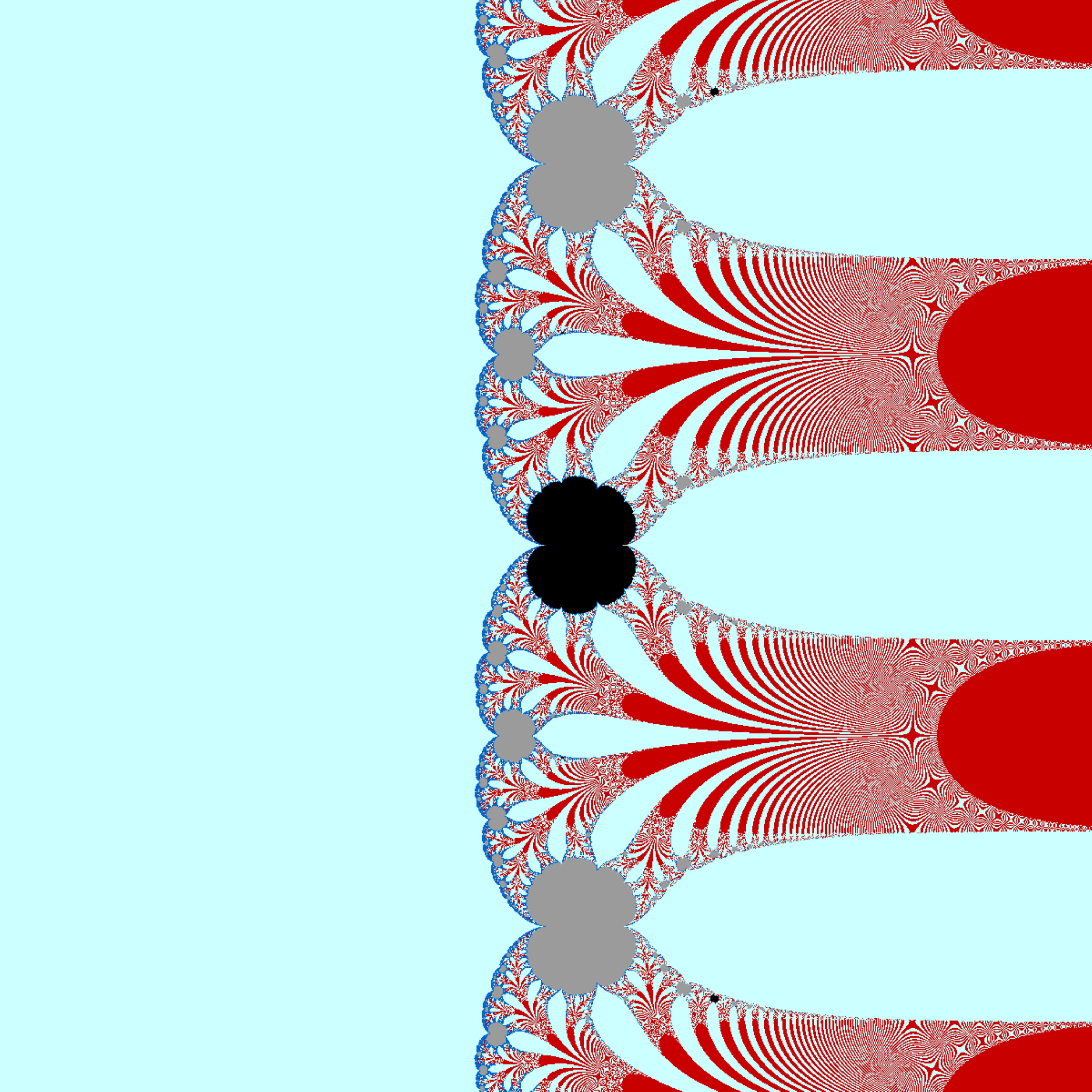}}
\end{center}
\caption{\label{figparab} \small Left:  Dynamical plane of $F$ in Example \ref{exparab}. The  super-attracting basin of  $w=0$ is shown in light blue, while in gray we see the parabolic basin of $w=1$. Right: Dynamical plane of $f$. In blue the Baker domain (lift of the superattracting basin). In black the parabolic invariant basin at $z=0$. In gray the wandering domains.  The range is $[-9,9]\times[-9,9]$.}
\end{figure}

\begin{ex}\label{exparab}
As another example, consider
\[
f(z)=2z + 1 - e^{z} \quad \text{and}\quad F(w)=e\,w^2e^{-w},
\]
which belongs to the same family as the Example in Figure 1,  both closely related to an example of Bergweiler \cite{bergweilerexample}. In this case, $f$ is a lift of~$F$ under $w=e^{z}$, and $F$ has an invariant parabolic component associated with the fixed point at~1 which lifts to congruent, bounded, simply connected Fatou components, $V_n$, $n\in \Z$, say, of~$f$ such that $0\in \partial V_0$ and
\[
V_n=V_0+2n\pi i,\;\text{ for } n\in \Z,\quad \text{and}\quad f(V_n)=V_{2n},\;\text{ for } n\in \Z.
\]
From this it follows that $V_{2n}$,  $n\ge 1$, is a sequence of bounded, escaping, simply connected wandering domains which, by Corollary~\ref{weak}, are contracting but not strongly contracting.
\end{ex}

\section{Convergence to the boundary: Proof of Theorem C}
\label{sec:convergence}

In this section we give the proof of Theorem C, the classification of simply connected wandering domains in terms of whether orbits of points converge to the boundary. Recall that the Euclidean distance of a point $z$ from the boundary of a hyperbolic domain $U$ is closely related to the hyperbolic density $\rho_U(z)$ of the point in the domain. Indeed, if $U_0$ is a simply connected wandering domain of a transcendental entire function $f$ and $U_n$ is the Fatou component containing $f^n(U_0)$, and $z \in  U_0$, then
\[
\deucl(f^n(z), \partial U_n) \to 0 \iff \rho_{U_n}(f^n(z)) \to \infty.
\]
We prove Theorem C by considering the hyperbolic densities $\rho_{U_n}(f^n(z))$. In fact, we show that a trichotomy as in Theorem C occurs if we consider the quantities $a_n\rho_{U_n}(f^n(z))$, for {\it any} sequence $a_n$ and not just for $a_n = 1$. As we mentioned in the introduction, the issue of convergence to the boundary is somehow delicate in that it is tightly connected to the shape of the wandering domains, and there may be situations where it is more appropriate to use an alternative definition  involving different sequences $a_n$. For example, if the domains $U_n$ are shrinking then it may make sense to say that $z_n$ converges to the boundary if $a_n \rho_{U_n}(f^n(z)) \to \infty$ as $n \to \infty$ where
$$a_n=\sup_D \{\operatorname{diam}D :  D\; \text{is a disk contained in}\;U_n\}.$$

In order to prove Theorem C we need the following lemma, which can be thought of as a Harnack inequality for hyperbolic density in a simply connected domain; see \cite[Lemma~6.2]{BC08} for a similar  type of result (with a different proof) for hyperbolic density in the unit disk.
\begin{lem}[Estimate of hyperbolic quantities]
\label{lem:hyp density}
Let $U \subset \mathbb{C}$ be a simply connected domain. Then, for all $z,z'\in U$,
%\[
%\dist_U(z,z') \leq C \quad \Longrightarrow \quad \rho_U(z') \geq e^{-2C}  \rho_U(z).
%\]
\[
\exp(-2\dist_U(z,z'))\le \frac{\rho_U(z')}{\rho_U(z)} \le \exp(2\dist_U(z,z')).
\]
\end{lem}
%In other words, since $f: U_n\ra U_{n+1}$ decreases hyperbolic  distances, for any $z,w\in U_0$, there exists a constant $c=\exp(-2\dist_{U_0}(z,w))$ independent of $n$ such that
%$$
%\rho_{U_n}(f^n(w))\geq c \, \rho_{U_{n}}(f^n(z)).
%$$
 \begin{proof}
Let $z,z' \in U$ and let $\phi:\D\ra U$ be a Riemann map with $\phi(0)=z$ and $\phi(r)=z'$, for some $r \in [0,1)$. By conformal invariance of the hyperbolic metric, together with~\eqref{hdist0},
$$
\dist_U(z,z')=\dist_{\D}(0,r)=\log\frac{1+r}{1-r},
$$
and, by the definition of the hyperbolic density on~$U$,
\begin{align*}
\rho_U(z)  &= {\rho_\D(0)}/{|\phi'(0)|}= {2}/{|\phi'(0)|}\\
\rho_U(z')  &=\rho_\D(r)/|\phi'(r)|=\frac{2}{1-r^2}\left/ |\phi'(r)| \right..
\end{align*}

Also, by a standard distortion theorem for conformal maps \cite{Pommerenke},
$$
\frac{1-r}{(1+r)^3}\leq\frac{|\phi'(r)|}{|\phi'(0)|}\leq \frac{1+r}{(1-r)^3}.
$$
Putting everything together we obtain the lower bound,
\[
\frac{\rho_U(z')}{\rho_U(z)}=\frac{1}{1-r^2}\frac{|\phi'(0)|}{|\phi'(r)|}\geq \frac{(1-r)^2}{(1+r)^2} = \exp(-2\dist_U(z,z')),
\]
and the upper bound follows by symmetry.
\end{proof}
{\it Remark} \; It is easy to check that the inequalities in Lemma~\ref{lem:hyp density} are sharp in the case when~$U$ is $\mathbb{C}\setminus (-\infty,0]$ and the points $z,z'$ lie on the positive real axis.

We now prove the main result of this section.

\begin{thm}\label{thm:everybody converges to the boundary or nobody}
Let $U_0$ be a simply connected wandering domain of a transcendental entire function $f$, let $U_n$ be the Fatou component containing $f^n(U)$, for $n \in \N$, and let $(a_n)$ be a real positive sequence.
\begin{itemize}
\item[\rm(a)] If there is a subsequence $n_k\to \infty$ and a point $z\in U_0$ such that $a_{n_k} \rho_{U_{n_k}}(f^{n_k}(z))\ra\infty$, then the same is true for all other points in $U_0$.

\item[\rm(b)]If there is a subsequence $m_k\to \infty$ and a point $z\in U_0$ such that $a_{m_k} \rho_{U_{m_k}}(f^{m_k}(z))$ is bounded, then the same is true for all other points in $U_0$.
\end{itemize}
 \end{thm}
\begin{proof}
\begin{itemize}
\item[\rm(a)]
Suppose that $a_{n_k} \rho_{U_{n_k}}(f^{n_k}(z))\ra\infty$ as $k \to \infty$ and let $z' \in U_0$ with  $z'\neq z$. By the contraction property of the hyperbolic metric, we have that
$$
\dist_{U_n}(f^{n}(z),f^{n}(z')) \leq\dist_{U_0}( z,z')=:C, \quad \text{ for } n \in \N.
$$
By  Lemma \ref{lem:hyp density},   $\rho_{U_n}(f^n(z')) \geq e^{-2C} \rho_{U_n}(f^n(z))$, for $n \in \N$.
Hence
\[
a_{n_k} \rho_{U_{n_k}}(f^{n_k}(z'))\geq e^{-2C}  a_{n_k} \rho_{U_{n_k}}(f^{n_k}(z))\ra \infty \quad \text{as } k \to \infty
\]

\item[\rm(b)] Now suppose that $a_{m_k} \rho_{U_{m_k}}(f^{m_k}(z)) \leq M$, for $k \in \N$, and let $z' \in U_0$ with  $z'\neq z$. Again, by the contraction property of the hyperbolic metric, we have that
\[
\dist_{U_{m_k}}(f^{m_k}(z),f^{m_k}(z')) \leq\dist_{U_0}( z,z')=:C.
\]
Now applying Lemma   \ref{lem:hyp density} and interchanging $z$ and $z'$, we obtain that
\[
\rho_{U_{m_k}} (f^{m_k}(z')) \leq e^{2C} \rho_{U_{m_k}}(f^{m_k}(z)),
\]
which implies that
\[
a_{m_k} \rho_{U_{m_k}}(f^{m_k}(z'))\leq e^{2C}  a_{m_k} \rho_{U_{m_k}}(f^{m_k}(z)) \leq M e^{2C},
\]
so $a_{m_k} \rho_{U_{m_k}}(f^{m_k}(z'))$ is bounded, for $k \in \N$.
\end{itemize}
\end{proof}

The result of Theorem~C follows from Theorem~\ref{thm:everybody converges to the boundary or nobody}, by taking $a_n = 1$, for $n \in \N$.

\section{Constructing wandering domains}\label{sect:examples}

%%%%%%%%%%%%%%%%%%%%%%%%%%%
We begin this section with the proof of Theorem D, which  we then use  together with an extension of Runge's Approximation Theorem to prove  Theorem~\ref{thm:main construction}. This result  enables us to construct bounded simply connected wandering domains in which various different dynamical behaviours can be specified and is the main tool that we use to construct examples in Section~\ref{sect:proof of examples}.

%%%%%%%%%%%%%%%%%%%%%%%%%%%
\subsection{Proof of Theorem D}

Let $f$ be a transcendental entire function and let  $\gamma_n$, $\Gamma_n$, $n_k$, $L_k$,   and  $D$ be as in  Theorem~D; see Figure \ref{fig:thmD}.
It follows from properties~(a) and~(b) of Theorem D that
for each $n,m\in\N$ with $n\neq m$ the curve $\gamma_n$ is in $\operatorname{ext}\gamma_m$
%the curves $\gamma_n$ are all exterior to each other
and so, by property (c) and Montel's theorem, there exist Fatou components $U_n$ such that
\begin{equation}
\label{eq:intgamma}
\overline{\operatorname{int}\gamma_n} \subset U_n, \;\text{for}\;n \geq 0.
\end{equation}
 Notice that, a priori, the components $U_n$ need not be different from each other. One of our  goals is to show that they are indeed different, by proving that $U_n \subset \operatorname{int}\Gamma_n$, for $n\ge 0$.
\begin{figure}[hbt!]
\begin{tikzpicture}[scale=2.5] %Changes the scale

\draw [blue] plot [smooth cycle] coordinates {(-1.6,0) (-1.4,0.55) (-0.8,0.32)(-0.6,0.35) (-0.4,-0.1) (-0.8,-0.45)(-0.9,-0.35)(-1.2,-0.4)};
\draw [fill,gray!35] plot [smooth cycle] coordinates {(-1.2,0) (-1.3,0.2) (-0.9,0.25)(-0.65,0.2) (-0.7,-0.3)(-0.85,-0.15) (-1.1,-0.2)};
\draw [blue] plot [smooth cycle] coordinates {(-1.2,0) (-1.3,0.2) (-0.9,0.25)(-0.65,0.2) (-0.7,-0.3)(-0.85,-0.15) (-1.1,-0.2)};
\draw [blue] plot [smooth cycle] coordinates {(-3.2,0) (-3,0.6) (-2.3,0.4)(-2,0.45) (-2.1,-0.1) (-2.8,-0.65)};
\draw [fill,gray!35] plot [smooth cycle] coordinates {(-2.9,0) (-2.7,0.2) (-2.5,0.3)(-2.4,0) (-2.5,-0.25) (-2.7,-0.1)};
\draw [blue] plot [smooth cycle] coordinates {(-2.9,0) (-2.7,0.2) (-2.5,0.3)(-2.4,0) (-2.5,-0.25) (-2.7,-0.1)};
\draw [blue] plot [smooth cycle] coordinates {(0.7,0) (0.9,0.35) (1,0.45)(1.3,0.25) (1.5,-0.2) (1.2,-0.45)(1,-0.37)};
\draw [fill,gray!35] plot [smooth cycle] coordinates {(1,0) (1.05,0.3) (1.25,0.15)(1.2,-0.25)(1.1,-0.15) (0.9,-0.2)};
\draw [blue] plot [smooth cycle] coordinates {(1,0) (1.05,0.3) (1.25,0.15)(1.2,-0.25)(1.1,-0.15) (0.9,-0.2)};
\draw [blue] plot [smooth cycle] coordinates {(2,0) (2.2,0.35) (2.3,0.45)(2.7,0.3) (3,-0.25) (2.5,-0.4)(2.2,-0.37)(2.4,-0.3)};
\draw [fill,gray!35] plot [smooth cycle] coordinates {(2.2,0) (2.3,0.2) (2.35,0.25)(2.6,0.1)(2.7,0.2)(2.6,-0.1) (2.4,-0.25)(2.3,0.1)};
\draw [blue] plot [smooth cycle] coordinates {(2.2,0) (2.3,0.2) (2.35,0.25)(2.6,0.1)(2.7,0.2)(2.6,-0.1) (2.4,-0.25)(2.3,0.1)};
\draw [fill,gray!50] plot [smooth cycle] coordinates {(-3,-1.1) (-2.3,-1) (-2,-1.4)(-1.8,-1.65)(-2.5,-1.9)(-2.7,-1.5)};
\draw  plot [smooth cycle] coordinates {(-3,-1.1) (-2.3,-1) (-2,-1.4)(-1.8,-1.65)(-2.5,-1.9)(-2.7,-1.5)};
\draw [red] plot [smooth] coordinates {(-1.7,0.15) (-1.5,0.65) (-1,0.7)(-0.8,0.5)(-0.6,0.6) (-0.3,0.1)(-0.25,-0.2)(-0.6,-0.45)(-0.8,-0.55)(-1.4,-0.7) (-1.7,-0.15)(-1.65,-0.05)};
\draw [red] plot [smooth] coordinates {(0.7,0.15) (0.9,0.55) (1,0.5) (1.3,0.45)(1.6,0.1)(1.4,-0.6)(1.1,-0.5) (0.8,-0.5)(0.6,-0.05)};
\draw plot [smooth] coordinates {(-2.6,-0.05) (-2.4,0.2)(-2.3,0.1)(-2.2,0.2) (-2.1,-0.1)};
\draw plot [smooth] coordinates {(-1,-0.05) (-0.9,-0.1)(-0.8,0.1)(-0.75,-0.1)(-0.9,-0.35)};
\draw plot [smooth] coordinates {(1.2,-0.05) (1.1,-0.1)(1,-0.3)(1.1,-0.35)(1.3,-0.25)(1.2,-0.45)};
\draw plot [smooth] coordinates {(2.6,0.05) (2.4,0.1)(2.5,0.25)(2.7,0.3)};
\draw [->] (-2,0) to [out=40,in=150] (-1.7,0);
\draw [->] (-0.35,0) to [out=40,in=150] (-0.1,0);
\draw [->] (0.3,0) to [out=40,in=150] (0.6,0);
\draw [->] (1.5,0) to [out=36,in=154] (1.92,0);
\node at (-1.87,0.15) {$f$};
\node at (-0.22,0.15) {$f$};
\node at (0.45,0.15) {$f$};
\node at (1.72,0.18) {$f$};
\node at (-2.3,-0.05) {$C_0$};
\node at (-0.95,-0.25) {$C_1$};
\node at (1.35,-0.15) {$C_{n_k}$};
\node at (2.4,0.35) {$C_{n_k+1}$};
\node at (-2.6,0.6) {$\Gamma_0$};
\node at (-1.7,0.8) {$L_0$};
%\node at (-1.4,0.68) {$L_1$};
\node at (0.9,0.65) {$L_{k}$};
%\node at (2.3,0.5) {$L_n$};
\node at (-1,0.5) {$\Gamma_1$};
\node at (1.2,0.43) {$\Gamma_{n_k}$};
\node at (2.7,0.5) {$\Gamma_{n_k+1}$};
\node at (-2.9,0.2) {$\gamma_0$};
\node at (-1.35,0.25) {$\gamma_1$};
\node at (0.9,0.1) {$\gamma_{n_k}$};
\node at (2.7,-0.25) {$\gamma_{n_k+1}$};
\node at (-2.4,-1.35) {$D$};
\draw[fill] (0:0) circle [radius=0.01];
\draw[fill] (0:0.1) circle [radius=0.01];
\draw[fill] (0:0.2) circle [radius=0.01];
\draw[fill] (-2.6,-0.05) circle [radius=0.01];
\draw[fill] (-2.1,-0.1) circle [radius=0.01];
\draw[fill] (-1,-0.05) circle [radius=0.01];
\draw[fill] (-0.9,-0.35)  circle [radius=0.01];
\draw[fill] (1.2,-0.05) circle [radius=0.01];
\draw[fill] (1.2,-0.45) circle [radius=0.01];
\draw[fill] (2.6,0.05) circle [radius=0.01];
\draw[fill] (2.7,0.3)  circle [radius=0.01];
%\node at (-1.9,0.2) {$f_1$};
%\node at (-0.5,0.2) {$f_2$};
%\node at (0.5,0.2) {$f_{n-1}$};
%\node at (1.9,0.2) {$f_{n}$};
%\node at (0,1.35) {$F_n$};
\end{tikzpicture}
\caption{\label{fig:thmD}Sketch of the setup of the proof of Theorem D.}
\end{figure}
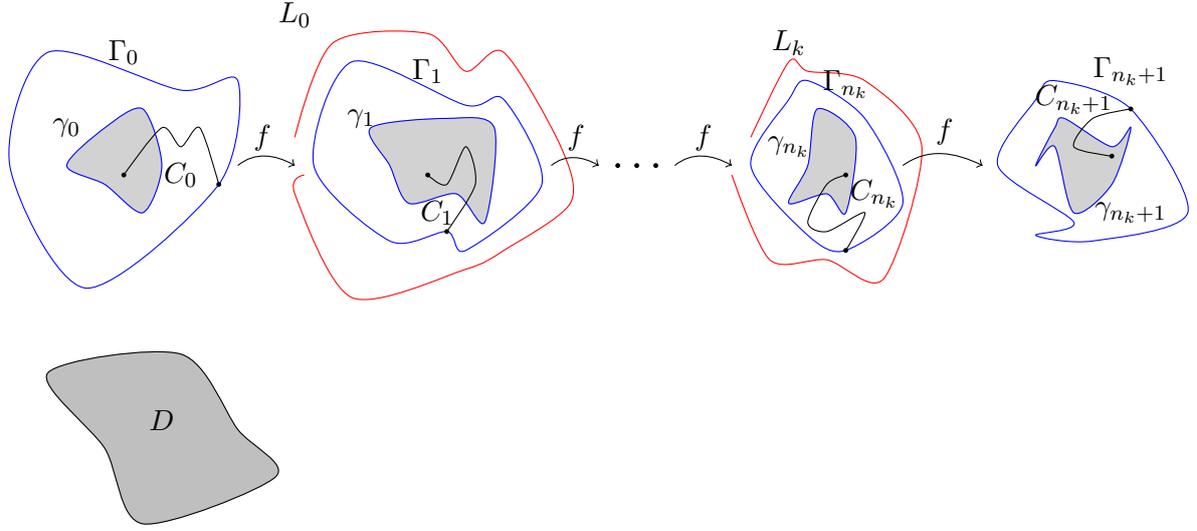

By property~(e), the domain~$D$ must contain an attracting fixed point and so it is contained  in an attracting Fatou component, say~$V$. It then follows by property~(e) that  for all $k \geq 0$ the set  $L_k$ is contained in a  union of Fatou components, $V_k$ say, that maps into $V$.   As above, notice that the $V_k$'s may all be part of the same component.
Since for every $n$ we have that  $\ov{D}\subset \ext{\Gamma_n}$  while $\gamma_n\subset \inter{\Gamma_n}$ %there are orbits in $U_0$ that do not enter $D$
we deduce that $U_0$ is not in the grand orbit of $V$ and hence that $\bigcup_{n\ge 0} U_n \cap \bigcup_{k\ge 0}V_k=\emptyset$. Therefore
\begin{equation}
\label{eq:thmD}
\operatorname{dist}(z',\partial U_k)<  \delta_k :=\max \{ \operatorname{dist}(z ,L_k) : z \in \Gamma_{n_k}\},\;\text{ for all}\; z' \in \Gamma_{n_k} \cap U_k.
\end{equation}

Note that $U_n$ is simply connected for $n \geq 0$. Indeed, if $U_n$ is multiply connected for some $n \geq 0$, then  it is a wandering domain and by \cite[Theorem 1.2]{BRS} there exists $N>0$ such that $f^k(\operatorname{int}\gamma_n)$ contains an annulus $A(r_k,R_k)$ for all $k\geq N$ with $R_k/r_k \to \infty$ as $k \to \infty$. It follows by property~(c) that $A(r_k,R_k)$ is contained in $\operatorname{int} \gamma_{n+k}$ and this contradicts property~(b). So $U_n$ must be simply connected for $n \geq 0$.

We now show that $U_n \subset \operatorname{int}\Gamma_n$, for $n \geq 0$, using proof by contradiction. If there exists $m \geq 0$ for which $U_m$ is not a subset of $\operatorname{int}\;\Gamma_m$, then
it follows from~\eqref{eq:intgamma} and property (a) that
$U_m \cap \Gamma_m \neq \emptyset$ and so we can take $z_m \in \operatorname{int}\gamma_m$ and $z_m'\in U_m \cap \Gamma_m$, and join  them by a compact curve $C_m \subset (U_m \cap \operatorname{int}\;\Gamma_m)$.

Then, by properties~(c) and~(d), we can choose simple curves $C_n$, $n \geq m$, such that $C_n \subset f^{n-m}(C_m) \subset  (U_n \cap \operatorname{int}\;\Gamma_n)$
and also $C_n$ joins $z_n:=f^{n-m}(z_m) {\in \operatorname{int}\gamma_{n}}$ to a point $z_n' \in \Gamma_n \cap f^{n-m}(C_m) \subset U_n$, while $C_n$ lies in $\overline{\operatorname{int} \Gamma_n}$. Such a curve $C_n$ must also intersect $\gamma_n$. Then, on the one hand, since $C_n \subset f^{n-m}(C_m)$ and $f^{n-m}:U_m \to U_n$ is a hyperbolic contraction, we have that
\begin{equation} \label{eq:curvelength}
\length_{U_n} C_n \leq  \length_{U_n}  f^{n-m}(C_m) \leq \length_{U_m} C_m < \infty,
\end{equation}
for all $n\geq m$. On the other hand, by Lemma \ref{lem:shapiro} and (\ref{eq:thmD}), for $n_k \geq m$, we have
\begin{align*}
\length_{U_{n_k}} C_{n_k}
&\geq \dist_{U_{n_k}} (z_{n_k},z_{n_k}')\\
&\geq \frac{1}{2} \log \left(1+ \frac{|z_{n_k}-z_{n_k}'|}{\min\{\operatorname{dist}(z_{n_k}, \partial U_{n_k}),\operatorname{dist}(z_{n_k}', \partial U_{n_k})\}}\right)\\
& \geq \frac{1}{2} \log \left(1+ \frac{|z_{n_k}-z_{n_k}'|}{\operatorname{dist}(z_{n_k}', \partial U_{n_k})}\right)\\
&\geq \frac{1}{2} \log \left(1+ \frac{\operatorname{dist}(\gamma_{n_k}, \Gamma_{n_k})}{\delta_k}\right).
\end{align*}

By property~(f), this quantity tends to infinity as $k\to\infty$,  which contradicts (\ref{eq:curvelength}), so $U_m \subset \inter \Gamma_m$ and hence $U_m$ is a bounded wandering domain by property~(b).

Finally, suppose that, for some $n \geq 0$, there exists $z_n \in \operatorname{int}\gamma_n$ such that both $f(\gamma_n)$ and $f(\Gamma_n)$ wind $d_n$ times round $f(z_n)$. Since $f(\Gamma_n)$ winds $d_n$ times around $f(z_n)$, we deduce that~$f$ takes the value $f(z_n)$ exactly $d_n$ times in $\operatorname{int}\Gamma_n$. Similarly,~$f$ takes the value $f(z_n)$ exactly $d_n$ times in $\operatorname{int}\gamma_n$. Hence~$f$ takes the value $f(z_n)$ exactly $d_n$ times in $U_n$.
Since $U_n$ is a bounded Fatou component, $f:U_n \to U_{n+1}$ is a proper map; since  the above argument holds for a neighbourhood of $f(z_n)$,  we deduce that the degree of $f$ on $U_n$ is equal to $d_n$.
%Since  the above argument holds for a neighbourhood of $f(z_n)$, and $U_n$ is a Fatou component, we deduce that the degree of $f$ on $U_n$ is equal to $d_n$ (since it cannot be greater than that, and $f:U_n \to U_{n+1}$ is a proper map because by the previous part $U_n$ is bounded).

\subsection{Main construction result}

In the proof of  our main construction result, Theorem~\ref{thm:main construction} below, we
  use the following extension of the main lemma in \cite{pathex}, which is a strong version of the well-known Runge's Approximation Theorem.

\begin{lem}[Approximating on infinitely many compact sets]
\label{thm:Runge's approx}
Let $(E_n)$ be a sequence of compact subsets of $\mathbb{C}$ with the following properties:
\begin{itemize}
\item[(i)] $\mathbb{C}\setminus E_n$ is connected, for $n \geq 0$;
\item[(ii)] $E_n \cap E_m = \emptyset$, for $n \neq m$;
\item[(iii)] $\min \{|z|: z \in E_n\} \to \infty$ as $n \to \infty$.
\end{itemize}
Suppose $\psi$ is holomorphic on   $E = \bigcup_{n=0}^{\infty} E_n$ and $j\in \N$. For $n \geq 0$, let $\varepsilon_n > 0$ and let $z_{n,i} \in E_n$, $1 \leq i \leq j$. Then there exists an entire function~$f$ satisfying, for $n \geq 0$,
\begin{equation}
|f(z)-\psi(z)|<\varepsilon_n, \quad \text{for } z\in E_n;
\end{equation}
\begin{equation}
f(z_{n,i})=\psi(z_{n,i}), \quad f'(z_{n,i}) = \psi'(z_{n,i}),\;\text{ for } 1\le i\le j.
\end{equation}

\end{lem}
The main lemma in \cite{pathex} allows for one point $z_n$ in every compact set at which~$f$ and~$f'$ can be specified, but its proof can easily be modified to hold for finitely many points in every $E_n$, as stated above.\\

The following lemma will also be used in the proof of Theorem \ref{thm:main construction}.

\begin{lem}[Hyperbolic distance on disks]
\label{lem:hyp estimate}
Suppose that $0<s<r<1<R$ and set
\[
c(s,R)= \frac{1-s^2}{R-s^2/R},\quad D_r=D(0,r) \quad\text{and}\quad D_R=D(0,R).
\]
If $|z|,|w|\leq s$, then
\begin{equation}
\label{hyp est 1}
\operatorname{dist}_{D_R}(z,w)= \operatorname{dist}_{\mathbb{D}}({z}/{R},{w}/{R})\geq c(s,R)\operatorname{dist}_{\mathbb{D}}(z,w),
\end{equation}
and
\begin{equation}
\label{hyp est 2}
\operatorname{dist}_{D_r}(z,w)= \operatorname{dist}_{\mathbb{D}}({z}/{r},{w}/{r})\leq \frac{1}{c(s/r,1/r)}\operatorname{dist}_{\mathbb{D}}(z,w).
\end{equation}
Also, $0<c(s,R)<1$ and if the variables~$s$,~$r$ and~$R$ satisfy in addition
\begin{equation}\label{srR}
1-r=o(1-s)\;\text{ as } s\to 1\quad\text{and}\quad R-1=O(1-r)\;\text{ as } r\to 1,
\end{equation}
then
\begin{equation}
\label{cR}
c(s,R)\to 1\;\text{ as}\;s \to 1,
\end{equation}
and
\begin{equation}\label{cr}
c\left(s/r,1/r \right) \to 1\;\text{ as}\;s \to 1.
\end{equation}
\end{lem}
\begin{proof}
Suppose that $0<s<r<1<R$ and take $z,w \in \mathbb{D}$ with $|z|,|w|\leq s$. Let $\gamma$ be the hyperbolic geodesic in $\mathbb{D}$ joining ${z}/{R}$ to ${w}/{R}$. Then
$$
\operatorname{dist}_{\mathbb{D}}({z}/{R},{w}/{R})= \int_{\gamma}\frac{2\,|dt|}{1-|t|^2}.
$$
Now substitute $\zeta= Rt$, $t \in \mathbb{D}$, so $|d\zeta|=R|dt|.$ Also let $R\gamma := \{Rz: z \in \gamma\}$. Since $R>1$, we have
\[
\operatorname{dist}_{\mathbb{D}}({z}/{R},{w}/{R})= \frac{1}{R}\int_{R\gamma}\frac{2\,|d\zeta|}{1-|\zeta|^2/R^2}= R\int_{R\gamma}\frac{2\,|d\zeta|}{R^2-|\zeta|^2}.
\]
Now for $\zeta \in R\gamma$ we have $|\zeta|\leq s$, so
\[
\frac{R^2-|\zeta|^2}{1-|\zeta|^2} \leq \frac{{ R^2-s^2}}{1-s^2}, \quad \text{for }\zeta \in R\gamma.
\]

Hence
\[
\operatorname{dist}_{\mathbb{D}}({z}/{R},{w}/{R})= R\int_{R\gamma}\frac{2\,|d\zeta|}{R^2-|\zeta|^2} \geq  \frac{R(1-s^2)}{R^2-s^2} \int_{R\gamma}\frac{2\,|d\zeta|}{1-|\zeta|^2}\geq c(s,R)\operatorname{dist}_{\mathbb{D}}(z,w),
\]
since
\[
\operatorname{dist}_{\mathbb{D}}(z,w)= \min \left\{\int_{\gamma'}\frac{2\,|d\zeta|}{1-|\zeta|^2}:\;\text{for all paths}\;\gamma'\;\text{joining}\;z\;\text{to}\;w\;\text{in}\; \mathbb{D}\right\}.
\]
This proves~(\ref{hyp est 1}).

Next,
$$\operatorname{dist}_{D_r}(z,w)=\operatorname{dist}_{\mathbb{D}}({z}/{r},{w}/{r})\;\;\text{and}\;\;\left|\frac{z}{r}\right|,\left|\frac{w}{r}\right|\leq \frac{s}{r}<1.$$
Hence, by (\ref{hyp est 1}), with $r$ and $R$ replaced by $s/r$ and $1/r$, and $z,w$ replaced by $z/r$ and $w/r$, we obtain
\[
\operatorname{dist}_{\mathbb{D}}({z}/{r},{w}/{r}) \leq \frac{1}{c(s/r,1/r)}\operatorname{dist}_{D_{1/r}}({z}/{r},{w}/{r})= \frac{1}{c(s/r,1/r)}\operatorname{dist}_{\mathbb{D}}(z,w).
\]
This proves~\eqref{hyp est 2}.

It is clear that $0<c(s,R)<1$ since $0<s<1<R$. Finally, suppose that \eqref{srR} holds. Then  $R-1=O(1-r)=o(1-s)$ as $s\ra1$ and hence
\[
c(s,R)=\frac{R(1-s)(1+s)}{(R-s)(R+s)}= \frac{R(1-s)}{1-s+o(1-s)}\,\frac{1+s}{R+s}\to 1\;\text{ as } s\to 1,
\]
and
\[
c(s/r,1/r)= \frac{(r-s)(1+s/r)}{(1-s)(1+s)}= \frac{1-s + o(1-s)}{1-s}\,\frac{1+s/r}{1+s} \to 1\;\text{ as } s\to 1,
\]
which give \eqref{cR} and \eqref{cr}.
\end{proof}

We now give our main construction result, which we use in Section \ref{sec:allexamples} to construct examples.
In these examples, we shall prescribe the orbits of at most two points $z_1,z_2 \in D(0,r_0)$, although the result below allows us to prescribe the orbits of any finite number of points in $D(0,r_0)$.

\begin{thm}[Main construction]\label{thm:main construction}

Let $(b_{n})$ be a sequence of Blaschke products of corresponding degrees $d_n \geq 1$, let $(T_n)$ be   the sequence of translations $z \mapsto z+4n$, and let $(D_n)$ be the sequence of disks $D_n= \{z: |z-4n|<1\},$ $n \geq 0$. Suppose also that $j\in \N$ and $z_i \in D_0$, $1\leq i\leq j$.
Then there exists a transcendental entire function~$f$ having an orbit of bounded, simply connected, escaping, wandering domains $U_n$
such that, for $n\geq 0$,
\begin{itemize}
\item[(i)]  $\overline{\Delta_n'}:=\overline{D(4n,r_n)} \subset U_n \subset D(4n, R_n):=\Delta_n$, where $0<r_n<1<R_n$ and $r_n, R_n \to 1$ as $n \to \infty$;
\item[(ii)] $f_{n+1}:= T_{n+1} \circ b_{n+1} \circ T_{n}^{-1}$ is holomorphic on $\overline{\Delta_n},$ and $|f(z)-f_{n+1}(z)|\to 0$ uniformly on $\overline{\Delta_n}$ as $n \to \infty$;
\item[(iii)] $f^n(z_i)=F_{n}(z_i)$ and $f'((f^n)(z_i))=  f'_{n+1}(F_{n}(z_i)),\;1 \leq i\leq j$, where $F_n=f_n\circ \dots \circ {f_1}$;
\item[(iv)] $f:U_{n} \to U_{n+1}$ has degree $d_{n+1}.$
\end{itemize}

 Finally, if $z, z' \in \overline{D(0,r_0)}$, then we have
\begin{equation}\label{eqtn:double inequality}
 k_n\operatorname{dist}_{D_{n}}(f^{n}(z), f^{n}(z')) \leq \operatorname{dist}_{U_{n}}(f^{n}(z), f^{n}(z'))\leq  K_n \operatorname{dist}_{D_{n}}(f^{n}(z), f^{n}(z')),
\end{equation}
where $0<k_n<1<K_n$ and $k_n,K_n \to 1$ as $n \to \infty$.
\end{thm}

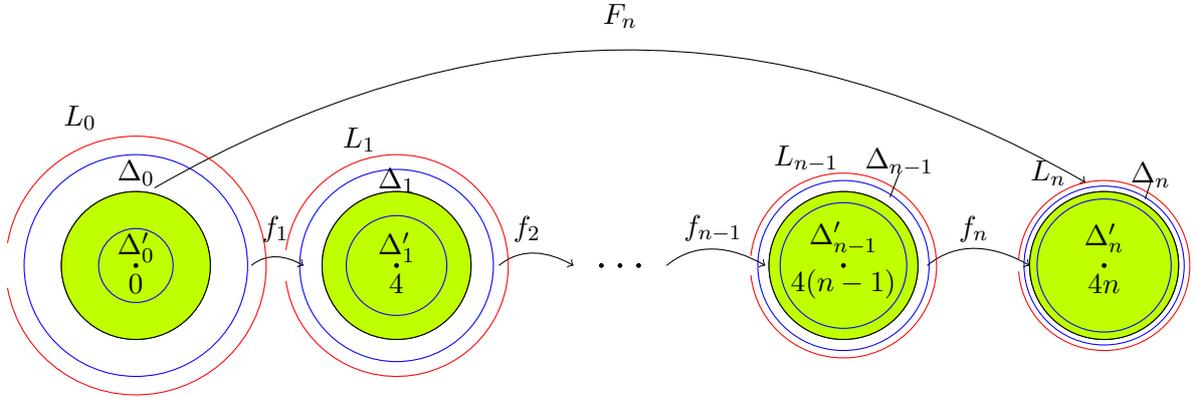
\begin{figure}[hbt!]

\centering
\begin{tikzpicture}[scale=2.45] %Changes the scale
\draw  (1.2,0) circle [radius=0.4]; %outer circle

\draw[blue]  (1.2,0) circle [radius=0.46]; %outer circle
\draw (2.6,0) circle [radius=0.4]; %inner circle

\draw[blue] (2.6,0) circle [radius=0.43]; %inner circle

\draw  (-1.2,0) circle [radius=0.4]; %outer circle
\draw[blue]  (-1.2,0) circle [radius=0.52]; %outer circle
\draw (-2.6,0) circle [radius=0.4]; %inner circle
\draw[blue] (-2.6,0) circle [radius=0.6]; %inner circle
\draw [->] (0:1.65) to [out=40,in=150] (0:2.2);
\draw [->] (0:0.25) to [out=40,in=150] (0:0.78);
\draw [->] (-1.98,0) to [out=40,in=150] (-1.7,0);
\draw [->] (0:-0.65) to [out=40,in=150] (0:-0.25);
\draw [->] (-2.5,0.42)to [out=30,in=150] (2.5,0.45);
\draw[fill=lime] (-2.6,0) circle [radius=0.4];
\draw[fill=lime] (-1.2,0) circle [radius=0.4];
\draw[fill=lime] (1.2,0) circle [radius=0.4];
\draw[fill=lime] (2.6,0) circle [radius=0.4];
\draw[fill] (-2.6,0) circle [radius=0.01];
\draw[fill] (0:-0.1) circle [radius=0.01];
\draw[fill] (0:0) circle [radius=0.01];
\draw[fill] (0:0.1) circle [radius=0.01];
\draw[fill] (1.2,0) circle [radius=0.01];
\draw[fill] (2.6,0) circle [radius=0.01];
\draw[fill] (-1.2,0) circle [radius=0.01];
\draw[blue] (-2.6,0) circle [radius=0.2]; %inner circle
\draw[blue]  (1.2,0) circle [radius=0.34]; %inner circle
\draw[blue]  (-1.2,0) circle [radius=0.27]; %inner circle
\draw[blue] (2.6,0) circle [radius=0.36]; %inner circle
\draw[red] (-1.9,0) arc (0:170:0.7);
\draw[red] (-2.6+0.7*cos{190},0.7*sin{190}) arc(190:360:0.7);
\draw[red] (-0.6,0) arc (0:172:0.6);
\draw[red] (-1.2+0.6*cos{188},0.6*sin{188}) arc(188:360:0.6);
\draw[red] (1.7,0) arc (0:174:0.5);
\draw[red] (1.2+0.5*cos{186},0.5*sin{186}) arc(186:360:0.5);
\draw[red] (3.06,0) arc (0:176:0.46);
\draw[red] (2.6+0.46*cos{184},0.46*sin{184}) arc(184:360:0.46);
\draw (1.45,0.37) to (1.5,0.5);
\draw (2.82,0.36) to (2.85,0.45);
\node at (-2.6,-0.1) {$0$};
\node at (-1.2,-0.1) {$4$};
\node at (1.2,-0.1) {$4(n-1)$};
\node at (2.6,-0.1) {$4n$};
\node at (-2.6,0.5) {$\Delta_0$};
\node at (-2.9,0.8) {$L_0$};
\node at (-1.4,0.68) {$L_1$};
\node at (1,0.58) {$L_{n-1}$};
\node at (2.3,0.5) {$L_n$};
\node at (-1.2,0.46) {$\Delta_1$};
\node at (1.2,0.15) {$\Delta_{n-1}'$};
\node at (2.85,0.49) {$\Delta_n$};
\node at (-2.6,0.1) {$\Delta_0'$};
\node at (-1.2,0.1) {$\Delta_1'$};
\node at (1.5,0.56) {$\Delta_{n-1}$};
\node at (2.6,0.15) {$\Delta_n'$};
\node at (-1.85,0.2) {$f_1$};
\node at (-0.5,0.2) {$f_2$};
\node at (0.5,0.2) {$f_{n-1}$};
\node at (1.9,0.2) {$f_{n}$};
\node at (0,1.35) {$F_n$};
\end{tikzpicture}

\caption{\label{fig:construction}Sketch of the setup of Theorem \ref{thm:main construction}. In green, the disks $D_n$ centred at $4n$. In blue, the boundaries of the disks of radii $r_n$ and $R_n$ in between which lie the boundaries of the wandering domains. In red, the curves $L_n$ introduced in the proof.}
\end{figure}

\begin{proof}
For $n \geq 0$, let
\[
b_n(z) = e^{i \theta_n} \prod_{j=1}^{d_n}\frac{z+a_{n,j}}{1+\overline{a_{n,j}}z},
\]
where $a_{n,j} \in \mathbb{D}$ are not necessarily different from each other, and $\theta_n \in [0, 2\pi)$.

We first define the increasing sequence $(r_n)$ and the decreasing sequence $(R_n)$ inductively. These sequences determine  the following circles which play a key role in the proof (see Figure~\ref{fig:construction}):
\begin{equation}\label{gammaGamma}
\gamma_n =\{z:|z-4n|=r_n\} \quad \text{and} \quad\Gamma_n=\{z:|z-4n|=R_n\}.
\end{equation}
First, take $R_0\in (1,3/2)$ such that $R_0< {1}/{\max_j\{|a_{1,j}|\}}\}$, which ensures that $b_1$ is holomorphic inside and in a neighborhood of $\Gamma_0$, and take $r_0\in (1/2,1)$ such that $r_0>\max_i|z_i|$ and also such that $b_{1}(z)=w$ has exactly $d_{1}$ solutions in $D(0,r_0)$ for $w \in D(0,1/2)$. Now assume that $r_k, R_k$ have been chosen for $k=0,\dots,n-1$, for some $n \in \N$. We choose $r_{n}$ and  $R_{n}$ so that the following statements all hold:
\begin{equation}\label{cond2}
0<1-r_{n} \leq \min \left\{\frac{1-r_{n-1}}{2}, {\operatorname{dist}(f_{n}(\gamma_{n-1}), \partial D_{n})^2} \right\};
\end{equation}
 \begin{equation}\label{cond3}
 f_{n}(\gamma_{n-1})\;\text{ winds exactly $d_{n}$ times round } D(4n,1/2);
 \end{equation}
\begin{equation}\label{cond1}
0<R_{n}-1 \leq \min \left\{\frac{R_{n-1}-1}{2}, {1-r_{n}}, \frac{1}{2} \operatorname{dist}(f_{n}(\Gamma_{n-1}), \partial D_{ n}), \frac{1}{\max_j\{|a_{n+1,j}|\}}- 1 \right\}.
\end{equation}
These properties prescribe the values $r_n$ and $R_n$, and hence the circles $\gamma_n$ and  $\Gamma_n$. In particular, by (\ref{cond2}) and (\ref{cond1}), the sequence $(r_n)$ increases to~1 and the sequence $(R_n)$ decreases to 1, and the maps $f_n$, $n \geq 0$, defined in property~(ii), satisfy
\begin{equation}\label{propa}
\gamma_{n+1}\;\text{ surrounds}\;f_{n+1}(\gamma_n),
\end{equation}
\begin{equation}\label{propb}
f_{n+1}(\Gamma_n)\;\text{ surrounds}\; \Gamma_{n+1}.
\end{equation}

Our aim is to use Lemma~\ref{thm:Runge's approx} to approximate all the maps $f_n$ by a single entire function $f$ such that, for $n \geq 0$, $\gamma_{n+1}$ surrounds $f(\gamma_n)$ and
 $f(\Gamma_n)$ surrounds $\Gamma_{n+1}$.

We first define
\begin{equation}\label{delta}
\delta_n= R_n-r_n,\quad \text{for } n \geq 0,
\end{equation}
and observe that $\delta_n \to 0$ as $n \to \infty$.

 We then define $L_n$, $n\geq 0$, to be the curve
 \begin{equation}\label{eq:reefs}
 L_n:=\{z:|z-4n|=R_n+\delta_n^2/2,\;|\operatorname{arg}(z)| \leq \pi - \delta_n^2\},
 \end{equation}
 so
 \begin{equation}\label{Lndist}
 \max \{\operatorname{dist}(z,L_n): z \in \Gamma_n\} \le 2\delta_n^2, \quad \text{for } n \geq 0,
 \end{equation}
 and define the error quantities
\begin{equation}\label{eq:error}
\varepsilon_n= \min \left\{\frac{1}{4} \operatorname{dist}(f_n(\gamma_{n-1}), \partial D_{n}), \frac{1}{4} \operatorname{dist}(f_n(\Gamma_{n-1}), \partial D_{n}), \delta_n/4 \right\}>0,
\end{equation}
for $n\geq 1$. Since $0<\varepsilon_n \leq \delta_n/4$, we have that $\varepsilon_n < \delta_0/4<1/4$, $n \geq 1$, and $\varepsilon_n \to 0$ as $n \to \infty$.

We now apply Lemma~\ref{thm:Runge's approx} to the sets $E_0=\overline{D(-4,1)}$ and $E_{2k+1} = L_{k}$, $E_{2k+2}=\overline{\Delta_{k}}$, for $k \ge 0$, with  the function $\psi$ defined by
 $$\psi(z)=
\begin{cases}
z/2 - 2,\;\text{ if}\;z \in \overline{D(-4,1)},\\
-4,\;\text{ if} \; z \in L_n, \; n \in \N,\\
f_n(z),\;\text{ if}\;z \in \overline{\Delta_n}, \; n \in \N.
\end{cases}
$$

Lemma~\ref{thm:Runge's approx} allows us to %requires that we
 choose finitely many points $z_{n,i}$, $1\le i \le j$, in each set $E_n$ where we do the approximation. The choice of these points in $\overline{D(-4,1)} \cup \bigcup_{n=0}^{\infty} L_n$ plays no role in our argument. In $E_2=\overline{\Delta_0}$ we choose $z_{2,i}=z_i \in D_0$, $1\le i \le j$, and in $E_{2k+2} = \overline{\Delta_{k}}$, $k \ge 1$, we choose $z_{2k+2,i} = F_k(z_i)$, $1\le i\le j$, where $F_k=f_k\circ \dots \circ {f_1}$.

It then follows from Lemma~\ref{thm:Runge's approx} that there exists an entire function~$f$ such that, for  $n \geq 0$,
\begin{equation}\label{approx1}
|f(z)-f_{n+1}(z)| < \varepsilon_{n+1}, \text{\ \  for }  z \in \overline{\Delta_n};
\end{equation}
\begin{equation}\label{approx2}
|f(z)+4| \leq 1/2,  \text{\ \  for }  z\in L_n;
\end{equation}
\begin{equation}\label{approx3}
|f(z) - z/2+2| <1/4,    \text{\ \  for }  z \in \overline{D(-4,1)};
\end{equation}
\begin{equation} \label{approx4}
f^n(z_i)=F_n(z_i),  \text{\ \  for }  1\leq i\leq j;
\end{equation}
\begin{equation}\label{approx5}
f'((f^n)(z_i))= f'(F_n(z_i))=f_{n+1}'(F_n(z_i)),  \text{\ \  for }  1\leq i\leq j.
\end{equation}

It follows from (\ref{cond2}), (\ref{cond1}), (\ref{approx1}) and (\ref{eq:error}) that, for $n \geq 0$,
\begin{equation}\label{propB}
\gamma_{n+1}\;\text{surrounds}\; f(\gamma_n) \;\text{ (which surrounds the point   } 4(n+1));
\end{equation}
\begin{equation}\label{propA}
f(\Gamma_n)\;\text{surrounds}\;\Gamma_{n+1}.
\end{equation}

We now apply Theorem D to the Jordan curves $\gamma_n$, $\Gamma_n$, $n \geq 0$, the compact curves $L_n$, $n \geq 0$, and the bounded domain $D = D(-4,1)\subset E_0$, noting that these sets satisfy the required hypotheses. Indeed, the hypotheses~(a) and~(b) are clearly true,~(c) follows from~\eqref{propB},~(d) follows from~\eqref{propA},~(e) holds by (\ref{approx2}) and (\ref{approx3}), and~(f) holds by~\eqref{delta} and~\eqref{Lndist}.

Part~(i) of our result now follows from Theorem D, part~(ii) is true by construction, and part~(iii) follows from~\eqref{approx4} and~\eqref{approx5}. We now show that part~(iv) holds.

By \eqref{eq:error} and \eqref{approx1}, we can write $f(z)=f_{n+1}(z)+e_{n+1}(z)$ for some holomorphic map $e_n(z)$ which satisfies  $|e_{n+1}(z)| < 1/4$, for $z\in \overline{\Delta_n}$.

By \eqref{cond3}, we have
 \[
 |f_{n+1}(z)-4(n+1)|\ge 1/2,\;\text{ for } z\in \gamma_n=\partial \Delta_n'.
 \]
It follows from this together with the fact that $|e_n(z)| < 1/4$, for $z\in \overline{\Delta_n}$ and \eqref{cond3} that
\begin{equation}\label{f-omits}
|f(z)-4(n+1)|\ge 1/4,\;\text{ for } z\in \gamma_n,
\end{equation}
and $f(\gamma_n)$ winds exactly $d_{n+1}$ times around $4(n+1)$, so~$f$ takes the value $4(n+1)$ exactly $d_{n+1}$ times in $\Delta_n'$. Similarly, by (\ref{cond3}), (\ref{eq:error}) and (\ref{approx1}), $f(\Gamma_n)$ winds exactly $d_{n+1}$ times around $4(n+1)$, so~$f$ takes the value $4(n+1)$ exactly $d_{n+1}$ times in $\Delta_n$. Therefore, by the final statement of Theorem~D, $f:U_n \to U_{n+1}$ has degree $d_{n+1}$.

It remains to prove the double inequality \eqref{eqtn:double inequality}, which compares the hyperbolic distances in $U_n$ between points of two orbits under~$f$ with the corresponding hyperbolic distances in the disks~$D_n$. To do this, we let $s_n:= 1-\tfrac34\operatorname{dist}(f_n(\gamma_{n-1}), \partial D_n)$, for $n\geq 1$, and note that, if $z,z' \in \overline{D(0,r_0)}$, then
\[
f^n(z), f^n(z')\in \overline{D(4n, s_n)}\subset \Delta_n', \;\text{ for } n\in \N,
\]
by \eqref{propa}, \eqref{eq:error} and \eqref{approx1}.

Now $1-r_n =o(1-s_n)$ as $n\to \infty$, by (\ref{cond2}), and $R_n-1 \leq 1-r_n$, by (\ref{cond1}), so the properties \eqref{srR} hold for the sequences $(s_n)$, $(r_n)$ and $(R_n)$. Also,
\[
 \operatorname{dist}_{\Delta_{n}}(f^n(z),f^n(z'))\le \operatorname{dist}_{U_{n}}(f^n(z),f^n(z'))\le \operatorname{dist}_{\Delta_{n}'}(f^n(z),f^n(z')),
\]
since $\Delta_n'\subset D_n \subset \Delta_n$. Therefore, we deduce from Lemma~\ref{lem:hyp estimate} (translated to the disks $D_n$) that
\[
c(s_n,R_n)\operatorname{dist}_{D_n}(f^n(z),f^n(z'))\le \operatorname{dist}_{U_{n}}(f^n(z),f^n(z'))\le \frac{1}{c(s_n/r_n,1/r_n)}\operatorname{dist}_{D_n}(f^n(z),f^n(z'))
\]
and
\[
 c(s_n,R_n) \to 1\;\text{ as } n \to \infty,\quad c(s_n/r_n,1/r_n) \to 1\;\text{ as } n \to \infty,
\]
which gives (\ref{eqtn:double inequality}).
\end{proof}

%%%%%%%%%%%%%%%%%%%%%%%%%%%%%
%%%%%%%%%%%%%%%%%%%%%%%%%%%%%

\section{Examples: Proof of Theorem E}\label{sect:proof of examples}
\label{sec:allexamples}
In this section we construct the examples described in Theorem~E. In every case we use Theorem~\ref{thm:main construction} and  the notation  there. Hence   $(b_{n})$ denotes the sequence of Blaschke products of degree $d_n \geq 1$; $(T_n)$ the sequence of real translations $z \mapsto z+4n$;  and  $(D_n)$ the sequence of disks $D_n= \{z: |z-4n|<1\},$ $n \geq 0$. Moreover, for $n \in \N$, we set $B_n=b_n\circ\dots\circ b_{1}$, $f_n=T_{n} \circ b_n \circ T_{n-1}^{-1},$ and  $F_n=f_n \circ\dots\circ f_{1}$, so $F_n = T_{n} \circ B_n$; see Figure~\ref{Bn-fig}.

\begin{center}%Centers diagram
\begin{figure}[hbt!]\centering
\begin{tikzpicture}[scale=2.4] %Changes the scale
\draw  (1.2,0) circle [radius=0.4]; %outer circle
\draw (2.6,0) circle [radius=0.4]; %inner circle
\draw  (-1.2,0) circle [radius=0.4]; %outer circle
\draw (-2.6,0) circle [radius=0.4]; %inner circle
\draw [->] (0:1.65) to [out=40,in=150] (0:2.15);
\draw [->] (0:0.25) to [out=40,in=150] (0:0.75);
\draw [->] (0:-2.15) to [out=40,in=150] (0:-1.65);
\draw [->] (0:-0.75) to [out=40,in=150] (0:-0.25);
\draw [->] (-2.5,0.5)to [out=30,in=150] (2.5,0.5);
\draw [->] (-2.5,-0.5)to [out=330,in=210] (2.5,-0.5);
\draw [->] (-3,-0.2)to [out=200,in=30](-3.3,-0.3) to [out=230,in=240] (-2.9,-0.3);
\draw[fill] (0:-0.1) circle [radius=0.01];
\draw[fill] (0:0) circle [radius=0.01];
\draw[fill] (0:0.1) circle [radius=0.01];
\node at (-2.6,0) {$D_0$};
\node at (-1.2,0) {$D_1$};
\node at (1.2,0) {$D_{n-1}$};
\node at (2.6,0) {$D_n$};
\node at (-1.9,0.2) {$f_1$};
\node at (-0.5,0.2) {$f_2$};
\node at (0.5,0.2) {$f_{n-1}$};
\node at (1.9,0.2) {$f_{n}$};
\node at (0,1.35) {$F_n$};
\node at (0,-1.1) {$T_n$};
\node at (-3.2,-0.12) {$B_n$};
\end{tikzpicture}
\caption{\label{Bn-fig} The maps $f_n, B_n, T_n$}
\end{figure}
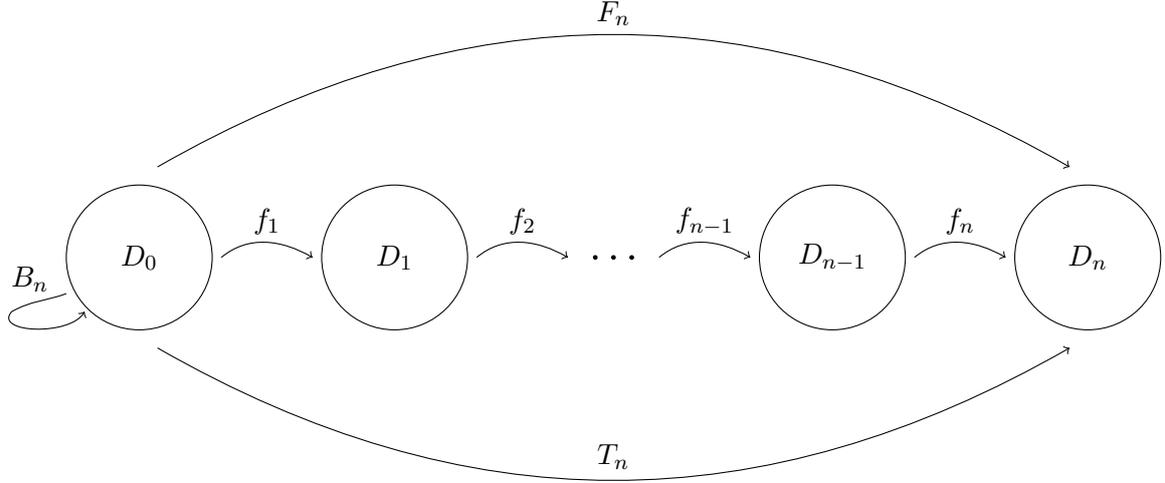
\end{center}

 \subsection{Preliminary lemmas}
 We first prove two lemmas that will be used in the constructions. 

\begin{lem}\label{bnboundary}
Let $f$ be a transcendental entire function with an orbit of bounded, simply connected, wandering domains $U_n$ arising from
Theorem~\ref{thm:main construction}, with Blaschke products $b_n$ and associated functions $B_n$ and $F_n$ such that $f^n(0) = F_n(0)$, for $n \in \N$. Then, we have the following cases.
\begin{itemize}
\item[\rm(a)] If $B_n(0) \to 0$ as $n \to \infty$, then, for all $z\in U_0$,
    \[
    \liminf_{n\to\infty} \operatorname{dist}(f^{n}(z),\partial U_{n})>0,
    \]
    that is, all orbits stay away from the boundary.
\item[\rm(b)] If there exists a subsequence $n_k\to \infty$ with $B_{n_k}(0) \to 1$ and a different subsequence $m_k\to\infty$ with $B_{m_k}(0) \to 0$, then $\operatorname{dist}(f^{n_k}(z),\partial U_{n_k})\to 0$ for all $z\in U$, while
    \[
    \liminf_{k \to \infty} \operatorname{dist}(f^{m_k}(z),\partial U_{m_k})>0, \quad\text{for all }z\in U.
    \]
\item[\rm(c)] If $B_n(0) \to 1$ as $n \to \infty$, then $\operatorname{dist}(f^{n}(z),\partial U_{n})\to 0$ for all $z\in U$, that is, all orbits converge to the boundary.
\end{itemize}

\end{lem}

\begin{proof}
It follows from Theorem C that all points in a simply connected wandering domain have the same limiting behaviour in relation to the boundary and so, in each case, it is sufficient to find just one point  whose orbit behaves as required. We choose this point to be $0 \in U_0$.

If $B_n(0) \to 0$ as $n \to \infty$ then
\[
f^n(0) - 4n = F_n(0) - 4n \to 0 \quad \text{as } n \to \infty
\]
and so, by Theorem~\ref{thm:main construction} part~(i), we have
\[
    \liminf_{n\to\infty} \operatorname{dist}(f^{n}(0),\partial U_{n}) = 1>0,
    \]
 which  is sufficient to prove part~(a).

If $B_n(0) \to 1$ as $n \to \infty$ then
\[
f^n(0) - (4n + 1) = F_n(0) - (4n+1) \to 0 \; \text{ as } n \to \infty
\]
and so, by Theorem~\ref{thm:main construction} part (i), we have
\[
     \operatorname{dist}(f^{n}(0),\partial U_{n}) \to 0\;\text{ as } n\to\infty,
    \]
 which is sufficient to prove part~(c).

The proof of part~(b) follows in a similar way.
\end{proof}

In some of our constructions we use the following properties about a specific family of Blaschke products of degree~2.

\begin{lem}
\label{lem:onethird}
Let $b(z)=\left( \frac{z+a}{1+az}\right)^2,$ where $1/3\le a<1$, and let $0<r<s<1$. Then
\begin{itemize}
\item[(a)] the function~$b$ has a fixed point at 1, which is attracting if $a>1/3$ and parabolic if $a=1/3$, and $b^n(r) \to 1$ as $n \to \infty$ for all $a\geq 1/3$;
\item[(b)] $\operatorname{dist}_{\mathbb{\D}}(b^n(r),b^n(s)) \nrightarrow 0$ as $n \to \infty$ if $a>1/3;$
\item[(c)] $\operatorname{dist}_{\mathbb{\D}}(b^n(r),b^{n+1}(r)) \sim O(1/n)$ as $n \to \infty$ if $a=1/3$.
\end{itemize}
\end{lem}
\begin{proof}
The proof of part (a) is straightforward.\\

For part (b) note first that
\begin{equation}\label{eq:lem6.1}
\operatorname{dist}_{\mathbb{D}}(b^n(r),b^{n}(s))= \int_{b^n(r)}^{b^n(s)} \frac{2\,dt}{1-t^2} \geq \int_{b^n(r)}^{b^n(s)} \frac{dt}{1-{t}}  = \log \frac{1 - b^n(r)}{1-b^n(s)}.\end{equation}

Also, since~$1$ is an attracting fixed point of $b$ when $a>1/3$, there exist $\lambda \in (0,1)$ and $d>c>0$ such that $1-b^n(r) \sim d\lambda^n$ and $1-b^n(s) \sim c\lambda^n$ as $n\to\infty$. Hence, by (\ref{eq:lem6.1}),
$$
\lim_{n\to \infty}\operatorname{dist}_{\mathbb{D}}((b^n(r),{b^{n}(s)})   \geq \log \frac{d}{c} >0.
$$
For part~(c) we use a similar approach. First note that
\begin{equation}\label{eq:lem6.1b}
\operatorname{dist}_{\mathbb{D}}(b^n(r),b^{n+1}(r)) \leq  \int_{b^n(r)}^{{ b^{n+1}(r)}} \frac{2\,dt}{1- t}= { 2 } \log \frac{1 - b^n(r)}{1-b^{n+1}(r)}.
\end{equation}
When $a = 1/3$, we have $b(1)=1$, $b'(1)=1$, $b''(1)=0$ and $b'''(1)\ne 0$, so $1-b^n(r) \sim c/n^{1/2}$ as $n \to \infty$, where $c>0$; see \eqref{eq:parabolic1}. We deduce that
\[
\frac{1-b^n(r)}{1-b^{n+1}(r)} \sim \frac{(n+1)^{1/2}}{n^{1/2}}= \left(1+\frac{1}{n}\right)^{1/2}=1+\frac{1}{2n}+O(1/n^2)\;\text{ as }n\to\infty.
\]
The result now follows by putting this estimate into~\eqref{eq:lem6.1b}.
\end{proof}

 \subsection{The nine types of simply connected wandering domains}
We now prove part (a) of Theorem E by constructing examples corresponding to each of the nine cases given in Theorems~A and~C. The following maps will play key roles in the constructions:
\[
b(z)=\left(\frac{z+a}{1+az}\right)^2, \;\text{ for } 1/3\le a<1,
\]
\begin{equation}\label{mun}
\mu_n(z) = \frac{z+a_n}{1+a_nz} \quad \text{and} \quad \tilde{\mu}_n(z)= \frac{z-a_n^2}{1-a_n^2z}, \quad \text{for } n \in \N,
\end{equation}
where $a_n \in (0,1)$ is an arbitrary sequence satisfying $a_n \to 1$ as $n \to \infty$.

Examples~1,~2 and~3, which follow, correspond to the three cases of Theorem~A. Within each of them we give three functions, corresponding to the three cases of Theorem~C.
\vspace{0.25cm}
\setcounter{ex}{0}
\begin{ex}[\bf Three contracting wandering domains]
For each of the cases (a), (b) and~(c)  of Theorem C, there exists a transcendental entire function $f$ having a sequence of bounded, simply connected, escaping  contracting wandering domains $(U_n)$ with the stated behaviour:
\begin{itemize}
\item[\rm(a)]
%there is no sequence $n_k\to \infty$, and no point $z\in U$ such that $\operatorname{dist}(f^{n_k}(z),\partial U_{n_k})\to 0$ as $k\to\infty$, that is, all orbits stay away from the boundary;
    for all $z\in U$,
    \[
    \liminf_{n\to\infty} \operatorname{dist}(f^{n}(z),\partial U_{n})>0,
    \]
    that is, all orbits stay away from the boundary;
\item[\rm(b)] there exists a subsequence $n_k\to \infty$ for which $\operatorname{dist}(f^{n_k}(z),\partial U_{n_k})\to 0$ for all $z\in U$, while for a different subsequence $m_k\to\infty$ we have that
    \[\liminf_{k \to \infty} \operatorname{dist}(f^{m_k}(z),\partial U_{m_k})>0, \quad\text{for }z\in U;\]
\item[\rm(c)] $\operatorname{dist}(f^{n}(z),\partial U_{n})\to 0$ for all $z\in U$, that is, all orbits converge to the boundary.
\end{itemize}
\end{ex}
\begin{proof}
%In each case we choose a sequence of Blaschke products $(b_n)$ and use Theorem~\ref{thm:main construction} to deduce the existence of a suitable transcendental entire function $f$ with an associated orbit of wandering domains $(U_n)$.

(a) Let $b_n(z)=z^2$, for $n\in \N$, and apply Theorem~\ref{thm:main construction} with the points $z_1 = 0$ and $z_2=1/2$. For $n \in \N$, we have
\begin{eqnarray*}
\operatorname{dist}_{D_{{n}}}(f^n(0),f^n(1/2))&  = & \operatorname{dist}_{\mathbb{D}}(F_n(0),F_n(1/2))\\
& = & \operatorname{dist}_{\mathbb{D}}(B_n(0),B_n(1/2))\\
& = & \operatorname{dist}_{\mathbb{D}}(0,1/2^{2^n}) \rightarrow 0  \;\text{ as}\;n\to \infty.
\end{eqnarray*}
It follows from~\eqref{eqtn:double inequality} that
 \[
 \operatorname{dist}_{U_{n}}(f^n(0),f^n(1/2)) \to 0\;\text{as}\;n\to \infty.
 \]
 By Theorem A, this is sufficient to show that $U_0$ is contracting. Since $B_n(0) = 0$, for $n \in \N$, the result now follows from case (a) of Lemma~\ref{bnboundary}.

 %
% Hence $\operatorname{dist}_{\mathbb{D}}(0,B_n(1/2))= \operatorname{dist}_{D_{n}}(4n,F_n(1/2)) \to 0$ as $n \to \infty$. It follows from Theorem \ref{thm:main construction} that there exists a transcendental entire function $f$ with a sequence of bounded, simply connected, escaping wandering domains $(U_n)$ such that $f^n(0)=4n$ and $f^n(1/2)= 4n+1/2^{ 2^{n}}$ and $0, 1/2 \in U_0$. {\violet  It also follows from (\ref{eqtn:double inequality})?} %the last part of Theorem \ref{thm:main construction}
%  that $\operatorname{dist}_{U_{n}}(f^n(0),f^n(1/2)) \to 0$ as $n \to \infty$ and by Theorem A we deduce that $\operatorname{dist}_{U_{n}}(f^n(z),f^n(w)) \to 0$ for all $z,w \in U_0$.
%We note that, for each $n \in \N$, $f:U_{n-1}\to U_{n}$ has degree 2, since $b_n$ has degree 2.

(b) In this case, for $n\in \N$, we let

$$b_n(z)=
\begin{cases}
z^2 & \text{if $n=3k-2$, $k\geq 1$},\\
\mu_k(z), & \text{if $n=3k-1$,  $k\geq 1$},\\
\mu_k^{-1}(z), & \text{if $n=3k$, $k\geq 1$},
\end{cases}
$$
where $\mu_k$ is as defined in~\eqref{mun}. As in case~(a), we apply Theorem~\ref{thm:main construction} with $z_1 = 0$ and $z_2=1/2$. For $k \in \N$, we have
\begin{eqnarray*}
 \operatorname{dist}_{D_{3k}}(f^{3k}(0),f^{3k}(1/2)) & = & \operatorname{dist}_{\mathbb{D}}(F_{3k}(0),F_{3k}(1/2))\\
 & = & \operatorname{dist}_{\mathbb{D}}(B_{3k}(0),B_{3k}(1/2))\\
& = & \operatorname{dist}_{\mathbb{D}}(0,1/2^{2^ k}) \rightarrow 0  \;\text{ as}\;k\to \infty.
\end{eqnarray*}
%{\red [In case the color is difficult to see, note that I changed the exponent from $2^{3k}$ to $2^k$. Is this correct?]}
 As in case~(a), this is sufficient to show that $U_0$ is contracting. Since $B_{3k}(0) = 0$, for $k \in \N$, and $B_{3k-1}(0)=a_{3k-1} \to 1$ as $k \to \infty$, the conclusion now follows from case (b) of Lemma~\ref{bnboundary}.

  (c) In this case we let $b_n(z)=b(z) =\left(\frac{z+1/3}{1+z/3}\right)^2$, for $n \in \mathbb{N}$, and we apply Theorem~\ref{thm:main construction} with $z_1 = 0$ and $z_2 = b(0)$. For $n \in \N$, we have
  \begin{eqnarray*}
  \operatorname{dist}_{D_{n}}(f^n(0),f^n(b(0))) & = &\operatorname{dist}_{\D}(F_n(0),F_n(b(0)))\\
  & = & \operatorname{dist}_{\D}(B_n(0),B_n(b(0)))\\
  & = & \operatorname{dist}_{\D}(b^n(0),b^{n+1}(0)) \to 0 \;\text{as}\;n\to \infty,
  \end{eqnarray*}
  by Lemma \ref{lem:onethird}(c). As before, this is sufficient to show that $U_0$ is contracting. It also follows from Lemma~\ref{lem:onethird}(a) that $B_n(0) = b^n(0) \to 1$ as $n \to \infty$ and the result now follows from case~(c) of Lemma~\ref{bnboundary}
  \end{proof}

\begin{rem*}
All three cases of Example 1 are in fact super-contracting (see Definition \ref{def:strongly contracting}).
\end{rem*}

\begin{ex}[\bf Three semi-contracting wandering domains]
For each of the cases~(a),~(b) and~(c) of Theorem C, there exists a transcendental entire function~$f$ having a sequence of bounded, simply connected, escaping,  semi-contracting, wandering domains $(U_n)$ with the stated behaviour.
%\begin{itemize}
%\item[\rm(a)]
%%there is no sequence $n_k\to \infty$, and no point $z\in U$ such that $\operatorname{dist}(f^{n_k}(z),\partial U_{n_k})\to 0$ as $k\to\infty$, that is, all orbits stay away from the boundary;
%    for all $z\in U$,
%    \[
%    \liminf_{n\to\infty} \operatorname{dist}(f^{n}(z),\partial U_{n})>0,
%    \]
%    that is, all orbits stay away from the boundary;
%\item[\rm(b)] there exists a subsequence $n_k\to \infty$ for which $\operatorname{dist}(f^{n_k}(z),\partial U_{n_k})\to 0$ for all $z\in U$, while for a different subsequence $m_k\to\infty$ we have that
%    \[\liminf_{k \to \infty} \operatorname{dist}(f^{m_k}(z),\partial U_{m_k})>0, \quad\text{for }z\in U;\]
%\item[\rm(c)] $\operatorname{dist}(f^{n}(z),\partial U_{n})\to 0$ for all $z\in U$, that is, all orbits converge to the boundary.
%\end{itemize}
\end{ex}
\begin{proof}
%In each case we choose a sequence of Blaschke products $(b_n)$ and use Theorem~\ref{thm:main construction} to deduce the existence of a suitable transcendental entire function $f$ with an associated orbit of wandering domains $(U_n)$.
%

%Let $\mu_n(z)= \frac{z+a_n}{1+a_n z},$ where $a_n \in (0,1)$ and $a_n \to 1$ as $n\to \infty$, $\tilde{\mu_n}(z)= \frac{z-a_n^2}{1-a_n^2z},$ and $\pi(z)=z^2$.
  (a) In this case we let $b_n(z) = \tilde{\mu}_n ((\mu_n(z))^2)$, for $n \in \N$, where $\mu_n$ and $\tilde{\mu}_n$ are as defined in~\eqref{mun}. We apply Theorem~\ref{thm:main construction} with the points $z_1 = 0$ and $z_2=1/2$

A calculation shows that, for $n \in \N$, we have $b_n(0) = 0$ and $b_n'(0)= \frac{2a_n}{1+a_n^2}\to 1$ as $n \to \infty$. Hence we can choose $(a_n)$ so that, in addition,
\begin{equation}\label{eqtn:conditions on an}
\sum_{n=1}^{\infty} (1-b_n'(0)) <\infty.
\end{equation}

It follows from Theorem~\ref{Thm zero}(b) that $B_n(1/2) \nrightarrow 0$ as $n \to \infty$. Thus
\begin{eqnarray*}
\operatorname{dist}_{D_{{n}}}(f^n(0),f^n(1/2))&  = & \operatorname{dist}_{\mathbb{D}}(F_n(0),F_n(1/2))\\
& = & \operatorname{dist}_{\mathbb{D}}(B_n(0),B_n(1/2))\\
& = & \operatorname{dist}_{\mathbb{D}}(0,B_n(1/2)) \nrightarrow 0  \;\text{as}\;n\to \infty.
\end{eqnarray*}

It follows from~\eqref{eqtn:double inequality} that
 \[
 \operatorname{dist}_{U_{n}}(f^n(0),f^n(1/2)) \nrightarrow 0\;\text{ as}\;n\to \infty,
 \]
and so $U_0$ is not contracting.
 Also, for $n \in \N$, the Blaschke product $b_n$ has degree 2 and so, by Theorem~\ref{thm:main construction} part (iv), $f: U_{n-1}\to U_n$ has degree 2.  Thus $U_0$ is not eventually isometric and so it follows from Theorem A that $U_0$ is semi-contracting.

 Since $B_n(0) = 0$, for $n \in \N$, the result now follows from case (a) of Lemma~\ref{bnboundary}.

(b) In this case, for $n \in \N$, we let

$$
b_n(z)=
\begin{cases}
\mu_k(z), & \text{if $n=3k-2$, $k\geq 1$},\\
z^2, & \text{if $n=3k-1$, $k\geq 1$},\\
\tilde{\mu}_k(z), & \text{if $n=3k$, $k\geq 1$},
\end{cases}
$$
where  $\mu_k$ and $\tilde{\mu}_k$ are as defined in~\eqref{mun}. Note the similarity to case~(a), where each Blaschke product $b_n$ was defined to be the composite of the three maps above.

As in case~(a), we apply Theorem~\ref{thm:main construction} with $z_1 = 0$ and $z_2=1/2$. Using similar arguments to those used in part (a), we can choose $(a_n)$ such that
\[
 \operatorname{dist}_{U_{3k}}(f^{3k}(0),f^{3k}(1/2)) \nrightarrow 0\;\text{ as}\;k\to \infty,
 \]
and so $U_0$ is not contracting. Also, for $k \in \N$, the Blaschke product $b_{3k-1}$ has degree 2 and so, by Theorem~\ref{thm:main construction} part (iv), $f: U_{3k-2}\to U_{3k-1}$ has degree 2.  Thus $U_0$ is not eventually isometric and so it follows from Theorem A that $U_0$ is semi-contracting.

Since $B_{3k}(0) = 0$, for $k \in \N$, and $B_{3k-2}(0)=a_{3k-2} \to 1$ as $k \to \infty$, the conclusion now follows from case (b) of Lemma~\ref{bnboundary}.

 (c) In this case we choose $a > 1/3$ and let $b_n(z)=b(z) = \left(\frac{z+a}{1+az}\right)^2$,  for $n \in \mathbb{N}$. As in Example 1(c), we apply Theorem~\ref{thm:main construction} with $z_1 = 0$ and $z_2 = b(0)$. For $n \in \N$, we have
\begin{align*}
\operatorname{dist}_{D_{n}}(f^n(0),f^n(b(0))) & = \operatorname{dist}_{\mathbb{D}}(F_n(0),F_n(b(0)))\\
& =  \operatorname{dist}_{\D}(B_n(0),B_n(b(0)))\\
& =  \operatorname{dist}_{\D}(b^n(0),b^{n+1}(0)) \nrightarrow 0 \;\text{ as}\;n\to \infty,
\end{align*}
by Lemma \ref{lem:onethird}(b). Arguing as above, this is sufficient to show that $U_0$ is not contracting. Also, for $n \in \N$, the Blaschke product $b_n$ has degree 2 and so, as above, it follows that $U_0$ is not eventually isometric. Hence, by Theorem~A, it is semi-contracting.

It also follows from Lemma~\ref{lem:onethird}(a) that $B_n(0) = b^n(0) \to 1$ as $n \to \infty$ and the result now follows from case~(c) of Lemma~\ref{bnboundary}
\end{proof}

\begin{ex}[\bf Three eventually isometric wandering domains]
For each of the cases (a), (b) and (c) of Theorem~C, there exists a transcendental entire function~$f$ having a sequence of bounded, simply connected, escaping, eventually isometric, wandering domains $(U_n)$ with the stated behaviour.
%\begin{itemize}
%\item[\rm(a)]
%%there is no sequence $n_k\to \infty$, and no point $z\in U$ such that $\operatorname{dist}(f^{n_k}(z),\partial U_{n_k})\to 0$ as $k\to\infty$, that is, all orbits stay away from the boundary;
%    for all $z\in U$,
%    \[
%    \liminf_{n\to\infty} \operatorname{dist}(f^{n}(z),\partial U_{n})>0,
%    \]
%    that is, all orbits stay away from the boundary;
%\item[\rm(b)] there exists a subsequence $n_k\to \infty$ for which $\operatorname{dist}(f^{n_k}(z),\partial U_{n_k})\to 0$ for all $z\in U$, while for a different subsequence $m_k\to\infty$ we have that
%    \[\liminf_{k \to \infty} \operatorname{dist}(f^{m_k}(z),\partial U_{m_k})>0, \quad\text{for }z\in U;\]
%\item[\rm(c)] $\operatorname{dist}(f^{n}(z),\partial U_{n})\to 0$ for all $z\in U$, that is, all orbits converge to the boundary.
%\end{itemize}
\end{ex}

\begin{proof}

(a) In this case we let $b_n(z)=z$, for $n \in \N$, and apply Theorem~\ref{thm:main construction} with $z_1 = 0$. For $n \in \N$, the map $b_n$ is univalent and so, by Theorem~\ref{thm:main construction} part (iv), $f: U_{n-1}\to U_n$ is also univalent. Thus $U_0$ is eventually isometric.

Since $B_n(0) = 0$, for $n \in \N$, the result now follows from case (a) of Lemma~\ref{bnboundary}.

(b) In this case, for $n \in \N$, we let

$$b_n(z)=
\begin{cases}
\mu_k(z), & \text{if $n=2k-1$, $k\geq 1$},\\
\mu_k^{-1}(z), & \text{if $n=2k$, $k\geq 1$},
\end{cases}
$$
where $\mu_k$ is as defined in~\eqref{mun}. We apply Theorem~\ref{thm:main construction} with $z_1 = 0$. For $n \in \N$, the map $b_n$ is univalent and so, as in case (a), $U_0$ is eventually isometric.

Since $B_{2k}(0) = 0$, for $k \in \N$, and $B_{2k-1}(0)=a_{2k-1} \to 1$ as $k \to \infty$, the conclusion now follows from case (b) of Lemma~\ref{bnboundary}.

(c) In this case, for $n \in \N$, we let $b_n(z) = \frac{z + 1/2}{1 + z/2}$. As in cases (a) and (b), we apply Theorem~\ref{thm:main construction} with $z_1 = 0$ and, since the map $b_n$ is univalent, for $n \in \N$, we deduce that $U_0$ is eventually isometric. Since $b_n(x) > x$, for $x \in [0,1)$, we deduce that $B_n(0) \to 1$ as $n \to \infty$. The conclusion now follows from case~(c) of Lemma~\ref{bnboundary}.
 \end{proof}

%%%%%%%

\subsection{$2-$super--attracting wandering domains}
We prove part (b) of Theorem~E by  giving an example of a transcendental entire function with a sequence of wandering domains $(U_n)$ containing two orbits consisting of critical points. The case of finitely many critical orbits is completely analogous.

We choose the sequence of Blaschke products $(b_n)$. We let
$$b_1(z)= z^2 \frac{z+a_1}{1+a_1z},$$
with $a_1<0$ chosen so that $1/2$ is a critical point, and, for $n \in \N$ define $b_{n+1}$ inductively by setting
$$b_{n+1}(z)= z^2 \frac{z+a_{n+1}}{1+a_{n+1}z},$$
with  $a_{n+1}\in (-1,1)$  chosen so that $b_n\circ b_{n-1}\circ\cdots\circ b_1(1/2)$ is a critical point of $b_{n+1}$.  In this way we construct a sequence of Blaschke products $(b_n)$ of degree 3 such that, for $n \in \N$, we have  $b_n(0)=0$ and the two critical points of $b_n$ are $0, B_{n-1}(1/2)$.

We now apply Theorem \ref{thm:main construction} with $z_1 = 0$ and $z_2 = 1/2$. We deduce that there exists a transcendental entire function $f$ which has a sequence of bounded, simply connected, escaping, wandering domains $(U_n)$ such that, for $n \geq 0$, $f^n(0) = F_n(0) = 4n$, $f^n(1/2)= F_n(1/2)$ and $f'(f^n(0)) = f'(f^n(1/2))=0$. Hence, there are two points in $U_0$, namely 0 and 1/2, whose orbits under $f$ consist of critical points of $f$.

\begin{rem*}
It follows from Theorem \ref{thm:main construction} that $\operatorname{dist}_{U_n}(f^n(0),f^n(1/2)) \to 0$ as $n \to \infty$ and, in fact, one can check that these wandering domains are super-contracting.
\end{rem*}

%%%%%%%%%%%%%%%%%%%%%%%%%%%%%%%%%%%%%%%%%%%%%%%%%%%
%%%%%%%%%%%%%%%%%%%%%%%%%%%%%%%%%%%%%%%%%%%%%%%%%%%

\section{Wandering domains whose boundaries are Jordan curves}\label{sec:Jordan curves}

In this section we prove that, if the Blaschke products in Theorem~\ref{thm:main construction} satisfy certain conditions, 
then the boundaries of the resulting wandering domains are Jordan curves. Apart from wandering domains arising from lifting constructions, as far as we are aware these are the first examples of simply connected wandering domains for which it is possible to obtain information  concerning  the boundary.  Examples of {\it multiply} connected wandering domains  for which it is known that connected components of the  boundary are Jordan curves can be found in \cite{Bish1} and in \cite{Baumgartner}.

In order for the boundaries of the resulting wandering domains to be Jordan curves, we need the Blaschke products in Theorem~\ref{thm:main construction} to be  uniformly expanding in the following precise sense.

\begin{defn}[Uniformly expanding Blaschke products]\label{defn:uniformly expanding} Let $(b_n)$ be a sequence of Blaschke products. We say that the Blaschke products in the sequence  $(b_n)$ are \emph{uniformly expanding } if there exists  $\xi>1$ and  an $\epsilon$-neighborhood $U_\epsilon$ of the unit circle such that
\begin{enumerate}
\item each $b_n$ is holomorphic  in $U_\epsilon$, that is, $b_n$ has no poles in $U_\epsilon$;
\item\label{expansivity} $|b_n'|\geq\xi$ on $U_\epsilon$.
\end{enumerate}
\end{defn}
Note that the second  condition  implies that the $b_n$ have no critical points in $U_\epsilon$.\\

Our aim for this section is to prove the following theorem.
\begin{thm}\label{thm:boundaries are Jordan curves} Let $(b_n)$ be a sequence of uniformly expanding Blaschke products such that $\max_n\{{\rm deg}\, b_n\} < \infty$ and let $(U_n)$ be the resulting orbit of wandering domains given by  Theorem~\ref{thm:main construction}. Then, for $n \geq 0$, the boundary of the wandering domain $U_n$ is a Jordan curve.
\end{thm}

The proof of Theorem~\ref{thm:boundaries are Jordan curves} follows in outline the proof that the Julia sets of certain quadratic polynomials are Jordan curves (see \cite[Section 9.9]{Beardon}, for example) but with significant adjustments and additional arguments due to the fact that we are dealing with a lack of uniformity arising from the associated non-autonomous system of maps.

The proof has three steps:
\begin{itemize}
\item[1.]
For each $n \in \N$, we let $A_n$ be the annulus bounded by the circles $\gamma_n$ and $\Gamma_n$, which were defined in~\eqref{gammaGamma} and played a key role in the proof of Theorem~\ref{thm:main construction}. We consider the annulus $\widehat A_n$ lying in $A_n$ between $\Gamma_n$ and a component of $f^{-1}(A_{n+1})$ and show that the vertical geodesics of $\widehat A_n$ have uniformly bounded Euclidean length.

\item[2.] We then use pullbacks under $f^{-n}$ of these vertical geodesics together with the uniformly expanding property of the functions $b_n$ (and hence of~$f$ on the annuli $A_n$) to induce a continuous map $\Sigma$ from $\Gamma_0$ to a closed curve $\Sigma(\Gamma_0)$, and a continuous map $\sigma$ from $\gamma_0$ to a closed curve $\sigma(\gamma_0)$.

\item[3.] Finally, we show that $\partial U_0$, which is squeezed between $\Sigma(\Gamma_0)$ and  $\sigma(\gamma_0)$, is a Jordan curve.
\end{itemize}

The first step in the proof relies on a general geometric result of independent interest about the Euclidean lengths of vertical geodesics of annuli, which we prove using the Fej\'er-Riesz inequality. We prove this and other preliminary geometric results in Section 7.1, and then give the proof of Theorem~\ref{thm:boundaries are Jordan curves} in Section~7.2.

\subsection{Preliminary results}
We begin with a result about pre-images of annuli bounded by Jordan curves. Related results appear in \cite[Lemma~11.1]{Bish1} and \cite[Lemma~5]{Slow}.

In this lemma, we denote the inner boundary component of a topological annulus~$A$ by $\partial A_{{\rm inn}}$ and the outer boundary component by $\partial A_{{\rm out}}$.   As usual, when we say that $f$ is holomorphic on $\overline{A}$ we mean that $f$ is holomorphic on a neighborhood of $\overline{A}$.

\begin{lem}\label{lem:preimages of Jordan} Let~$A$ and~$B$ be annuli with Jordan curve boundary components, both surrounding~$0$, and let~$f$ be holomorphic on $\overline A$, with $f$ and $f'$ non-zero on~$\overline A$. Suppose that
\begin{equation}\label{inn-out}
\partial B_{{\rm inn}}\;\text{ surrounds }\; f(\partial A_{{\rm inn}})\quad\text{and}\quad f(\partial A_{{\rm out}})\; \text{ surrounds }\;  \partial B_{{\rm out}}.
\end{equation}
Then~$A$ contains a unique component~$\widehat A$ of $f^{-1}(B)$, which is an annulus that surrounds~$0$ with Jordan curve boundary components that satisfy
\begin{equation}\label{Jordan}
f(\partial \widehat A_{{\rm inn}})= \partial B_{{\rm inn}}\quad\text{and} \quad f(\partial \widehat A_{{\rm out}})= \partial B_{{\rm out}}.
\end{equation}
\end{lem}

 \begin{proof}
Since $f(\partial A) \cap B= \emptyset$ we deduce that every connected component $H$ of $f^{-1}(B)$ which intersects $A$ is in fact contained in $A$. Since  $f(\partial A)$ intersects both components of $B^c$, we deduce that $f(A) \supset B$ and that there exists at least one component~$\hat A$ of $f^{-1}(B)$ that  is contained  in~$A$. We now claim that $\hat A $ is doubly connected, surrounds 0, and that there are no other preimage components of $B$ in $A$.

Let $H$ be any preimage component of $B$ which is contained in  $A$. By the Riemann--Hurwitz formula $H$  is at least doubly connected. Let $X$ be a bounded complementary component of $H$. Since $f:H\ra B$ is proper, $f(X)$ is the bounded complementary component of $B$, hence $0\in f(X)$. If $X$ does not contain $0$, then $X\subset A$, on which $f$ cannot take the value $0$ by assumption, giving a contradiction.
Hence every preimage component $H$ of $B$ in $A$ is doubly connected and surrounds $0$.
Now suppose that there are two such components. Then one of them is contained in the bounded complementary component of the other, and hence maps to the bounded complementary component of $B$, again a contradiction since it is a preimage component of~$B$.

Since the boundary components of $B$ are Jordan curves, the fact that $f:\hat A\ra B$ is proper, together with the fact that $f'\ne 0$ on $\overline A$, implies that the boundaries of $\hat A$ are also Jordan curves, and also implies~\eqref{Jordan}.
\end{proof}

The main result in this subsection concerns the notion of a {\it vertical foliation} of an annulus which we now define.

\begin{defn}[Vertical foliations] Let $A$ be an open annulus and consider  the straight annulus $\A_{\rho}=\{z:\rho<|z|<1\}$ such that $\phi:\A_{\rho}\ra A$ is a biholomorphism. The {\it vertical foliation} $\FF_A$ of $A$ consists of the image curves under~$\phi$ of the radial segments connecting the two circles which form  the boundary of $\A_\rho$. Each of these image curves is a hyperbolic geodesic which we refer to as a {\it vertical geodesic}.
\end{defn}

We can now state the main result of this subsection, which concerns the Euclidean lengths of geodesics in vertical foliations of annuli. The proof uses the Fej\'er-Riesz Inequality (stated below); similar reasoning using instead the Gehring-Hayman Theorem (see \cite{GH} or \cite[Section~4]{Pommerenke}) is possible.

\begin{thm}\label{lem:geometry of annuli}
Let $A$ be an annulus for which both  boundary components are analytic Jordan curves with length at most~$S$, and such that the bounded component of $\C\setminus A$ contains a disk of radius $r>0$. Then there exists $M=M(S,r)>0$ such that
$$\leucl(\gamma)\leq M, \text{\ \ \  for all $\gamma\in\FF_A$}.$$
\end{thm}

We prove Theorem~\ref{lem:geometry of annuli} using the following technical lemma.

\begin{lem}\label{lem:most geodesics have bounded length}
Let $A$ be  an annulus for which both  boundary components are analytic Jordan curves with length at most~$S$, and consider the straight annulus $\A_{\rho}=\{z:\rho<|z|<1\}$ such that $\phi:\A_{\rho}\ra A$ is a biholomorphism.
%Let  $\FF_A$ be  the vertical foliation on $A$.
For $\theta\in [0,2\pi]$, let $\sigma_\theta:=\phi(\{r e^{i\theta}: \rho<r<1\})$ and $\ell(\theta):=\leucl(\sigma_\theta)$.

Then the Lebesgue measure of the set of $\theta$ such that $\ell(\theta)<\frac{2S}{\rho}(1-\rho)$ is at least  $2\pi-\frac{1}{2}$. \end{lem}
\begin{proof}
Consider the integral
\[
I=\int_{\rho}^1\int_{0}^{2\pi}|\phi'(re^{i\theta})|\,drd\theta=\int_{\rho}^1 dr \int_{0}^{2\pi}|\phi'(re^{i\theta})|\,d\theta.
\]
The function
$$
I(r):=\int_{0}^{2\pi}|\phi'(re^{i\theta})|\,d\theta,\quad \rho<r<1,
$$
is a convex function of $\log r$ since $|\phi'|$ can be extended to be subharmonic in a neighbourhood of $\overline{\A_{\rho}}$, so
$I(r)\leq\max\{I(\rho), I(1)\}$, for $\rho \le r\le 1$. Then
\begin{align*}
\rho I(\rho)&= \int_{0}^{2\pi}\rho |\phi'(\rho e^{i\theta})|\,d\theta=\leucl(\partial A_{{\rm in}})\leq S,  \\
        I(1)&= \int_{0}^{2\pi}|\phi'(e^{i\theta})|\,d\theta=\leucl(\partial A_{{\rm out}})\leq S.
\end{align*}
Hence $I(r)\leq S/\rho$, for $\rho \le r\le 1$, so
  $$
  I=\int_\rho^1 I(r)\,dr\leq \frac{S}{\rho}(1-\rho).
  $$
Changing the order of integration we obtain
$$
I=\int_{0}^{2\pi} \left(\int_{\rho}^1  |\phi'(re^{i\theta}| \,dr\right)\,d\theta=\int_{0}^{2\pi} \ell(\theta)\,d\theta\leq \frac{S}{\rho}(1-\rho).
$$
Hence the Lebesgue measure of the set $\{\theta: \ell(\theta)> \frac{2S}{\rho}(1-\rho)\}$ is at most  $1/2$.
\end{proof}

In particular, Lemma~\ref{lem:most geodesics have bounded length} shows that the annulus~$A$ has many vertical geodesics whose Euclidean length is at most $2S(1-\rho)/\rho$.

Next, we state the following classical result (see\cite{FR21} and \cite[Theorem 3.13]{Duren}) about the space $H^p$, $p>0$, of functions $g$ holomorphic in $\D$ such that
\[
\sup_{0\le r< 1}\left\{\int_0^{2\pi}|g(re^{i\theta})|^p\,d\theta\right\} <\infty.
\]
\begin{lem}[Fej\'er-Riesz Inequality]\label{thm:FR}
If $g\in H^p$, then
$$  \int_{-1}^1|g(x)|^p\, dx \leq\frac{1}{2}\int_0^{2\pi}|g(e^{i\theta})|^p\,d\theta.$$
\end{lem}

We can now give a proof of Theorem~\ref{lem:geometry of annuli}.

\begin{proof}[Proof of Theorem~\ref{lem:geometry of annuli}] Since the result remains true under a translation,  we can assume that $\C\setminus A$ contains a disk of radius $r$ centered at $0$.

We first claim that there exists $L=L(S,r)$ and a vertical geodesic $\sigma\in\FF_A$ such that $\leucl(\sigma)\leq L$.   Since $\C\setminus A$ contains a disk of radius $r$ centered at $0$, and the outer boundary has length at most $S$, the modulus of~$A$ is bounded from above by a constant depending only on $S$ and $r$.  The claim then follows by Lemma~\ref{lem:most geodesics have bounded length}, since~$\rho$ is bounded from below by a positive constant depending only on $S$ and $r$.

Now let $\log A$ be a lift of $A\setminus \sigma$ under the exponential map, using a suitable branch of the logarithm.  Observe that $\log A$ is simply connected and that vertical geodesics in~$A$ lift to geodesic cross cuts in $\log A$. For any vertical geodesic $\gamma \in \FF_A$ consider its lift $\log \gamma$ in $\log A$.
Let $\psi:\D\ra\log A$ be a biholomorphism  such that $\psi(\{z: -1 < \Re z < 1\})=\log\gamma$. This can be done by mapping~$0$ to a  point in $\log \gamma$, observing that geodesics are mapped to geodesics, and pre-composing with a rotation if necessary. By applying the Fej\'er-Riesz Inequality (Lemma~\ref{thm:FR}) with $p=1$ and $g=\psi'$, we obtain
\begin{equation}\label{eqtn:consequence of FR}
\leucl(\log\gamma)\leq\frac{1}{2}\leucl(\partial\log A).
\end{equation}

So it remains to show that $\leucl (\partial\log A)$ is bounded by a uniform constant and that the resulting bound on $\leucl(\log\gamma)$ can be translated into a bound for $\leucl(\gamma)$. We do this by studying the distortion of lengths of curves under the lift via the exponential.
Let $t\mapsto z(t)$ for $t\in [0,1]$ be a parametrization of a curve $C$ in $\ov{A}$ and let  $\log C$ be its lift in $\ov{\log A}$. Then $t\mapsto\log(z(t))$ for $t\in [0,1]$ is a parametrization of the  curve $\log C$, so
$$
\leucl(C)=\int_{0}^{1} |z'(t)|\,dt \; \text{ and  }\;  \leucl(\log C)=\int_{0}^{1} \left|\frac{z'(t)}{z(t)}\right|dt.
$$
Since ${A}$ is contained in the straight annulus $\A(r,S/2)$  (this follows from considering the extremal case for $\partial A_{out}$),  we have that $r\leq|z(t)|\leq S/2$, so
\begin{equation}\label{eqtn:distortion of lengths}
\frac{2}{S}\leucl(C)\leq\leucl(\log C)\leq\frac{1}{r}\leucl(C).
\end{equation}
It follows that $\leucl (\log\sigma)\leq \frac{1}{r}\leucl(\sigma)\leq \frac{L}{r}$ and that if $\alpha=\partial A_{{\rm inn}}$ and $\beta=\partial A_{{\rm out}}$ are the inner and outer boundary components, respectively, of~$A$, then we have
\[
\leucl (\log \alpha)+\leucl (\log \beta)\leq \frac{1}{r}(\leucl (\alpha)+\leucl (\beta))\leq \frac{2S}{r}.
\]
So
\[
\leucl(\partial\log A) = 2\leucl (\log\sigma) + \leucl (\log \alpha)+\leucl (\log \beta) \leq \frac{2L}{r} + \frac{2S}{r}.
\]
%This gives the desired bound for $\leucl (\partial\At)$  independent of the geometry of $A$.
It now follows from (\ref{eqtn:consequence of FR}) and \eqref{eqtn:distortion of lengths} that, for any vertical geodesic $\gamma \in \FF_A$, we have

$$\leucl(\gamma)\leq \frac{S}{2} \leucl(\log\gamma)\leq\frac{S(L+S)}{2r}. $$
This concludes the proof of Theorem~\ref{lem:geometry of annuli}.
\end{proof}

\subsection{Proof of Theorem~\ref{thm:boundaries are Jordan curves}}
Let $(b_n)$ be a sequence of uniformly expanding Blaschke products of degree at most~$d$ and let~$f$ be the transcendental entire function with an associated orbit of wandering domains $(U_n)$ arising from Theorem~\ref{thm:main construction}. We will show that the boundary of $U_0$ is a Jordan curve.

For each $n \in \N$, we let $A_n$ be the annulus bounded by the circles $\gamma_n$ and $\Gamma_n$ which were defined in~\eqref{gammaGamma} in the proof of Theorem~\ref{thm:main construction}.
%Let $f$ be the function and $U_n$ be the sequence of wandering domains given by Theorem~\ref{thm:main construction} for the sequence of Blaschke products $(b_n)$. Since the $U_n$ are bounded, $f:U_n\ra U_{n+1}$ is proper for every $n$.
%
%We will show that the boundary of $U_0$ is a Jordan curve.     Since $f^n:\partial U_0\ra\partial U_n$ is proper this implies that $\partial U_n$ is  a Jordan curve for every $n$, and the same is true for all preimages of $U_0$ (possibly, they will be unbounded Jordan curves, or there will be several complementary components of $f^{-n}(\partial U_0)$) which are preimages of $U_0$).
By the uniform expansivity condition on the functions~$b_n$ and the fact  that $\max\{|f(z)-f_n(z)|: z\in A_n\}\ra 0$ as $n\ra\infty$ (see \eqref{approx1}), we deduce using Cauchy's estimate that there exists $\eta > 1$ such that, for sufficiently large $n \in \N$,
\begin{equation}\label{eqtn:expansivity f}
|f'|\geq\eta>1 \text{ on  a neighborhood of $A_n$};
\end{equation}
in particular,~$f$ has no critical points in a neighborhood of $A_n$ for such~$n$. Relabeling the $U_n$ if necessary, we can assume that the above conditions hold for any $n\geq0$.\\

{\bf Step 1}\;
For each $n \geq 0$, we let $\widehat A_n$ denote the pre-image component of $A_{n+1}$ under~$f$ in $A_n$, given by Lemma~\ref{lem:preimages of Jordan}, with inner and outer boundary components $\widehat \gamma_n$ and $\widehat \Gamma_n$, say, respectively. Then let $\breve A_n$ denote the annulus lying between $\widehat A_n$ and $\Gamma_n$ (see Figure \ref{fig:Jordan curves}). Our first claim is that there exists $M=M(\eta,d)>0$  such that, for $n\ge 0$, each geodesic in the vertical foliation $\FF_n$ of the annulus $\breve A_n$ has Euclidean length at most $M$.

We start by showing that each Jordan curve $\widehat \Gamma_n$ has length which is uniformly bounded by $3\pi d/\eta$. Indeed, we can parametrize $\widehat \Gamma_n:[t_0,t_1]\cup [t_1,t_2]\cup\cdots \cup [t_{d-1},t_d]\ra \C$, where $t_0<t_1<\cdots<t_d$, with $\widehat \Gamma_n(t_0)=\widehat \Gamma_n(t_d)$, in such a way that~$f$ is univalent on $\widehat \Gamma_n(t_i, t_{i+1})$, $0\le i<d-1$. This can be done because the degree of $b_n$ (and hence of~$f$ on $\breve A_n$, for $n$ sufficiently large) is bounded above by~$d$, and $\widehat \Gamma_n$ is a Jordan curve. Notice that $f(\widehat \Gamma_n[t_i,t_{i+1}])\subseteq \Gamma_{n+1}$. For~$n$ large, we have, by \eqref{eqtn:expansivity f},
\begin{align*}
3\pi\geq \leucl (\Gamma_{n+1})&\geq \frac{1}{d}\left(\int_{t_0}^{t_1} |f'(\widehat \Gamma_n(t))||\widehat \Gamma_n'(t)|\,dt+\cdots+ \int_{t_{d-1}}^{t_d}|f'(\widehat \Gamma_n(t))||\widehat \Gamma_n'(t)|\,dt\right) \\ &\geq \frac{\eta}{d} \int_{t_0}^{t_d}|\widehat \Gamma_n'(t)|\,dt=\frac{\eta}{d} \leucl (\widehat \Gamma_n).
\end{align*}
(The second inequality becomes an equality if $f:\widehat \Gamma_n[t_i,t_{i+1}]\ra \Gamma_{n+1}$ is surjective for every~$i$.)
Since, by construction, the bounded component of $\C\setminus \widehat A_n$ contains the circle $\gamma_n=\{z:|z|=r_n\}$ and $r_n \ge 1/2$ for $n\ge 0$, the annulus $\widehat A_n$ satisfies the hypotheses of Theorem~\ref{lem:geometry of annuli}, giving the claim.

\begin{figure}[hbt!]
\centering
\includegraphics[width=0.9\textwidth]{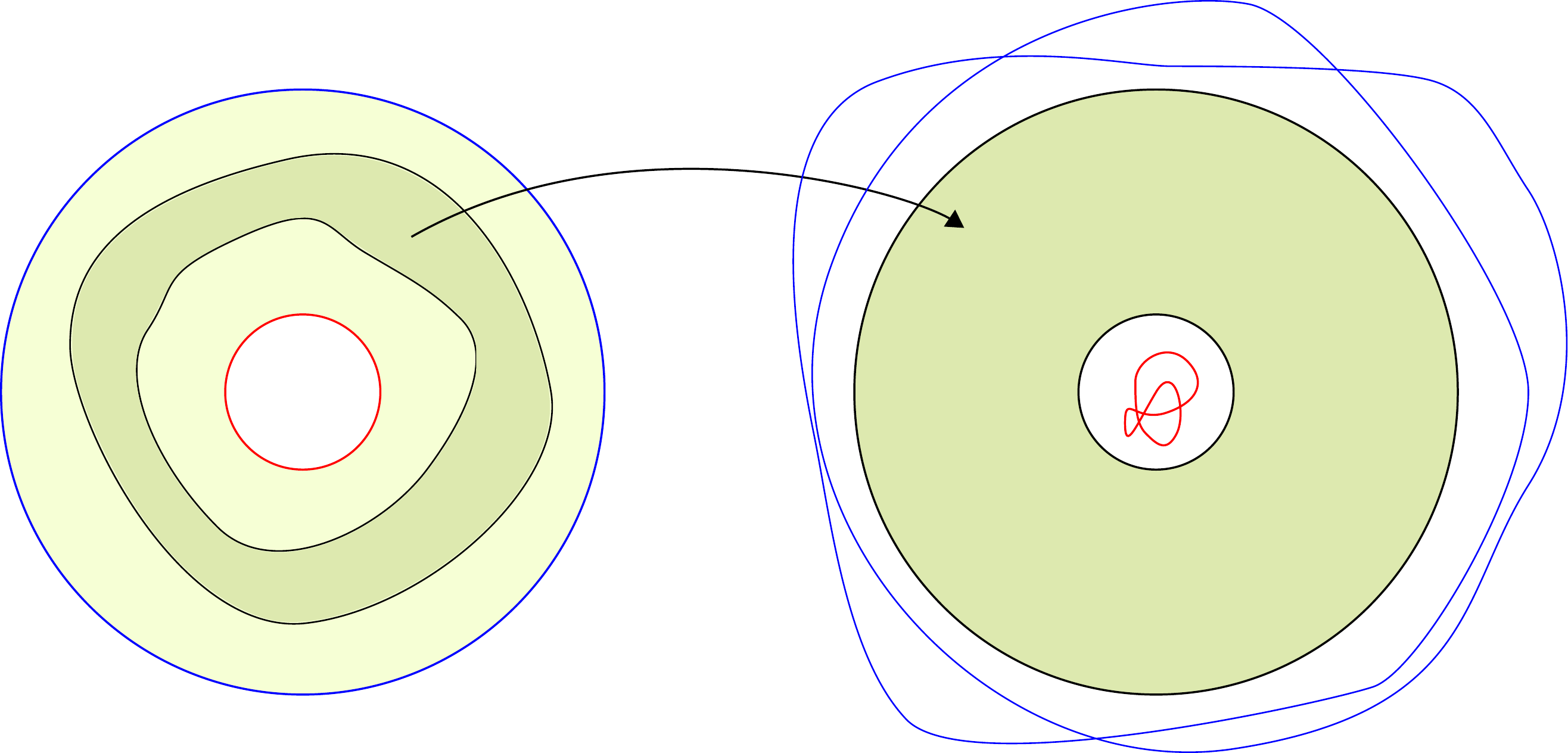}
\setlength{\unitlength}{0.9\textwidth}
\put(-0.79,0.22){\small $\gamma_n$}
\put(-0.6,0.22){\small $\Gamma_n$}
\put(-0.92,0.11){\small $\widehat{\Gamma}_n$}
\put(-0.87,0.16){\small $\widehat{\gamma}_n$}
\put(-0.72,0.16){$\widehat{A}_n$}
\put(-0.82,0.13){$A_n$}
%\put(-0.82,0.15){$A_n$}
\put(-0.3,0.13){\large ${A}_{n+1}$}
\put(-0.21,0.22){\small $\gamma_{n+1}$}
\put(-0.07,0.22){\small $\Gamma_{n+1}$}
%\put(-0.73,0.17){$\widehat{A}_n$}
\put(-0.75,0.07){$\breve{A}_n$}
\put(-0.16,0.46){\small $f(\Gamma_{n})$}
\put(-0.34,0.22){\small $f(\gamma_{n})$}
\put(-0.56,0.39){\small $f$}
\caption{\label{fig:Jordan curves} Sketch of the setup in Step 1. }
\end{figure}
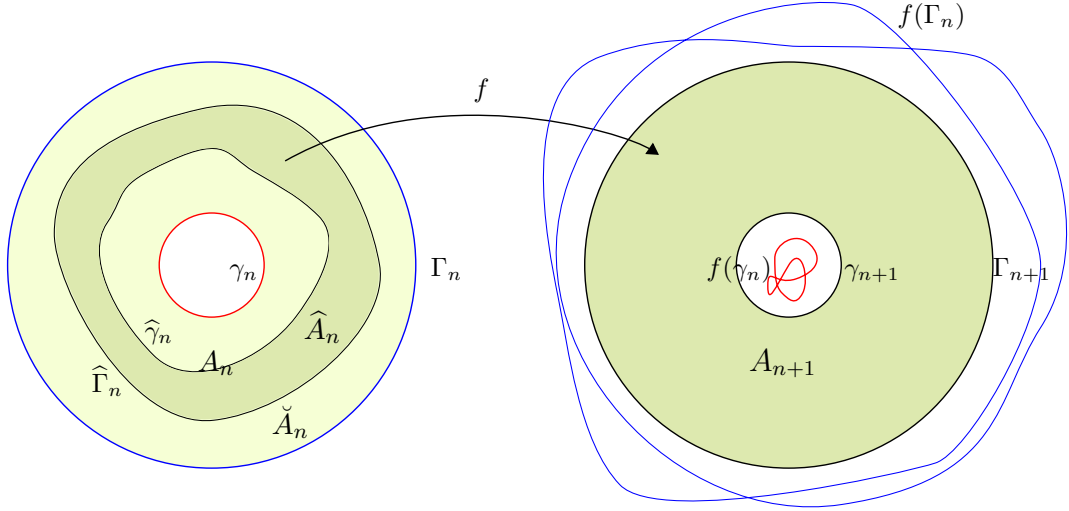

{\bf Step 2}\; Now let $\tilde{\Gamma}_n$, $\tilde{\gamma}_n$ be those pre-images under $f^n$ of $\Gamma_n,\gamma_n$, respectively, that are contained in  $A_0$ and are such that $\tilde{\Gamma}_n$ surrounds $\Gt_{n+1}$ for every $n$, while  $\tilde{\gamma}_n$  is surrounded by $\tilde{\gamma}_{n+1}$ for every~$n$. The existence of $\tilde{\gamma}_n$ and $\tilde{\Gamma}_n$, and the fact that they are Jordan curves, follows by applying inductively Lemma~\ref{lem:preimages of Jordan}, since there are no critical points in $A_n$ for any $n$ by the uniform expansivity condition.

We now concentrate on the family of Jordan curves $\Gt_n$ and construct a continuous map $\Sigma$ from $\Gamma_0$ to the limit of the $\Gt_n$, defined in an appropriate way.
% They can be defined by taking preimages inductively along the backward orbits of the wandering domains.

Fix $z_0\in\Gamma_0$. Let   $\Sigma_0(z_0)$ be the (unique) geodesic in $\FF_0$ which connects $z_0$ to some point  $z_1\in\Gt_1$, and let us parametrize it as a curve $\Sigma_0(z_0,t)$, $t\in[0,1]$, with $\Sigma_0(z_0,0)=z_0$,  $\Sigma_0(z_0,1)=z_1$. Consider $f(z_1)\in\Gamma_1$, and the geodesic $\omega_1\in\FF_1$ which connects $f(z_1)$  to some point $z_2'$ say in $f^{-1}(\Gamma_2)$. Notice that the definition  of $\FF_1$ automatically specifies the connected component of $f^{-1}(\Gamma_2)$ to which $z'_2$ belongs. The preimage of $\omega_1$ under $f$ which contains $z_1$ is an arc that can be parametrized as $\Sigma_1 (z_0,t)$, $t\in[1,2]$, connecting $z_1$ to some point $z_2\in f^{-1}(z'_2)\cap \Gt_2$.
Proceeding in this way, for each~$n$ we can construct a point $z_n$ and a curve $\Sigma_n (z_0,t)$, $t\in[n,n+1]$, connecting $z_n$ to $z_{n+1}$. This can be repeated for any starting point $z\in\Gamma_0$ to construct a continuous injective curve
$$
\Sigma(z,t):\Gamma_0\times [0,\infty]\ra A_0,
$$
such that $\Sigma(z,n)\in\Gt_n$ and, for each $z\in \Gamma_0$ and each $j\leq n$, $n \in \N$, we have $f^j(\Sigma(z,[n,n+1]))\subset A_j$ and $f^n(\Sigma(z,[n,n+1]))$ is a geodesic in $\FF_n$.

Recalling that the Euclidean length of elements in $\FF_n$ is bounded uniformly in $n$ by a constant $M=M(\eta,d)>0$, and using the expansivity estimate (\ref{eqtn:expansivity f}) on $f$, we obtain
$$
\leucl( \Sigma(z,[n,n+1]))\leq \frac{1}{\eta^n}M,\;\text{ for } z\in \Gamma_0, n\in \N.
$$
%{\violet Do we need injectivity for the last estimate?}
It follows that for each~$z\in \Gamma_0$ the curve $\Sigma(z,t)$ converges to $\Sigma(z)$, say, as $t\ra\infty$, and moreover that the map $\Sigma: \Gamma_0\ra \Sigma(\Gamma_0)$ is continuous in $z$, so $\Sigma(\Gamma_0)$ is a closed  curve. Note, however, that we have not shown that $\Sigma$ is a Jordan curve, since the map $z\mapsto \Sigma(z)$ has not been shown to be injective.

We can construct an analogous map $\sigma: \gamma_0\ra \sigma(\gamma_0)$ and obtain a closed curve $\sigma(\gamma_0)$ as a uniform limit using the Jordan curves $\tilde \gamma_n$ in a similar manner.\\

{\bf Step 3}\; We now do the final step of showing that $\partial U_0$ is indeed a Jordan curve. It is sufficient (see \cite[Chapter VI, Theorem 16.1]{New61}) to show that each point of $\partial U_0$ is accessible from both complementary components. By the construction in Step~2, it is enough to show that
\[
\partial U_0 \subset \sigma(\gamma_0) \cap \Sigma(\Gamma_0).
\]

%We claim that $\Sigma(\Gamma_0)=\sigma(\gamma_0)=\partial U_0$.

Let $C_n$, $c_n$ denote the bounded complementary components of the Jordan curves $\tilde{\Gamma_n}$, $\tilde{\gamma_n}$, respectively.
Then $(C_n)$ and $(c_n)$ each form a sequence of nested topological disks which are respectively  decreasing and  increasing, because for each~$n$ we have that $f(\Gamma_n)$ surrounds $\Gamma_{n+1}$ and $f(\gamma_n)$ is surrounded by $\gamma_{n+1}$. Notice that for each~$n$, the annulus $\tilde A_n:=C_n\setminus \ov{c_n}$ contains $\partial U_0$. This is because, for $n \ge 0$, we have $\Gamma_n$ surrounds $U_n$ and $\gamma_n\subset U_n$.

Now suppose that $\partial U_0 \not\subset \Sigma(\Gamma_0)$, and let $\zeta\in \partial U_0 \setminus \Sigma(\Gamma_0)$. Then for some $r>0$ the disk $D(\zeta,r)$ does not meet the closed curve $\Sigma(\Gamma_0)$, so there is some open disk $D(\zeta',r')\subset D(\zeta,r)$ that lies in both the exterior of $\overline{U_0}$ and in the bounded complementary component of $\Sigma(\Gamma_0)$ which contains $U_0$. Hence $D(\zeta',r')\subset \tilde A_n$, for all $n\in \N$, and by construction $f^n(D(\zeta',r'))\subset A_n$ for every~$n$. Now, the maximal radii of the Euclidean disks contained in the annuli $A_n$ are bounded, and indeed converge to~$0$ as $n\to\infty$. For large~$n$, this contradicts the fact that, since $|({f^n})'(\zeta')|\geq \eta^n$, the image $f^n (D(\zeta',r'))$ contains a disk of radius at least $Br' \eta^n$, where $B>0$ is Bloch's constant. Hence $\partial U_0 \subset \Sigma(\Gamma_0)$.

A similar argument shows that $\partial U_0\subset \sigma(\gamma_0)$, which completes the proof.

\bibliographystyle{amsalpha}
\bibliography{Wandering}

{\em Emails}: ambenini@gmail.com, 
 vasiliki.evdoridou@open.ac.uk,  fagella.nuria@gmail.com,\\ phil.rippon@open.ac.uk, gwyneth.stallard@open.ac.uk.

\end{document}